\documentclass[12pt,a4paper]{article}
\usepackage{amssymb,amsmath,amsthm}
\usepackage{latexsym}
\usepackage{color}

\newtheorem{theorem}{Theorem}
\newtheorem{lemma}{Lemma}
\newtheorem{proposition}{Proposition}
\newtheorem{corollary}{Corollary}
\newtheorem{remark}{Remark}
\numberwithin{equation}{section}
\numberwithin{theorem}{section}
\numberwithin{lemma}{section}
\numberwithin{proposition}{section}
\numberwithin{corollary}{section}
\numberwithin{remark}{section}

\begin{document}
\title{Regularity properties of a generalized Oseen evolution operator 
in exterior domains, with applications to the Navier-Stokes initial value problem}
\author{Yosuke Asami and Toshiaki Hishida \\
Graduate School of Mathematics \\
Nagoya University \\
Nagoya 464-8602, Japan \\
\texttt{yosuke.asami.e6@math.nagoya-u.ac.jp} \\
\texttt{hishida@math.nagoya-u.ac.jp} \\
}
\date{}
\maketitle
\begin{abstract}
Consider a generalized Oseen evolution operator in 3D exterior domains, that is generated by a non-autonomous linearized system
arising from time-dependent rigid motions.
This was found by Hansel and Rhandi, and then the theory was developed by the second author, however,
desired regularity properties such as estimate of the temporal derivative as well as the H\"older estimate have remained open.
The present paper provides us with those properties together with weighted estimates of the evolution operator.
The results are then applied to the Navier-Stokes initial value problem, so that a new theorem on existence of a unique strong 
$L^q$-solution locally in time is proved.


\end{abstract}

\section{Introduction}
\label{intro}

Let us consider the motion of a viscous incompressible fluid filling the exterior of an obstacle in $\mathbb R^3$,
where the obstacle is a rigid body whose motion is prescribed and can be time-dependent.
We prefer to take a frame attached to the moving obstacle in order that the problem is reduced to the one in a fixed exterior domain $D$,
whose boundary $\partial D$ is assumed to be of class $C^{1,1}$.
The resultant set of equations reads, see \cite{Ga02},
\begin{equation}
\begin{split}
&\left.
\begin{array}{rl}
\partial_tu+u\cdot\nabla u&=\Delta u+(\eta+\omega\times x)\cdot\nabla u-\omega\times u-\nabla p \\
\mbox{div $u$}&=0 \\
\end{array}
\right\} \;\;\mbox{in $D\times (0,\infty)$},  \\
&\quad u|_{\partial D}=\eta+\omega\times x, \qquad
\lim_{|x|\to\infty}u=0, \qquad
u(\cdot,0)=u_0,
\end{split}
\label{ns}
\end{equation}
where $u=u(x,t)\in\mathbb R^3$ and $p=p(x,t)\in \mathbb R$ denote the velocity and pressure
of the fluid, whereas $\eta=\eta(t)\in \mathbb R^3$ and $\omega=\omega(t)\in\mathbb R^3$ are, respectively, translational and angular
velocities of the body, and $u_0$ stands for a given initial velocity.

The Navier-Stokes initial value problem \eqref{ns} has attracted much attention of several authors 
since the work \cite{Bor92} on the Leray-Hopf weak solutions, because the drift operator 
$(\omega\times x)\cdot\nabla$ arising from rotation of the body causes a hyperbolic aspect and makes the problem challenging.
In fact, the semigroup generated by the linearized operator with a constant rigid motion 
is no longer analytic (\cite{FN07, Hi99}) unless the rotation is absent, nevertheless it enjoys remarkable smoothing effect
(\cite{GHH06, Hi99}), see 
\cite{Hi13} for the details and further references.
In this paragraph, let us focus on the autonomous system with purely rotating body
(constant $\omega\in\mathbb R^3\setminus \{0\}$ and $\eta=0$). 
A unique mild solution 
was constructed locally in time first by the second author \cite{Hi99} in $L^2$ and then by Geissert, Heck and Hieber \cite{GHH06} in $L^q$
via the semigroup approach, while Galdi and Silvestre \cite{GS05} adopted the other method using the Galerkin approximation
to show the existence of a unique local solution in $L^2$ 
when the initial velocity satisfies
\begin{equation}
u_0\in W^{2,2}(D), \quad
(\omega\times x)\cdot\nabla u_0\in L^2(D), \quad 
u_0|_{\partial D}=\omega\times x, \quad 
\mbox{div $u_0$}=0.
\label{ass-GS}
\end{equation}
The latter solution is a strong one with $\partial_tu$
belonging to $L^\infty$ in time with values in $L^2(D)$.
In \cite{GS05} it was also proved that if \eqref{ass-GS} is fulfilled for $u_0$ replaced by $u_0-u_s$ with 
$u_s$ being a steady solution whose magnitude is controlled by $|\omega|$,
then problem \eqref{ns} admits a global strong solution that
tends to $u_s$ as $t\to\infty$ under smallness conditions on $\|u_0-u_s\|_{W^{1,2}(D)}$ and $|\omega|$,
yielding the stability of $u_s$.
Later on, the stability of steady solutions was studied within the framework of $L^q$ by Shibata and the second
author \cite{HiShi} under the smallness on $\|u_0-u_s\|_{L^{3,\infty}(D)}$ as well as $|\omega|$
with use of $L^q$-$L^r$ decay estimates of the semigroup, where $L^{3,\infty}$ denotes the weak-$L^3$ space (one of the
Lorentz spaces).

When taking into account not only the rotation but also the translation, where constant vectors $\eta$ and $\omega$ are parallel to each other,
$L^q$-$L^r$ estimates as well as generation of the semigroup for the linearized autonomous problem 
were established by Shibata \cite{Shi08, Shi10}.
The similar result for the case of 
a Lipschitz boundary $\partial D$ has been recently 
deduced by Takahashi and Watanabe \cite{TW}, in which one needs a restriction on the summability exponent
dependently on the Lipschitz character.
The latter result seems likely even for the non-autonomous case since the argument is based 
mainly on \cite{Hi18, Hi20} below as well as \cite{GHH06}.

For the non-autonomous regime with time-dependent rigid motions $\eta(t)+\omega(t)\times x$, which 
we are going to address in this paper,
the existence of a unique mild solution locally in time to \eqref{ns} 
was proved by Hansel and Rhandi \cite{HR14} under the assumption
\begin{equation}
\eta,\, \omega\in C^1([0,\infty);\, \mathbb R^3)
\label{ass0-rigid}
\end{equation}
when the initial velocity fulfills
\begin{equation}
u_0\in L^q(D)\;\mbox{with $q\in [3,\infty)$}, \quad 
\nu\cdot \big(u_0-\eta(0)-\omega(0)\times x\big)|_{\partial D}=0, \quad \mbox{div $u_0$}=0,
\label{ass-HR}
\end{equation}
where $\nu$ denotes the outer unit normal
to $\partial D$.
The assumption \eqref{ass0-rigid} is needed only for the well-definedness of a forcing term arising from
a lift of the rigid motion at $\partial D$ as a continuous function in $t$, see \eqref{lift} and \eqref{lift-force} in the next section.
Under less regularity
\begin{equation}
\eta,\;\omega\in C^{0,\theta}_{\rm loc}([0,\infty);\,\mathbb R^3)
\label{ass1-rigid}
\end{equation}
for some $\theta\in (0,1]$,
the essential contribution of \cite{HR14} is to construct the evolution operator $\{T(t,s)\}_{t\geq s\geq 0}$ 
in $L^q$, which provides a solution operator $f\mapsto u(t)=T(t,s)f$ to the initial value problem for the linearized system
\begin{equation}
\begin{split}
&\partial_tu+{\mathcal L}(t)u
+\nabla p=0, \qquad \mbox{div $u$}=0 \quad\mbox{in $D\times (s,\infty)$}, \\
&u|_{\partial D}=0, \qquad \lim_{|x|\to\infty}u=0, \qquad  u(\cdot,s)=f,
\end{split}
\label{linear}
\end{equation}
in their own way without using any existing theory found in \cite{Fri, Ta, Ya}
and to deduce $L^q$-$L^r$ smoothing estimates of $T(t,s)$, where
\begin{equation}
{\mathcal L}(t)=-\Delta -(\eta(t)+\omega(t)\times x)\cdot\nabla+\omega(t)\times.
\label{pre-proj-0}
\end{equation}
The second author \cite{Hi18, Hi20} established $L^q$-$L^r$ decay estimates of this evolution operator 
when assuming
\begin{equation}
\eta,\;\omega\in C^{0,\theta}([0,\infty);\,\mathbb R^3)\cap L^\infty(0,\infty;\,\mathbb R^3)
\label{ass2-rigid}
\end{equation}
for some $\theta\in (0,1]$.
Estimates obtained there were employed to show the existence of time-periodic solutions
(with $\eta$ and $\omega$ being periodic as well as regular to some extent) and their stability, see \cite{Hi22, Huy23}.

All those results for \eqref{ns} mentioned above, except \cite{GS05} in $L^2$, are concerned with mild solutions, that is, 
solutions to the integral equation described by use of the evolution operator or semigroup.
The difficulty of obtaining a strong solution stems from lack of analyticity of the semigroup generated by the corresponding autonomous system, as mentioned before,
and analysis is even more involved on account of the non-autonomous character.
In fact, we do not know whether or not the solution obtained by
\cite{HR14}
becomes a strong one even if assuming more regularity of $\eta,\, \omega$ than \eqref{ass0-rigid}.

The objective of this paper is to provide a strong $L^q$-solution to \eqref{ns}
in the sense that it is of class $C^1$ in time with values in $L^q(D)$
under the assumption
\begin{equation} 
\eta,\;\omega\in C^{1,\vartheta}_{\rm loc}((0,\infty);\,\mathbb R^3), \qquad
|\eta^\prime(t)|+|\omega^\prime(t)|=O(t^{-\gamma})\quad\mbox{as $t\to 0$}
\label{ass3-rigid}
\end{equation}
for some $\vartheta\in (0,1]$ and $\gamma\in [0,1)$, yielding \eqref{ass1-rigid} with $\theta=1-\gamma$
as well as the well-definedness of $\eta,\,\omega$ up to $t=0$,
see \eqref{imp-hoe} and Remark \ref{rem-ns} below,
when the initial velocity satisfies
\begin{equation}
\rho u_0\in L^q(D)\;\mbox{with $q\in (3,\infty)$}, \quad 
\nu\cdot \big(u_0-\eta(0)-\omega(0)\times x\big)|_{\partial D}=0, \quad \mbox{div $u_0$}=0,
\label{ass-IC}
\end{equation}
where $\rho$ stands for the weight function $\rho(x)=1+|x|$, which is related to the coefficient of the drift operator 
$(\omega\times x)\cdot\nabla$.
Our condition \eqref{ass-IC} on $u_0$ should be compared with \eqref{ass-HR} for the mild $L^q$-solution and with \eqref{ass-GS}
for the strong $L^2$-solution with constant $\omega$ (and $\eta=0$).
We emphasize that the result is new even for the autonomous case.
The critical case $q=3$ is unfortunately missing in \eqref{ass-IC}, and this is because the boundedness 
$\|\rho^2Pg\|_q\leq C\|\rho^2g\|_q$ of the Fujita-Kato projection $P$ associated with the Helmholtz decomposition holds true 
in the weighted $L^q$ space if and only if $q\in (3,\infty)$; indeed, this 
comes from the Muckenhoupt theory of the solenoidal $L^q$
space, see Farwig and Sohr \cite{FS97}, and the recent contribution by Tomoki Takahashi \cite{T24}.

To the end described above, a novelty is
to develop the regularity theory of the evolution operator $T(t,s)$
under the reasonable condition \eqref{ass1-rigid} on $\eta,\,\omega$ 
(indeed, the only stage we need  
\eqref{ass3-rigid} is to take care of the rigid motion at $\partial D$, see \eqref{ns}, in constructing a Navier-Stokes flow),
specifically, the H\"older estimate of $t\mapsto T(t,s)f$ and estimates of 
$\partial_tT(t,s)f$ as well as $\nabla^2T(t,s)f$ in $L^q(D)$, 
see \eqref{2nd-sm} for $f\in L^q(D)$, \eqref{str-est} and \eqref{hoe-est} when $\rho f\in L^q(D)$ with $\rho(x)=1+|x|$.
By \cite[(5.18)]{Hi20} we are already aware of 
\eqref{Y-Z} for $\partial_tT(t,s)f$ under the additional condition $\rho\nabla f\in L^q(D)$,
that is however too restrictive to study the nonlinear problem.
On the other hand, without any additional condition, it is possible to deduce a subtle
estimate \eqref{c1-weak} for $\partial_tT(t,s)f$ in $W^{-1,q}(D_R)$, 
where
$D_R=D\cap B_R$, see \cite[Proposition 5.1]{Hi20}; indeed this was useful in \cite{Hi20}, 
but it is weaker than desired.
A key ingredient of our analysis is to make use of weighted estimate, to be precise, estimate of
$\|\rho^\alpha \nabla^jT(t,s)f\|_q$ with $\alpha\in [0,3)$, $j\in \{0,1\}$ and $q\in (\frac{3}{3-\alpha},\infty)$,
see \eqref{wei-est1}.
What should be still done to discuss \eqref{ns} is to study the temporal derivative of the Duhamel term
\begin{equation}
v(t)=\int_s^t T(t,\sigma)g(\sigma)\,d\sigma
\label{duha-0}
\end{equation}
under suitable assumptions on the forcing term $g(t)$.
If the evolution operator were of parabolic type in the sense of Tanabe and Sobolevskii, 
then the theory would be standard, see \cite{Fri, Ta, Ya}; however, it is not the case, so that the issue is quite nontrivial.
The only clue is to look deep into the structure of the evolution operator due to \cite{HR14}, 
that is constructed by use of the evolution operator, $U(t,s)$, in the whole space $\mathbb R^3$ and the one, $V(t,s)$, in a bounded domain
$D_R$ near the body through a certain iteration procedure.
One is able to see from the representation of $\partial_tv(t)$, see \eqref{repre-C1-alt} below which is valid for $q\in (3/2,\infty)$, 
that the argument must be involved.
In view of the structure of this formula, analysis of the aforementioned evolution operator $U(t,s)$ in detail 
for the whole space problem should be an important step.

Although our linear theory is applied merely to the classical setting \eqref{ns} with a prescribed rigid motions in the present paper,
we could think of possibility of applications even to a fluid-structure interaction problem, see \cite{Ga02, Hi24} and the references therein,
where $\eta$ and $\omega$ to be determined through the interaction obey a system of ordinary differential equations
describing the balance for linear and angular momentum.
Since these equations involve $u$ and $p$ through boundary integrals over $\partial D$ being the force and torque
exerted by the fluid,
it turns out from the trace inequality that 
we do need estimates of $\partial_tu, \nabla^2 u$ and $\nabla p$ over a bounded domain $D_R$ near the body 
rather than $D$ itself.
Having this in mind, we deduce estimates of those derivatives of $T(t,s)f$ and $v(t)$ given by \eqref{duha-0}
in $L^q(D_R)$ as well in terms of less information of data, 
see \eqref{str-local-est} and \eqref{duha-str-loc-est}. 

Finally, let us remark that there is the other way to formulate the problem at the beginning (also might be the only way 
when there are several independent rigid bodies)
by use of a local transformation, which keeps the situation far from the body as it is, 
so that the drift term $(\omega\times x)\cdot\nabla u$ does not appear unlike our reduction \eqref{ns}. 
This nonlinear transformation was proposed by Cumsille and Tucsnak \cite{CT06}
to show the local existence of a unique strong solution in $L^2$ for the autonomous case.
Dintelmann, Geissert and Hieber \cite{DGH09} adopted the same transformation even for the case of time-dependent rigid motions
in order to discuss the local well-posedness
within the class of maximal regularity, that is, 
$L^p$ in time with values in $L^q(D)$.
Nevertheless, with this transformation, nontrivial basic flows such as a steady solution are lost and thus one cannot
proceed to the next stage of analysis, say, stability/attainability of those flows; indeed, this 
would be a disadvantage as long as the obstacle is a single rigid body.

The paper is organized as follows.
In section \ref{result}, after some preliminaries, we give our main theorems.
We collect in section \ref{sect-auxi} useful results on several auxiliary problems.
The whole space problem is studied in section \ref{sect-whole}.
Section \ref{sect-wei} is devoted to weighted estimates of the evolution operator. 
In section \ref{sect-strong} we discuss the temporal derivative $\partial_tT(t,s)$ as well as the spatial one $\nabla^2 T(t,s)$, so that
the evolution operator indeed provides a strong solution to the initial value problem \eqref{linear} for the linearized system 
when taking the initial velocity from the weighted space \eqref{sole-chara-wei} below with $\alpha=1$.
The H\"older estimates of the evolution operator are deduced in section \ref{sect-hoe}.
In section \ref{sect-duha} we study the inhomogeneous linear problem with forcing  and derive the regularity estimate of
the function \eqref{duha-0}. 
With complete linear theory at hand, in the final section, the main theorem on local existence of
strong $L^q$-solutions to \eqref{ns} is established.

\section{Main results}
\label{result}

Let us begin with introducing notation.
Given a domain $G\subset \mathbb R^3$, $q\in [1,\infty]$ and integer $k\geq 0$, the standard Lebesgue and Sobolev spaces
are denoted by $L^q(G)$ and $W^{k,q}(G)$.
We abbreviate the norm $\|\cdot\|_{q,G}=\|\cdot\|_{L^q(G)}$ and even $\|\cdot \|_q=\|\cdot\|_{q,D}$,
where $D$ is the exterior domain under consideration with $C^{1,1}$-boundary $\partial D$.
We assume that
\begin{equation}
\mathbb R^3\setminus D\subset B_1,
\label{obsta}
\end{equation}
where $B_R$ denotes the open ball centered at the origin with radius $R>0$.
We set $D_R=D\cap B_R$ for $R\geq 1$.
The class $C_0^\infty(G)$ comprises all $C^\infty$-functions with compact support in $G$ and the space $W^{k,q}_0(G)$
with integer $k\geq 0$ stands for the completion of $C_0^\infty(G)$ in $W^{k,q}(G)$. 
By $W^{-1,q}(G)$ we denote the dual space of $W^{1,q^\prime}_0(G)$, where $1/q^\prime+1/q=1$.
We adopt the same symbols for denoting scalar and vector function spaces as long as there is no confusion.
Let $X_1$ and $X_2$ be two Banach spaces.
Then ${\mathcal L}(X_1,X_2)$ stands for the Banach space consisting of all bounded linear operators from $X_1$ into $X_2$.
We simply write ${\mathcal L}(X_1)={\mathcal L}(X_1,X_1)$.
Several positive constants are denoted by $C$, which may change from line to line.

Let us introduce the solenoidal function spaces.
Let $G\subset \mathbb R^3$ be one of the following domains;
the exterior $C^{1,1}$-domain $D$ under consideration, a bounded domain $D_R$ for $R\geq 1$, and the whole space $\mathbb R^3$.
The class $C_{0,\sigma}^\infty(G)$ consists of all divergence-free vector fields being in $C_0^\infty(G)$.
Let $1<q<\infty$, and denote by $L^q_\sigma(G)$ the completion of $C_{0,\sigma}^\infty(G)$, then it is characterized as
\begin{equation}
L^q_\sigma(G)=\{u\in L^q(G);\; \mbox{div $u$}=0,\, \nu\cdot u|_{\partial G}=0\},
\label{sole-chara}
\end{equation}
where $\nu$ stands for the outer unit normal to $\partial G$ and $\nu\cdot u$ is understood in the sense of normal trace
on $\partial G$ (this boundary condition is absent when $G=\mathbb R^3$).
The space of $L^q$-vector fields admits the Helmholtz decomposition
\begin{equation}
L^q(G)=L^q_\sigma(G)\oplus \{\nabla p\in L^q(G);\; p\in L^q_{\rm loc}(\overline{G})\}
\label{helm}
\end{equation}
which was proved by Fujiwara and Morimoto \cite{FM}, Miyakawa \cite{Mi}, and Simader and Sohr \cite{SiSo}.
By $P_G=P_{G,q}: L^q(G)\to L^q_\sigma(G)$, we denote the Fujita-Kato bounded projection associated with the decomposition above.
Note the duality relation $(P_{G,q})^*=P_{G,q^\prime}$ as well as $L^q_\sigma(G)^*=L^{q^\prime}_\sigma(G)$, where
$1/q^\prime+1/q=1$.
We simply write $P=P_D$ for the exterior domain $D$ under consideration.
When $G=\mathbb R^3$, we see that $P_{\mathbb R^3}=I+{\mathcal R}\otimes {\mathcal R}$, where
${\mathcal R}=\nabla (-\Delta)^{-1/2}$ denotes the Riesz transform and $I$ stands for the identity operator.

We are in a position to introduce the generator which is associated with the non-autonomous linearized system \eqref{linear}.
Suppose that the rigid motion $\eta(t)+\omega(t)\times x$ fulfills \eqref{ass1-rigid} for some $\theta\in (0,1]$ as in \cite{HR14}, and set
\begin{equation}
\begin{split}
&|(\eta,\omega)|_{0,{\mathcal T}}:=\sup_{0\leq t\leq {\mathcal T}}\; (|\eta(t)|+|\omega(t)|), \\
&|(\eta,\omega)|_{\theta,{\mathcal T}}:=\sup_{0\leq s<t\leq {\mathcal T}}\frac{|\eta(t)-\eta(s)|+|\omega(t)-\omega(s)|}{(t-s)^\theta}, \\
\end{split}
\label{quan1}
\end{equation}
for ${\mathcal T}>0$.
Note that \eqref{ass3-rigid} is not needed until analyzing the nonlinear problem \eqref{ns}.
Let $1<q<\infty$. We define the family of linear operators $\{L(t)\}_{t\geq 0}$ on $L^q_\sigma(D)$ by
\begin{equation}
\begin{split}
D_q(L(t))&=\{u\in D_q(A);\; (\omega(t)\times x)\cdot\nabla u\in L^q(D)\}, \\
L(t)u&=P{\mathcal L}(t)u 
=Au-(\eta(t)+\omega(t)\times x)\cdot\nabla u+\omega(t)\times u,
\end{split}
\label{ig}
\end{equation}
where ${\mathcal L}(t)$ is the differential operator given by \eqref{pre-proj-0} and
$A$ denotes the Stokes operator on $L^q_\sigma(D)$, that is,
\begin{equation}
D_q(A)=L^q_\sigma(D)\cap W^{1,q}_0(D)\cap W^{2,q}(D), \qquad Au=-P\Delta u.
\label{stokes}
\end{equation}
In \eqref{ig} we have taken into account
\begin{equation}
P\big[(\eta+\omega\times x)\cdot\nabla u-\omega\times u\big]=(\eta+\omega\times x)\cdot\nabla u-\omega\times u
\label{proj-drift}
\end{equation}
and this is useful to show \eqref{str-local-est} and \eqref{Av-local} below. 
The reason why \eqref{proj-drift} holds true is that
\begin{equation}
\begin{split}
&\mbox{div $\big[(\eta+\omega\times x)\cdot\nabla u-\omega\times u\big]$}=(\eta+\omega\times x)\cdot\nabla\mbox{div $u$}=0, \\
&\nu\cdot\big[(\eta+\omega\times x)\cdot\nabla u-\omega\times u\big]\big|_{\partial D}=0,
\end{split}
\label{grund}
\end{equation}
see \eqref{sole-chara}--\eqref{helm}.
The latter of \eqref{grund} in our context was observed by Galdi and Silvestre \cite{GS05} when $q=2$.
Their finding is still valid even for $1<q<\infty$ since
the normal trace of the drift term vanishes, namely,
$\nu\cdot (\partial_ju)|_{\partial D}=0$ for every $j\in \{1,2,3\}$ and $u$ being in $L^q_\sigma(D)\cap W^{1,q}_0(D)$
that coincides with 
the completion of $C_{0,\sigma}^\infty(D)$ in $W^{1,q}(D)$.
As for the coincidence of those spaces, which the point here, see, for instance,
Galdi \cite[Theorem III.4.2]{Ga-b}; further references are found in \cite[Section III.7]{Ga-b}.
The initial value problem \eqref{linear} is now formulated as
\begin{equation}
\frac{du}{dt}+L(t)u=0, \quad t\in (s,\infty); \qquad u(s)=f
\label{evo}
\end{equation}
in $L^q_\sigma(D)$.

To overcome the difficulty that the domain $D_q(L(t))$ is time-dependent, the idea due to Hansel and Rhandi \cite{HR14}
is to introduce the regularity space
\begin{equation}
Y_q(G):=\{u\in L^q_\sigma(G)\cap W^{1,q}_0(G)\cap W^{2,q}(G);\; \rho\nabla u\in L^q(G)\}, \qquad G=D,\,\mathbb R^3, 
\label{Y}
\end{equation}
endowed with norm
$\|u\|_{Y_q(G)}=\|u\|_{q,G}+\|\rho \nabla u\|_{q,G}+\|\nabla^2u\|_{q,G}$.
The weight function $\rho$ is given by
\begin{equation}
\rho(x):=1+|x|
\label{weight}
\end{equation}
throughout this paper.
It is obvious that $Y_q(D)\subset D_q(L(t))$ for every $t\geq 0$.
The auxiliary space adopted in \cite{Hi20} is
\begin{equation}
Z_q(G):=\{u\in L^q_\sigma(G);\; \rho\nabla u\in L^q(G)\}, \qquad G=D,\, \mathbb R^3,
\label{Z}
\end{equation}
endowed with norm
$\|u\|_{Z_q(G)}=\|u\|_{q,G}+\|\rho\nabla u\|_{q,G}$,
where the homogeneous Dirichlet condition is not involved, while the similar space in \cite{HR14} is contained in $W^{1,q}_0(D)$.
This modification was indeed necessary in \cite[Section 5]{Hi20}.
The other similar space with the weight $\rho^\alpha$ to \eqref{Z} will be needed later, see \eqref{Z-ge},
where $\alpha\in [0,3)$.
Due to Hansel and Rhandi \cite{HR14}, along with a slight refinement by the second author \cite[Proposition 2.1]{Hi20},
we know the following.
\begin{proposition}
Suppose that $\eta$ and $\omega$ fulfill \eqref{ass1-rigid} for some $\theta\in (0,1]$.
Let $q\in (1,\infty)$.
The operator family $\{L(t)\}_{t\geq 0}$ given by \eqref{ig} generates an evolution operator
$\{T(t,s)\}_{t\geq s\geq 0}$ on $L^q_\sigma(D)$ such that $T(t,s)$ is a bounded operator from $L^q_\sigma(D)$
into itself with the semigroup property
\[
T(t,\tau)T(\tau,s)=T(t,s)\quad (t\geq \tau\geq s\geq 0); \qquad T(s,s)=I,
\]
in ${\mathcal L}(L^q_\sigma(D))$ with $I$ being the identity operator and that the map
\[
\{t\geq s\geq 0\}\ni (t,s)\mapsto T(t,s)f\in L^q_\sigma(D)
\]
is continuous for every $f\in L^q_\sigma(D)$.
Furthermore, we have the following properties.

\begin{enumerate}
\item
Let $j\in \{0,1\}$ and $r\in [q,\infty]$. 
For each ${\mathcal T}\in (0,\infty)$ and $m\in (0,\infty)$, there is a constant 
$C=C({\mathcal T},m,q,r,\theta,D)>0$ such that
\begin{equation}
\|\nabla^jT(t,s)f\|_r\leq C(t-s)^{-j/2-(3/q-3/r)/2}\|f\|_q
\label{sm}
\end{equation}
for all $(t,s)$ with $0\leq s<t\leq {\mathcal T}$ and $f\in L^q_\sigma(D)$ whenever
\begin{equation}
|(\eta,\omega)|_{0,{\mathcal T}}+|(\eta,\omega)|_{\theta,{\mathcal T}}\leq m,
\label{ever}
\end{equation}
see \eqref{quan1}.

\item
Let $q\in (\frac{3}{2},\infty)$. 
For every $f\in Z_q(D)$ we have
\begin{equation}
\begin{split}
&T(\cdot,s)f\in C^1((s,\infty);\, L^q_\sigma(D))\cap C((s,\infty);\, W^{2,q}(D)), \\
&T(t,s)f\in Y_q(D)\quad \forall\,t\in (s,\infty),  \quad
L(\cdot)T(\cdot,s)f\in C((s,\infty);\, L^q_\sigma(D)),
\end{split}
\label{st-cl}
\end{equation}
with
\begin{equation}
\partial_tT(t,s)f+L(t)T(t,s)f=0, \qquad t\in (s,\infty),
\label{str}
\end{equation}
in $L^q_\sigma(D)$. 

Let $\delta\in (0,\frac{1}{2q})$. 
For each ${\mathcal T}\in (0,\infty)$ and $m\in (0,\infty)$, there is a constant
$C=C({\mathcal T},m,q,\delta,\theta,D)>0$ such that 
\begin{equation}
\|\partial_tT(t,s)f\|_q+
\|T(t,s)f\|_{Y_q(D)}\leq C(t-s)^{-1+\delta}\|f\|_{Z_q(D)}
\label{Y-Z}
\end{equation}
for all $(t,s)$ with $0\leq s<t\leq {\mathcal T}$ and $f\in Z_q(D)$ whenever \eqref{ever} is satisfied.
Here, the spaces $Y_q(D)$ and $Z_q(D)$ are given by \eqref{Y} and \eqref{Z}, respectively.

\item
Fix $t>0$. For every $f\in Y_q(D)$, we have
\[
T(t,\cdot)f\in C^1([0,t];\, L^q_\sigma(D))
\]
with
\[
\partial_sT(t,s)f=T(t,s)L(s)f, \qquad s\in [0,t],
\]
in $L^q_\sigma(D)$.

%\item
\end{enumerate}
\label{so-far}
\end{proposition}

Several remarks are in order.

As for the smoothing estimate \eqref{sm}, the case $(j,r)=(1,\infty)$ is missing in \cite{HR14, Hi20}, but this follows from
\eqref{wei-est1} with $\alpha=0$ in Theorem \ref{weighted-est} or from \eqref{2nd-sm} in Theorem \ref{strong-sol}
of the present paper.

Among the regularity described in \eqref{st-cl}, the continuity with values in $W^{2,q}(D)$ is not explicitly claimed in \cite{HR14, Hi20},
but it is indeed true.
In fact, by Remark \ref{rem-hoe-wh} together with \eqref{reg-sob-bdd}
we find that 
$W(\cdot,s)f$ with $f\in Z_q(D)$ 
is continuous on $(s,\infty)$ with values in $W^{2,q}(D)$, 
and thereby, so is $T_j(\cdot,s)f$ for every $j\geq 0$, 
see \eqref{appro} and \eqref{iteration}, provided $q\in (\frac{3}{2},\infty)$,
where the right continuity is easy to show, while the proof
of the left continuity needs a bit device done by Tomoki Takahashi \cite[Lemma 4.6]{T22}.
Since the evolution operator \eqref{evo-series} is convergent in $W^{2,q}(D)$ uniformly in $(t,s)$ with
$0\leq s<t\leq {\mathcal T}$ and $t-s\geq \varepsilon$ for every $\varepsilon \in (0,{\mathcal T})$, see Lemma \ref{lem-ite},
we are led to $T(\cdot,s)f\in C((s,\infty);\, W^{2,q}(D))$.
From \eqref{st-cl} it follows that $(\omega(\cdot)\times x)\cdot\nabla T(\cdot,s)f\in C((s,\infty);\, L^q(D))$
for every $f\in Z_q(D)$, however,
it does not seem to be clear whether 
$T(\cdot,s)f\in C((s,\infty);\, Y_q(D))$ for such $f$, see Remark \ref{rem-conti} below.

From the Muckenhoupt theory of singular integrals, see 
Farwig and Sohr \cite[Section 2]{FS97}, Stein \cite[Chapter V]{St} and Torchinsky \cite[Chapter IX]{Tor},
the Riesz transform
${\mathcal R}=\nabla (-\Delta)^{-1/2}$ satisfies
$\|\rho{\mathcal R}h\|_{q,\mathbb R^3}\leq C\|\rho h\|_{q,\mathbb R^3}$
with the weight \eqref{weight} if and only if $\frac{n}{n-1}=\frac{3}{2}<q<\infty$ ($n$ denotes the space dimension),
see \eqref{riesz} below for a bit generalized form.
The restriction $q\in (\frac{3}{2},\infty)$ in the item 2 of Proposition \ref{so-far} stems from this fact, 
and a key observation for the proof is that
the Fujita-Kato projection $P$ enjoys \eqref{key-20} below with $\alpha=1$
for such $q$, 
see \cite[Lemma 5.2]{Hi20}.
Thus, it would be difficult to remove the restriction on $q$.
Estimate \eqref{Y-Z} can be proved under the assumption \eqref{ass1-rigid} as in \cite[(5.18)]{Hi20},
where the same estimate for all $(t,s)$ with $0<t-s\leq {\mathcal T}$ was deduced under \eqref{ass2-rigid}.
Nevertheless, the smoothing rate in \eqref{Y-Z} does not seem to be sharp.
The item 2 of Proposition \ref{so-far} tells us that $u(t)=T(t,s)f$ with $f\in Z_q(D)$ 
provides a strong solution to \eqref{evo}, however,
the condition $\rho\nabla f\in L^q(D)$ is not very useful to proceed to the nonlinear stage.
It is reasonable to expect that a strong solution is still available
even if replacing this condition by $\rho f\in L^q(D)$.
Having this in mind, we introduce weighted $L^q$ spaces of solenoidal vector fields
with the specific weight of the form $\rho^\alpha$ 
by following Farwig and Sohr \cite{FS97}, who discussed more general weights.

We denote by $G$ either the exterior domain $D$ under consideration or the whole space $\mathbb R^3$.
Let $q\in [1,\infty]$ and $\alpha\geq 0$, then
the space $L^q_\alpha(G)$ consists of all measurable functions $f$ such that $\rho^\alpha f\in L^q(G)$
with the weight function $\rho$ given by \eqref{weight}.
For $q\in (1,\infty)$ we also define the space $L^q_{\alpha,\sigma}(G)$ of solenoidal vector fields by
\begin{equation}
L^q_{\alpha,\sigma}(G)
:=L^q_\alpha(G)\cap L^q_\sigma(G)
=\left\{
\begin{array}{ll}
\{u\in L^q_\alpha(D);\, \mbox{div $u$}=0,\, \nu\cdot u|_{\partial D}=0\},\quad &G=D,  \\
\{u\in L^q_\alpha(\mathbb R^3);\, \mbox{div $u$}=0\}, &G=\mathbb R^3, \\
\end{array}
\right.
\label{sole-chara-wei}
\end{equation}
see \eqref{sole-chara}.
Both $L^q_\alpha(G)$ and $L^q_{\alpha,\sigma}(G)$
are Banach spaces endowed with norm $\|\rho^\alpha(\cdot)\|_{q,G}$.
If, in particular, $\alpha\in (0,3)$ and $q\in (\frac{3}{3-\alpha},\infty)$, then the weight function $\rho^{\alpha q}$
belongs to the Muckenhoupt class ${\mathcal A}_q(G)$, see \cite[Section 2]{FS97}
for the details, which implies that singular integral operators are bounded with respect to the 
norm $\|\rho^\alpha(\cdot)\|_{q,G}$.
What we particularly need is that 
the Fujita-Kato projection $P$ 
(through the weak Neumann problem) enjoys
\begin{equation}
\|\rho^\alpha P_Gf\|_{q,G}\leq C\|\rho^\alpha f\|_{q,G}, \qquad G=D,\, \mathbb R^3,
\label{FK-wei}
\end{equation}
for every vector field $f\in L^q_\alpha(G)$ 
as long as $\alpha\in (0,3)$ and $q\in (\frac{3}{3-\alpha},\infty)$.
For such $\alpha$ and $q$,
Farwig and Sohr \cite[Theorem 1.3, Corollary 4.4]{FS97} 
established the Helmholtz decomposition of the weighted space of vector fields
\[
L^q_\alpha(G)=L^q_{\alpha,\sigma}(G)\oplus \{\nabla p\in L^q_\alpha(G);\, p\in L^q_{\rm loc}(\overline{G})\}
\]
and showed that the class
$C_{0,\sigma}^\infty(G)$ 
is dense in $L^q_{\alpha,\sigma}(G)$.
When $G=D$, this denseness
had been already found by
Specovius-Neugebauer \cite[Lemma 9]{SN90} 
for the case of the specific weight $\rho^\alpha$ 
even before \cite{FS97}.

We do not know whether the evolution operator $T(t,s)$ obtained in Proposition \ref{so-far}
satisfies the continuity
$(t,s)\mapsto T(t,s)f\in L^q_{\alpha,\sigma}(D)$ for each $f\in L^q_{\alpha,\sigma}(D)$, where $\alpha\in (0,3)$ and 
$q\in (\frac{3}{3-\alpha},\infty)$, 
however, we do not have to regard it as the evolution operator on $L^q_{\alpha,\sigma}(D)$ at all in the present paper.
For instance, in Theorem \ref{duha-strong} below,
we may consider the function \eqref{duha} under the assumptions \eqref{ass2-force}--\eqref{ass3-force} with
$L^q_{2,\sigma}(D)$ defined by \eqref{sole-chara-wei} even for $q\leq 3$, although
this space can be no longer the solenoidal part of the 
aforementioned Helmholtz decomposition of $L^q_2(D)$
unless $q\in (3,\infty)$.

Let us begin with smoothing estimates of the evolution operator near $t=s$ in the weighted space.
The same estimates for the Stokes semigroup are well known and those for the Oseen semigroup even with
an anisotropic weight have been deduced by Tomoki Takahashi \cite[Theorem 2.4]{T24}.
\begin{theorem}
Suppose that $\eta$ and $\omega$ fulfill \eqref{ass1-rigid} for some $\theta\in (0,1]$.
Let $\alpha\in [0,3)$, $q\in (\frac{3}{3-\alpha},\infty)$, $r\in [q,\infty]$ and $j\in \{0,1\}$.
For each ${\mathcal T}\in (0,\infty)$ and $m\in (0,\infty)$, 
there is a constant $C=C({\mathcal T},m,\alpha,q,r,\theta,D)>0$ such that
\begin{equation}
\|\rho^\alpha\nabla^jT(t,s)f\|_r
\leq C(t-s)^{-j/2-(3/q-3/r)/2}\|\rho^\alpha f\|_q
\label{wei-est1}
\end{equation}
for 
all $(t,s)$ with $0\leq s<t\leq {\mathcal T}$ and $f\in L^q_{\alpha,\sigma}(D)$ whenever
\eqref{ever}
is satisfied, where the weight function $\rho$ is given by \eqref{weight}.
Furthermore, we have
\begin{equation}
\lim_{t\to s}\, (t-s)^{j/2+(3/q-3/r)/2}\|\rho^\alpha\nabla^jT(t,s)f\|_r=0
\label{wei-small1}
\end{equation}
for every $f\in L^q_{\alpha,\sigma}(D)$ except the case $(j,r)=(0,q)$;
in that case, it holds that
\begin{equation}
\lim_{t\to s}\|\rho^\alpha\big(T(t,s)f-f\big)\|_q=0
\label{IC-1}
\end{equation}
for every $f\in L^q_{\alpha,\sigma}(D)$.
\label{weighted-est}
\end{theorem}

We next show that the evolution operator indeed provides a strong solution to \eqref{evo} for every $f\in L^q_{1,\sigma}(D)$,
see \eqref{sole-chara-wei}, as long as $q\in (3/2,\infty)$.
In fact, the item 3 of the following theorem is an improvement of the item 2 of Proposition \ref{so-far}.
\begin{theorem}
Suppose that $\eta$ and $\omega$ fulfill \eqref{ass1-rigid} for some $\theta\in (0,1]$.

\begin{enumerate}
\item
Let $1<q\leq r<\infty$.
For each ${\mathcal T}\in (0,\infty)$ and $m\in (0,\infty)$, there is a constant
$C=C({\mathcal T},m,q,r,\theta,D)>0$ such that
\begin{equation}
\|\nabla^2 T(t,s)f\|_r\leq C(t-s)^{-1-(3/q-3/r)/2}\|f\|_q
\label{2nd-sm}
\end{equation}
for all $(t,s)$ with $0\leq s<t\leq {\mathcal T}$ and $f\in L^q_\sigma(D)$ whenever \eqref{ever} is satisfied.

\item
Let $q\in (\frac{3}{2},\infty)$.
For every $f\in L^q_\sigma(D)$ and $t\in (s,\infty)$, we have $T(t,s)f\in D_q(A)$, see \eqref{stokes}, and therefore,
$T(t,s)f|_{\partial D}=0$ in the sense of trace.

For each ${\mathcal T}\in (0,\infty)$ and $m\in (0,\infty)$, there is a constant $C=C({\mathcal T},m,q,\theta,D)>0$ such that
\begin{equation}
\|AT(t,s)f\|_q\leq C(t-s)^{-1}\|f\|_q
\label{A-evo}
\end{equation}
for all $(t,s)$ with $0\leq s<t\leq {\mathcal T}$ and $f\in L^q_\sigma(D)$ whenever \eqref{ever} is satisfied.

\item
Let $q\in (\frac{3}{2},\infty)$. 
For every $f\in L^q_{1,\sigma}(D)$, the function 
$T(t,s)f$ 
is of class \eqref{st-cl}
and satisfies \eqref{str} in $L^q_\sigma(D)$.

For each ${\mathcal T}\in (0,\infty)$,
$m\in (0,\infty)$ and  $R\in (1,\infty)$, there
are constants $C=C({\mathcal T},m,q,\theta,D)>0$ and $C^\prime=C^\prime({\mathcal T},m,q,R,\theta,D)>0$ such that 
\begin{equation}
\|\partial_tT(t,s)f\|_q+\|T(t,s)f\|_{Y_q(D)}+\|\nabla p(t)\|_q  
\leq C(t-s)^{-1}\|f\|_q+C(t-s)^{-1/2}\|\rho f\|_q
\label{str-est}
\end{equation}
\begin{equation}
\|\partial_tT(t,s)f\|_{q,D_R}+\|\nabla p(t)\|_{q,D_R}
\leq C^\prime(t-s)^{-1}\|f\|_q
\label{str-local-est}
\end{equation}
for all $(t,s)$ with $0\leq s<t\leq {\mathcal T}$ and $f\in L^q_{1,\sigma}(D)$ whenever \eqref{ever} is satisfied,
where $p(t)$ is the pressure associated with $T(t,s)f$.
\end{enumerate}
\label{strong-sol}
\end{theorem}
\begin{remark}
In \cite{Hi99} the second author studied 
the autonomous case, in which $\omega\in \mathbb R^3\setminus \{0\}$ is a constant vector and $\eta=0$, 
and proved that
the semigroup $e^{-tL}f$ provides a strong solution in $L^2_\sigma(D)$ if, for instance, $(\omega\times x)\cdot\nabla f\in W^{-1,2}(D)$
as well as $f\in D_2(A^\delta)$ with arbitrarily small $\delta>0$, where the desired case $\delta=0$ was at that time 
excluded for some technical reason.
This condition already suggested our condition $f\in L^q_{1,\sigma}(D)$ in the item 3 above.
We observe that \eqref{str-est} is consistent with (2.10) of \cite{Hi99} for $(\delta,s)=(0,1)$. 
\label{rem-ic}
\end{remark}
\begin{remark}
In \eqref{str-est} we have less singularity with the weighted norm of initial data near the initial time.
This must be useful to investigate the regularity of the Duhamel term \eqref{duha-0}.
For \eqref{str-local-est} near the body, 
the weighted norm is no longer needed and this might help us to study the fluid-structure interaction problem, see section \ref{intro}.
Estimate \eqref{str-local-est} can be compared with the weak estimate
\begin{equation}
\|\partial_tT(t,s)f\|_{W^{-1,q}(D_R)}+\|p(t)\|_{q,D_R}
\leq C(t-s)^{-(1+1/q)/2}\|f\|_q
\label{c1-weak}
\end{equation}
for all $(t,s)$ with $0\leq s<t\leq {\mathcal T}$ and $f\in L^q_\sigma(D)$, where $1<q<\infty$.
Here, the pressure $p(t)$ is singled out in such a way that
$\int_{D_R}p(t)\,dx=0$.
In \cite[Propsition 5.1]{Hi20} this was proved for all $(t,s)$ with $0< t-s\leq {\mathcal T}$ under the condition \eqref{ass2-rigid},
but the proof of \eqref{c1-weak} for $(t,s)$ mentioned above under \eqref{ass1-rigid} is exactly the same.
\label{rem-str-est}
\end{remark}
\begin{remark}
For the autonomous case, estimate \eqref{2nd-sm} with $r=q$
for the spatial derivative of second order was already found by Shibata \cite{Shi10}.
\label{rem-2nd-sm}
\end{remark}
\begin{remark}
The desired continuity 
$T(\cdot,s)f\in C((s,\infty);\, Y_q(D))$ for $f\in L^q_{1,\sigma}(D)$
is still unclear,
see Remark \ref{rem-2nd-wh} for the whole space problem.
If, in particular, $\omega(t)$ keeps a constant direction for all $t$ such that
$D_2(L(t))=\widetilde Y_2(D):=\{u\in D_2(A); (e_3\times x)\cdot\nabla u\in L^2(D)\}$,
then one can apply the Kato theory of
evolution operators of hyperbolic type (\cite{Ta})
to show that $T(\cdot,s)f\in C((s,\infty);\widetilde Y_2(D))$ for $f\in \widetilde Y_2(D)$, see \cite{Hi01}, 
and thereby for $f\in L^2_{1,\sigma}(D)$
on account of the item 3 of Theorem \ref{strong-sol}.
\label{rem-conti}
\end{remark}

The H\"older estimate of the evolution operator is a consequence of Theorems \ref{weighted-est} and \ref{strong-sol}.
\begin{theorem}
Suppose that $\eta$ and $\omega$ fulfill \eqref{ass1-rigid} for some $\theta\in (0,1]$.
Let $q\in (\frac{3}{2},\infty)$, $r\in [q,\infty]$ and $j\in \{0,1\}$. 
For each ${\mathcal T}\in (0,\infty)$, $m\in (0,\infty)$ and
given $\mu$ satisfying
\begin{equation*}
\begin{array}{ll}
0<\mu\leq 1&(j=0,\,r<\infty), \\
0<\mu<1 &(j=0,\,r=\infty), \\
0<\mu\leq 1/2 &(j=1,\, r<\infty), \\
0<\mu<1/2 \qquad &(j=1,\, r=\infty),
\end{array}
\end{equation*}
there is a constant $C=C({\mathcal T},m,q,r,\mu,\theta,D)>0$ such that 
\begin{equation}
\|\nabla^j T(t,s)f-\nabla^j T(\tau,s)f\|_r
\leq C(t-\tau)^\mu(\tau-s)^{-j/2-(3/q-3/r)/2-\mu}\|\rho f\|_q
\label{hoe-est}
\end{equation}
for all $(t,\tau,s)$ with $0\leq s<\tau<t\leq {\mathcal T}$
and $f\in L^q_{1,\sigma}(D)$ whenever \eqref{ever} is satisfied.
\label{hoelder}
\end{theorem}
\begin{remark}
In fact, to be more precise than \eqref{hoe-est}, we deduce
\eqref{hoe-est1} and \eqref{hoe-est2}--\eqref{hoe-est4}, in which the weighted norm is needed only for less singular terms
near $\tau=s$.
\label{rem-precise}
\end{remark}

Let $q\in (\frac{3}{2},\infty)$ and
$0\leq s<{\mathcal T}<\infty$.
Assume that
\begin{equation}
(\cdot -s)^\kappa g\in L^\infty(s,{\mathcal T};\,L^q_{1,\sigma}(D))
\label{ass1-force}
\end{equation}
with some $\kappa\in [0,1)$.
Then one can derive the H\"older continuity of the function
\begin{equation}
v(t)=\int_s^tT(t,\sigma)g(\sigma)\,d\sigma \qquad (s<t\leq {\mathcal T})
\label{duha}
\end{equation}
in the standard manner
as a simple corollary to Theorem \ref{hoelder}.
\begin{corollary}
In addition to the assumptions of Theorem \ref{hoelder} on $\eta, \omega, q, r$ and $j$, 
assume further that $q$ and $r$ satisfy
\[
\frac{1}{q}-\frac{1}{r}<\frac{1}{3}\qquad\mbox{when $j=1$}
\]
and that, given $s$ and ${\mathcal T}$ with $0\leq s<{\mathcal T}<\infty$, 
the function $g(t)$ fulfills \eqref{ass1-force} for some $\kappa\in [0,1)$.
For each $m\in (0,\infty)$ and given $\mu$ satisfying
\begin{equation}
\kappa-\frac{j}{2}-\beta
<\mu<1-\frac{j}{2}-\beta, \qquad \beta:=\frac{3}{2}\left(\frac{1}{q}-\frac{1}{r}\right),
\label{duha-hoe-exp}
\end{equation}
there is a constant $C=C({\mathcal T},m,q,r,\mu,\kappa,\theta,D)>0$ 
such that the function $v(t)$ given by \eqref{duha} enjoys
\begin{equation}
\|\nabla^jv(t)-\nabla^jv(\tau)\|_r\leq C
(t-\tau)^\mu (\tau-s)^{-j/2-\beta-\mu}({\mathcal T}-s)^{1-\kappa}[g]_{q,1,\kappa}
\label{duha-hoe-est}
\end{equation}
for all $(t,\tau)$ with $s<\tau<t\leq {\mathcal T}$ whenever \eqref{ever} is satisfied, where
\begin{equation}
[g]_{q,\alpha,\kappa}:=
\sup_{\sigma\in (s,{\mathcal T})}(\sigma-s)^{\kappa}\|\rho^\alpha g(\sigma)\|_q.
\label{g-wei}
\end{equation}
\label{duha-hoelder}
\end{corollary}

We remark that the coefficient 
$({\mathcal T}-s)^{1-\kappa}$ in the right-hand side of \eqref{duha-hoe-est} sometimes plays a bit role in applications,
see section \ref{sect-ns}.

Under further conditions
\begin{equation}
(\cdot-s)^\kappa g\in L^\infty(s,{\mathcal T};\, L^q_{2,\sigma}(D))
\label{ass2-force}
\end{equation}
\begin{equation}
g\in C((s,{\mathcal T}];\, L^q_\sigma(D)), \qquad
\|g(t)-g(\tau)\|_q
\leq C(t-\tau)^\mu (\tau-s)^{-\kappa-\mu}\qquad \Big(\frac{s+t}{2}<\tau<t\leq {\mathcal T}\Big)
\label{ass3-force}
\end{equation}
with some $\mu\in (0,1]$, $\kappa\in [0,1)$ 
and $C=C({\mathcal T})>0$, 
we find the following theorem which tells us that
the function $v(t)$ given by \eqref{duha} provides a strong solution to the inhomogeneous equation with the forcing term $g(t)$.
In \eqref{ass2-force} one asks more weight than the previous condition \eqref{ass1-force}.
The reason is related to the non-autonomous character
and stems only from \eqref{evo-semi-wh-0} for the whole space problem. 
On the other hand,
the reason why the weight is not needed in the H\"older estimate \eqref{ass3-force} 
is that we have less singularity with the weighted norm in \eqref{str-est}.
Notice that the singularity in \eqref{ass3-force} near $\tau=s$ is consistent with \eqref{ass2-force}.
The quantity relating to \eqref{ass3-force} is introduced as \eqref{g-hoe-semi} below, in which the supremum
is taken with respect to $t\in (s,{\mathcal T}]$ and $\sigma\in (\frac{s+t}{2},t)$.
\begin{theorem}
Suppose that $\eta$ and $\omega$ fulfill \eqref{ass1-rigid} for some $\theta\in (0,1]$.
Let $q\in (\frac{3}{2},\infty)$ and assume that, given $s$ and ${\mathcal T}$ with $0\leq s<{\mathcal T}<\infty$,
the function $g(t)$ fulfills \eqref{ass2-force}--\eqref{ass3-force} for some $\mu\in (0,1]$ and $\kappa\in [0,1)$.
Then the function $v(t)$ given by \eqref{duha} is of class
\begin{equation}
\begin{split}
&v\in C^1((s,{\mathcal T}];\, L^q_\sigma(D)), \qquad
v(t)\in Y_q(D)\quad \forall\,t\in (s,{\mathcal T}],  \\
&L(\cdot)v\in C((s,{\mathcal T}];\, L^q_\sigma(D))
\end{split}
\label{duha-cl}
\end{equation}
and satisfies
\begin{equation}
\frac{dv}{dt}+L(t)v=g(t), \qquad t\in (s,{\mathcal T}]
\label{duha-eq}
\end{equation}
in $L^q_\sigma(D)$.

For each $m\in (0,\infty)$ and  $R\in (1,\infty)$, there
are constants $C_1=C_1({\mathcal T},m,q,\mu,\kappa,\theta,D)>0$, $C_2=C_2({\mathcal T},m,q,\kappa,\theta,D)>0$ and 
$C_3=C_3({\mathcal T},m,q,R,\mu,\kappa,\theta,D)>0$ such that
\begin{equation}
\|\partial_tv(t)\|_q+\|\nabla^2v(t)\|_q 
\leq C_1(t-s)^{-\kappa} \big([g]_{q,2,\kappa}+\{g\}_{q,\mu,\kappa}\big)
\label{duha-str-est}
\end{equation}
\begin{equation}
\|\rho\nabla v(t)\|_q
\leq C_2(t-s)^{-\kappa+1/2}[g]_{q,1,\kappa}
\label{duha-drift-est}
\end{equation}
\begin{equation}
\|\partial_tv(t)\|_{q,D_R}+\|\nabla^2v(t)\|_{q,D_R}
\leq C_3(t-s)^{-\kappa} 
\big([g]_{q,1,\kappa}+\{g\}_{q,\mu,\kappa}\big)
\label{duha-str-loc-est}
\end{equation}
for all $t\in (s,{\mathcal T}]$ whenever \eqref{ever} is satisfied, where
\begin{equation}
\{g\}_{q,\mu,\kappa}:=
\sup_{\frac{s+t}{2}<\sigma<t\leq {\mathcal T}}\frac{(\sigma-s)^{\kappa+\mu} \|g(t)-g(\sigma)\|_q}{(t-\sigma)^\mu}
\label{g-hoe-semi}
\end{equation}
while $[g]_{q,\alpha,\kappa}$ is defined as \eqref{g-wei}.

If, in particular, the generator $L(t)=L$ is independent of $t$ (so that the evolution operator is the semigroup $e^{-tL}$),
then all these results above are available under less assumptions \eqref{ass1-force} and \eqref{ass3-force}, where
$[g]_{q,2,\kappa}$ is replaced by $[g]_{q,1,\kappa}$ in \eqref{duha-str-est},
and we have $v\in C((s,{\mathcal T}];\,W^{2,q}(D))$. 
\label{duha-strong}
\end{theorem}
\begin{remark}
In the existing literature such as Yagi \cite{Ya} on evolution operators of parabolic type, 
\eqref{ass3-force} for $s<\tau<t\leq {\mathcal T}$ with the additional condition
$\mu+\kappa<1$ is often imposed to deduce the analogous estimate to \eqref{duha-str-est}.
The condition is, however, improved as above by splitting the integral at $\sigma=\frac{s+t}{2}$, see
\eqref{repre-C1}--\eqref{repre-Lv-alt}.
This finding is not new; in fact, it was already found for the autonomous case in, for instance, \cite[Lemma 3.5]{Hi91}.
Even though \eqref{ass3-force} is replaced by the local H\"older continuity on $(s,{\mathcal T}]$
without any behavior near $\tau=s$, the regularity \eqref{duha-cl}
is still available, but the behavior \eqref{duha-str-est} for $\partial_tv(t)$ near the initial time is lost.
See also Remark \ref{rem-ns} below on this issue within the nonlinear context.
\label{rem-yagi}
\end{remark}
\begin{remark}
In the integrand of \eqref{duha}, as mentioned before, $T(t,\sigma)$ is regarded as the evolution operator 
not on $L^q_{2,\sigma}(D)$, see \eqref{sole-chara-wei}, but on $L^q_\sigma(D)$; thus, \eqref{ass2-force} makes sense
even for $q\leq 3$.
In many situations, nevertheless, we would have the forcing term of the form $g(t)=Ph(t)$ with $h(t)\in L^q_2(D)$ 
when applying Theorem \ref{duha-strong}.
In this case the condition $q\in (3,\infty)$ is actually needed through \eqref{FK-wei}.
But the condition $q\in (3,\infty)$ is never optimal for obtaining \eqref{duha-cl}--\eqref{duha-str-est}
under the assumptions \eqref{ass2-force}--\eqref{ass3-force}.
See also Remark \ref{rem-duha-alt} for further discussions.
\label{rem-duha0}
\end{remark}
\begin{remark}
We often face the H\"older estimate of the form
\[
\|g(t)-g(\tau)\|_q\leq \sum_{i=1}^l C_i(t-\tau)^{\mu_i}(\tau-s)^{-\kappa-\mu_i}
\]
for $s<\tau<t\leq {\mathcal T}$ with some constants $C_i\;(i=1,\cdots,l)$ when applying Theorem \ref{duha-strong}.
In fact, it is indeed the case in section \ref{sect-ns}.
Nevertheless, we may restrict ourselves to the case $\frac{s+t}{2}<\tau<t\leq {\mathcal T}$, for which the condition above
implies \eqref{ass3-force} with 
$\displaystyle{\mu=\min_{1\leq i\leq l}\mu_i}$.
\label{rem-hoe-multi}
\end{remark}
\begin{remark}
It is not clear whether $\nabla^2 v\in C((s,{\mathcal T}];\, L^q(D))$
in Theorem \ref{duha-strong} unless the equation is autonomous.
In view of \eqref{wei-hoe-wh}
below for the whole space problem,
we would expect the H\"older estimate of $\rho\nabla T(\cdot,s)f$ with values in $L^q(D)$ for $f\in L^q_{2,\sigma}(D)$
as long as $q\in (3,\infty)$,
which together with $L(\cdot)v\in C((s,{\mathcal T}];\, L^q_\sigma(D))$, see \eqref{duha-cl},
leads to $Av\in C((s,{\mathcal T}];\, L^q_\sigma(D))$ and thereby
$v\in C((s,{\mathcal T}];\, W^{2,q}(D))$.
There are some technical details to be filled in and to be discussed elsewhere.
\label{rem-duha1}
\end{remark}
\begin{remark}
In \cite[Lemmas 3.1 and 3.2]{Hi22} a modest regularity of the function $v(t)$ given by \eqref{duha}
was discussed without introducing any weight for the forcing term $g(t)$.
In fact, we have
$v\in C^1_w((s,{\mathcal T});\, W^{-1,q}(D_R))$
provided that $g\in BC_w((s,{\mathcal T});\, L^q_\sigma(D))$, $1<q<\infty$.
\label{rem-duha-sofar}
\end{remark}

Let us proceed to study the Navier-Stokes initial value problem \eqref{ns}.
In order to get a strong solution subject to estimate \eqref{ns-est2} 
below for $\partial_tv(t)$ near the initial time, see 
Remark \ref{rem-ns} as well as the latter half of Remark \ref{rem-yagi},
one has to make a bit further assumption as well as 
\eqref{ass3-rigid}:
\begin{equation}
|\eta^\prime(t)-\eta^\prime(\tau)|+|\omega^\prime(t)-\omega^\prime(\tau)|
\leq C(t-\tau)^\vartheta \tau^{-\gamma-\vartheta}\qquad \Big(\frac{t}{2}<\tau<t\leq {\mathcal T}\Big)
\label{ass4-rigid}
\end{equation}
with some $\vartheta\in (0,1]$, $\gamma\in [0,1)$
and $C=C({\mathcal T})>0$, where ${\mathcal T}\in (0,\infty)$ is arbitrary.
Here, $\gamma$ is the same as in \eqref{ass3-rigid}, but
the condition $\gamma+\vartheta<1$ is not needed as described in the first half of Remark \ref{rem-yagi}.
As is readily seen, \eqref{ass3-rigid} ensures that $\eta$ and $\omega$ are well-defined up to $t=0$ and that
$\eta,\,\omega\in C_{\rm loc}^{0,1-\gamma}([0,\infty);\, \mathbb R^3)$ with
\begin{equation}
|(\eta,\omega)|_{1-\gamma,1}\leq \frac{1}{1-\gamma}\sup_{0<t\leq 1}t^\gamma\big(|\eta^\prime(t)|+|\omega^\prime(t)|\big),
\label{imp-hoe}
\end{equation}
see \eqref{quan1} with, say, ${\mathcal T}=1$, for we intend to construct the Navier-Stokes flow on some interval
$[0,{\mathcal T}_0]$ with ${\mathcal T}_0\leq 1$.
By following \eqref{ass4-rigid}--\eqref{imp-hoe},
let us introduce the quantities
\begin{equation}
\begin{split}
&m_0:=|(\eta,\omega)|_{0,1}+\sup_{0<t\leq 1}t^\gamma\big(|\eta^\prime(t)|+|\omega^\prime(t)|\big),  \\
&m_1:=\sup_{\frac{t}{2}<\tau<t\leq 1}
\frac{\tau^{\gamma+\vartheta}\big(|\eta^\prime(t)-\eta^\prime(\tau)|+|\omega^\prime(t)-\omega^\prime(\tau)|\big)}{(t-\tau)^{\vartheta}}
+m_0+m_0^2,
\end{split}
\label{quan2}
\end{equation}
where the supremum is taken with respect to $t\in (0,1]$ and $\tau\in (\frac{t}{2},t)$ in the latter.

Having \eqref{obsta} in mind,
we fix a cut-off function $\phi\in C_0^\infty(B_2)$ such that $\phi=1$ in $B_1$ and set
\begin{equation}
\begin{split}
b(x,t)&=\frac{1}{2}\,\mbox{rot $\Big\{\phi(x)\big(\eta(t)\times x-|x|^2\omega(t)\big)\Big\}$}  \\
&=\phi(x)\big(\eta(t)+\omega(t)\times x\big)
+\nabla\phi(x)\times \big(\eta(t)\times x-|x|^2\omega(t)\big).
\end{split}
\label{lift}
\end{equation}
Then it is indeed a solenoidal lift of the rigid motion
\begin{equation}
\begin{split}
&b|_{\partial D}=\eta+\omega\times x, \qquad \mbox{div $b$}=0, \qquad b(t)\in C_0^\infty(B_2),  \\
&b\in C_{\rm loc}^{0,1-\gamma}([0,\infty);\,W^{2,q}(\mathbb R^3)), \qquad
\partial_tb\in C_{\rm loc}^{0,\vartheta}((0,\infty);\, L^q(\mathbb R^3))
\end{split}
\label{lift-prop}
\end{equation}
for every $q\in [1,\infty]$
subject to
\begin{equation}
\sup_{0\leq t\leq 1}
\|b(t)\|_{W^{2,q}(\mathbb R^3)}
+\sup_{0\leq \tau<t\leq 1}\frac{\|b(t)-b(\tau)\|_{W^{2,q}(\mathbb R^3)}}{(t-\tau)^{1-\gamma}}\leq Cm_0.
\label{lift-est}
\end{equation}
As to estimates of $\partial_tb(t)$ for later use, see \eqref{force-est} below.

We look for a solution to \eqref{ns} of the form
\[
u(x,t)=b(x,t)+v(x,t),
\]
then $v(t)$ should obey
\begin{equation}
\begin{split}
&\frac{dv}{dt}+L(t)v+P(v\cdot\nabla v+b\cdot\nabla v+v\cdot\nabla b)=PF(t), \qquad t\in (0,\infty),  \\
&v(0)=v_0:=u_0-b(0),
\end{split}
\label{ns-evo}
\end{equation}
where
\begin{equation}
F(t):=F_0(t)-\partial_tb, \qquad
F_0(t):=
\Delta b+(\eta+\omega\times x)\cdot\nabla b-\omega\times b-b\cdot\nabla b,
\label{lift-force}
\end{equation}
which satisfy the following by \eqref{FK-wei}, \eqref{ass4-rigid}--\eqref{lift} and \eqref{lift-est}:
\begin{equation}
\begin{split}
&\sup_{0<t\leq 1}t^\gamma\|\rho^\alpha P\partial_tb(t)\|_q\leq Cm_0, \qquad
\sup_{\frac{t}{2}<\tau<t\leq 1}\frac{\tau^{\gamma+\vartheta}\|P\partial_tb(t)-P\partial_\tau b(\tau)\|_{r}}{(t-\tau)^\vartheta}
\leq Cm_1, \\
&\sup_{0\leq t\leq 1}\|\rho^\alpha PF_0(t)\|_q 
+\sup_{0\leq \tau<t\leq 1}\frac{\|PF_0(t)-PF_0(\tau)\|_r}{(t-\tau)^{1-\gamma}}\leq C(m_0+m_0^2)
\end{split}
\label{force-est}
\end{equation}
for every $\alpha\in [0,3)$, $q\in (\frac{3}{3-\alpha},\infty)$ and $r\in (1,\infty)$.
By virtue of the item 3 of Proposition \ref{so-far}, the problem \eqref{ns-evo} is converted into the integral equation
\begin{equation}
v(t)=T(t,0)v_0+\int_0^t T(t,s)P\big(F-v\cdot\nabla v-b\cdot\nabla v-v\cdot\nabla b\big)(s)\,ds.
\label{ns-int}
\end{equation}

We are in a position to give our main result.
We notice by \eqref{sole-chara-wei} that
 the following condition on the initial perturbation $v_0$ from the lift \eqref{lift} is exactly rephrased as \eqref{ass-IC} in terms of
the initial velocity $u_0$ for \eqref{ns}. 
\begin{theorem}
Suppose that $\eta$ and $\omega$ fulfill \eqref{ass3-rigid} and \eqref{ass4-rigid} for some $\vartheta\in (0,1]$ and
$\gamma\in [0,1)$. 
Let $q\in (3,\infty)$.
For every $v_0\in L^q_{1,\sigma}(D)$, there is ${\mathcal T}_0 \in (0,1]$, dependent on $\|\rho v_0\|_q$ and $m_0$
(as well as $q,\,\gamma$ and $D$), see \eqref{quan2},
such that problem \eqref{ns-evo} admits a solution $v(t)$ on the interval $[0,{\mathcal T}_0]$, which is of class
\begin{equation}
\begin{split}
&v\in C([0,{\mathcal T}_0];\, L^q_\sigma(D))\cap C^1((0,{\mathcal T}_0];\, L^q_\sigma(D)), \\
&v(t)\in Y_q(D) \quad\forall\, t\in (0,{\mathcal T}_0],  \qquad
L(\cdot)v\in C((0,{\mathcal T}_0];\, L^q_\sigma(D))
\end{split}
\label{ns-cl}
\end{equation}
and enjoys
\begin{equation}
\sup_{0<t\leq {\mathcal T}_0}
\left(t^{(3/q-3/r)/2}\|\rho v(t)\|_r+t^{1/2+(3/q-3/r)/2}\|\rho\nabla v(t)\|_r\right)
\leq C\big(\|\rho v_0\|_q+1\big), 
\label{ns-est0}
\end{equation}
\begin{equation}
\lim_{t\to 0}\,t^{j/2+(3/q-3/r)/2}\|\rho\nabla^j v(t)\|_r=0, \quad j\in \{0,1\},\quad (j,r)\neq (0,q)
\label{ns-est1}
\end{equation}
for every $r\in [q,\infty]$ with some $C=C(m_0,q,r,\gamma,D)>0$,
\begin{equation}
\sup_{0<t\leq {\mathcal T}_0}\left(t\|\partial_tv(t)\|_q+t\|\nabla^2 v(t)\|_q\right)
\leq C\big(\|\rho v_0\|_q+m_1\big)
\label{ns-est2}
\end{equation}
with some $C=C(m_0,q,\vartheta,\gamma,D)>0$ as well as
the initial condition
\begin{equation}
\lim_{t\to 0}\|\rho\big(v(t)-v_0\big)\|_q=0,
\label{ns-ic}
\end{equation}
where the weight function $\rho$ is given by \eqref{weight}. 

Moreover, the solution $v(t)$ obtained above is unique within the class
\begin{equation}
v\in L^\infty(0,{\mathcal T}_0;\, L^q_\sigma(D)), \qquad
t^{1/2}\nabla v\in L^\infty(0,{\mathcal T}_0;\, L^q(D))
\label{cl-uni}
\end{equation}
satisfying \eqref{ns-int} in $L^q_\sigma(D)$.
\label{ns-thm}
\end{theorem}
\begin{remark}
As described in the introductory section,
solely with the assumption \eqref{ass3-rigid}, a strong $L^q$-solution is still available subject to the properties in Theorem \ref{ns-thm}
except \eqref{ns-est2}, because the assumption 
$\eta^\prime,\,\omega^\prime\in C^{0,\vartheta}_{\rm loc}([\varepsilon,\infty);\,\mathbb R^3)$
allows us to show that the mild solution becomes a strong one $v\in C^1((\varepsilon,{\mathcal T}_0];\, L^q_\sigma(D))$
for every $\varepsilon >0$, however, the behavior of $\partial_tv(t)$ as $t\to 0$ is lost.
See the latter half of Remark \ref{rem-yagi} at the linearized stage.
In Theorem \ref{ns-thm} we have made further assumption \eqref{ass4-rigid} to find 
the desired estimate \eqref{ns-est2} as an application of
Theorem \ref{duha-strong}.
\label{rem-ns}
\end{remark}
\begin{remark}
It seems unlikely to get a strong $L^3$-solution when $v_0\in L^3_{1,\sigma}(D)$ since \eqref{FK-wei}
with $\alpha=2$ does not hold for $q=3$, see \eqref{ass2-force}.
The assumption \eqref{ass2-force} can be replaced by \eqref{ass1-force} in Theorem \ref{duha-strong} for the autonomous case
and, therefore, a strong $L^3$-solution is available for every $v_0\in L^3_{1,\sigma}(D)$
by considering the evolution for $t>\varepsilon$ as
in Remark \ref{rem-ns} when both $\eta$ and $\omega$ are constant vectors;
however, one cannot deduce \eqref{ns-est2} on account of \eqref{ass1-force} for $g=P(v\cdot\nabla v)$ merely with $\kappa=1$
when $q=3$.
That is also the case for the usual Navier-Stokes initial value problem with $\eta=\omega=0$. 
\label{rem-ns1}
\end{remark}
\begin{remark}
As a by-product of the weighted analysis,
it would be interesting to observe that \eqref{ns-est0} with $r=\infty$ exhibits a spatial-temporal behavior of the Navier-Stokes flow.
The other thing is that the latter part 
of the right-hand side of \eqref{ns-est0} comes from the force \eqref{lift-force}, 
which is controlled as \eqref{force-est}, 
but the H\"older estimate of $P\partial_tb(t)$ is not needed here.
Since $m_0$ is already involved in the constant $C$ through several estimates of the evolution operator $T(t,s)$,
the bound $\|\rho v_0\|_q+m_0+m_0^2$ may be written as $\|\rho v_0\|_q+1$.
\label{rem-ns2}
\end{remark}
\begin{remark}
The class \eqref{cl-uni} together with \eqref{ns-int} implies the initial condition
$\displaystyle{\lim_{t\to 0}\|v(t)-v_0\|_q=0}$.
Likewise, \eqref{ns-est0} and \eqref{ns-int} lead to \eqref{ns-ic} by virtue of \eqref{IC-1}.
\label{rem-ns3}
\end{remark}

Let us close this section with two remarks about the descriptions in the rest of the paper.

Firstly, we sometimes use estimates of the evolution operator $T(t,s)$
obtained by the second author \cite{Hi18, Hi20}, all of which were deduced uniformly in
$(t,s)$ with $0<t-s\leq {\mathcal T}$
under the condition \eqref{ass2-rigid} in those literature;
otherwise, one cannot discuss the issue
of \cite{Hi18, Hi20}, that is, the large time behavior.
It readily turns out without any change of the proof that the same estimates hold true for $0\leq s<t\leq {\mathcal T}$ under less
assumption \eqref{ass1-rigid}. 

Secondly, since the present paper is concerned with the regularity issue,
every estimate in the proof of propositions and lemmas
holds for all $(t,s)$ with $0\leq s<t\leq {\mathcal T}$, 
where ${\mathcal T}\in (0,\infty)$ is fixed
arbitrarily, although this will not be always mentioned explicitly.

\section{Some auxiliary results}
\label{sect-auxi}

In this section we collect some preparatory lemmas for later use.
Let us begin with the boundary value problem for the divergence equation
\begin{equation}
\mbox{div $w$}=f\quad\mbox{in $G$}, \qquad w|_{\partial G}=0,
\label{bvp-div}
\end{equation}
where $G$ is a bounded domain in $\mathbb R^n$.
This admits many solutions as long as $f$ possesses an appropriate regularity with the compatibility condition
$\int_Gf\,dx=0$.
Among them, a paticular solutuion discovered by Bogovskii \cite{Bog} 
is convenient to recover the solenoidal condition in cut-off
procedures on account of several fine properties of his solution.
The result needed in this paper is summarized in the following lemma, see \cite[Theorem 1]{Bog}, \cite[Theorem 2.4 (a)--(c)]{BorSo90}, 
\cite[Theorem III.3.3]{Ga-b}, \cite[Theorem 2.5]{GHH06-2} and the references therein.
\begin{lemma}
Let $G$ be a bounded domain in $\mathbb R^n$, $n\geq 2$, with Lipschitz boundary $\partial G$.
There exists a linear operator $\mathbb B_G: C_0^\infty(G)\to C_0^\infty(G)^n$ with the following properties:
For every $q\in (1,\infty)$ and integer $k\geq 0$, there is a constant $C=C(q,k,G)>0$ being invariant under dilation
of the domain $G$ such that
\begin{equation}
\|\nabla^{k+1}\mathbb B_Gf\|_{q,G}\leq C\|\nabla^kf\|_{q,G}
\label{bog-est1}
\end{equation}
and that $w=\mathbb B_Gf$ is a solution to \eqref{bvp-div} if $\int_Gf\,dx=0$.
The operator $\mathbb B_G$, that we call the Bogovskii operator, 
extends uniquely to a bounded operator from $W_0^{k,q}(G)$ to $W_0^{k+1,q}(G)^n$.
Furthermore, it also extends uniquely to a bounded operator from $W^{1,q^\prime}(G)^*$ to $L^q(G)^n$, namely,
\begin{equation}
\|\mathbb B_Gf\|_{q,G}\leq C\|f\|_{W^{1,q^\prime}(G)^*}
\label{bog-est2}
\end{equation}
with some constant $C=C(q,G)>0$, where $1/q^\prime+1/q=1$.  
\label{lem-bog}
\end{lemma}

Let $D\subset \mathbb R^3$ be the exterior domain under consideration with $C^{1,1}$-boundary such that 
\eqref{obsta} is fulfilled. 
We fix $R>r>1$ and take a cut-off function $\phi\in C_0^\infty(B_R)$ satisfying $\phi=1$ in $B_r$.
Let $1<q<\infty$.
Given $f\in L^q_\sigma(D)$, consider the functions
\begin{equation}
f_0=(1-\phi)f+\mathbb B_G[f\cdot\nabla\phi], \qquad
f_1=\phi f-\mathbb B_G[f\cdot\nabla\phi]
\label{cutoff}
\end{equation}
with $G=B_{R}\setminus \overline{B_1}$, where $\mathbb B_G[f\cdot\nabla\phi]$ is understood as its extension by 
setting zero outside $G$.
Since $\int_Gf\cdot\nabla\phi\,dx=0$, we have $f_0\in L^q_\sigma(\mathbb R^3)$ and $f_1\in L^q_\sigma(D_{R})$ by Lemma \ref{lem-bog}.
If, in particular, $f\in D_q(A)$, see \eqref{stokes}, then it is also obvious by Lemma \ref{lem-bog} that
$f_0\in W^{2,q}(\mathbb R^3)\cap L^q_\sigma(\mathbb R^3)$ and that $f_1\in D_q(A_{R})$, where $A_{R}$ denotes the 
Stokes operator on the bounded domain $D_{R}$:
\begin{equation}
D_q(A_R)=L^q_\sigma(D_R)\cap W^{1,q}_0(D_R)\cap W^{2,q}(D_R), \qquad A_Ru=-P_{D_R}\Delta u.
\label{st-bdd}
\end{equation}
Now the question is whether both $f_0$ and $f_1$ inherit a fractional regularity with exact order from $f$.
Lemma \ref{lem-frac} below gives an affirmative answer in terms of the fractional powers of the Stokes operators $A$ and $A_R$ as well as 
the Bessel potential space $H^{2\delta}_q(\mathbb R^3)=\big[L^q(\mathbb R^3), W^{2,q}(\mathbb R^3)\big]_\delta$
with $\delta\in (0,1)$.
Here and in what follows, $[\cdot,\cdot]_\delta$ stands for the complex interpolation functor.
Note that the fractional powers $A^\delta$ and $A_R^\delta$ are well defined since both Stokes operators are generators of
bounded analytic semigroups, see \cite{BorSo87, FS94, Gi}.
Moreover, their domains coincide with the complex interpolation spaces, respectively, through analysis of
pure imaginary powers due to \cite{Gi85, GiSo89}, where $0<\delta<1$:
\begin{equation}
D_q(A^\delta)=\big[L^q_\sigma(D), D_q(A)\big]_\delta, \qquad
D_q(A_R^\delta)=\big[L^q_\sigma(D_R), D_q(A_R)\big]_\delta. 
\label{frac-ci}
\end{equation}
The following lemma was found by \cite[Lemma 4.8]{Hi99}, that is based on \cite{Hi00}, 
when $q=2$.
Since the only point is indeed \eqref{frac-ci} together with the definition of the complex interpolation, 
the proof for the general case $q\in (1,\infty)$ is essentially the same
as in \cite[Lemma 7.2]{Hi00} (rather than \cite{Hi99} because this literature relies on the Heinz-Kato inequality)
and may be omitted.
The result will be employed in section \ref{sect-strong}.
\begin{lemma}
Let $q\in (1,\infty)$ and suppose that $f\in D_q(A^\delta)$ with $\delta\in (0,1)$.
Then the functions $f_0$ and $f_1$ defined by \eqref{cutoff} respectively satisfy
\[
f_0\in H^{2\delta}_q(\mathbb R^3)\cap L^q_\sigma(\mathbb R^3), 
\qquad
f_1\in D_q(A_R^\delta)
\]
subject to
\begin{equation}
\|f_0\|_{H^{2\delta}_q(\mathbb R^3)}\leq C_1\|f\|_{D_q(A^\delta)}, \qquad
\|A_R^\delta f_1\|_{q,D_R}\leq C_2\|f\|_{D_q(A^\delta)}
\label{est-data-0}
\end{equation}
with some constants $C_1=C_1(q,\delta,D)>0$ and $C_2=C_2(q,\delta,R,D)>0$.
\label{lem-frac}
\end{lemma}

$W^{2,q}$-estimate over $D_R$ 
in terms of the Stokes operator $A$ 
rather than $A_R$ is useful, see section \ref{sect-duha},
for deduction of estimate of $\|\nabla^2v(t)\|_{q,D_R}$ from \eqref{Av-local}.
This is the issue of the following lemma.
\begin{lemma}
Let $A$ be the Stokes operator on the exterior domain $D$
with $C^{1,1}$-boundary $\partial D$, see \eqref{stokes},
where \eqref{obsta} is fulfilled. 
Let $q\in (1,\infty)$ and $R\in (1,\infty)$, then there is a constant $C=C(q,R,D)$ such that
\begin{equation}
\|u\|_{W^{2,q}(D_R)}
\leq C\big(\|Au\|_{q,D_{2R}}+\|\nabla u\|_{q,G_R}+\|u\|_{q,G_R}\big)
\label{St-loc-est-0}
\end{equation}
for all $u\in D_q(A)$, where $G_R=D_{2R}\setminus \overline{D_R}$.
\label{St-local-0}
\end{lemma}

\begin{proof}
Given $u\in D_q(A)$, we set $h=Au$.
Then there is a pressure $p$ such that
\[
-\Delta u+\nabla p=h, \qquad\mbox{div $u$}=0\quad\mbox{in $D$};\quad
u|_{\partial D}=0
\]
where $p$ is singled out in such a way that
$\int_{G_R}p\,dx=0$, yielding
\begin{equation}
\|p\|_{q,G_R}\leq C\|\nabla p\|_{W^{-1,q}(G_R)}
\leq C\big(\|h\|_{q,G_R}+\|\nabla u\|_{q,G_R}\big).
\label{p-est}
\end{equation}
Fix a cut-off function $\phi\in C_0^\infty(B_{2R})$ 
satisfying $\phi=1$ on $B_R$, and consider
$v=\phi u-\mathbb B[u\cdot\nabla\phi]$ and $p_v=\phi p$
with $\mathbb B=\mathbb B_{G_R}$ being the Bogovskii operator defined in Lemma \ref{lem-bog}, 
which should obey
\[
-\Delta v+\nabla p_v=f, \qquad\mbox{div $v$}=0\quad\mbox{in $D_{2R}$};\quad
v|_{\partial D_{2R}}=0
\]
with
\[
f=\phi h-2\nabla\phi\cdot\nabla u-(\Delta \phi)u
+\Delta\mathbb B[u\cdot\nabla\phi]+(\nabla\phi)p
\]
fulfilling
\[
\|f\|_{q,D_{2R}}
\leq C\big(\|h\|_{q,D_{2R}}+\|\nabla u\|_{q,G_R}+\|u\|_{q,G_R}\big)
\]
by taking into account \eqref{p-est} as well as \eqref{bog-est1}.
This combined with estimate of $v$ in $W^{2,q}(D_{2R})$ in terms of
$f$ leads to \eqref{St-loc-est-0} since $u=v$ on $D_R$.
\end{proof}

Let us consider the autonomous regime in which $\eta$ and $\omega$ are constant vectors, so that the linear operator $L(t)=L$
given by \eqref{ig} is independent of $t$.
Then the evolution operator $T(t,s)=e^{-(t-s)L}$ is the semigroup, whose generation was proved first by
Geissert, Heck and Hieber \cite{GHH06} and then by Shibata \cite{Shi10}, where the methodology in those papers are 
completely different.
The regularity estimate
\begin{equation}
\|u\|_{W^{2,q}(D)}\leq C\big(\|Lu\|_{q}+\|u\|_{q}\big)
\label{ell-L}
\end{equation}
will be
employed in section \ref{sect-duha}
to show the continuity of $t\mapsto v(t)\in W^{2,q}(D)$ for \eqref{duha} with the semigroup.
When $q=2$ and $\eta=0$, \eqref{ell-L} was proved by the second author \cite{Hi99-0}. 
Later on, 
the proof of \eqref{ell-L} for $q\in (1,\infty)$ 
was given by Shibata \cite[Appendix]{Shi10} when $\eta$ and $\omega$ are parallel to each other.
Indeed it is possible to reduce the general case to his situation, but we will give a somewhat different proof of \eqref{ell-L} 
for completeness.
Given $u\in D_q(L)$, we set $f=(1+L)u$, which corresponds to the boundary value problem
\begin{equation}
\begin{split}
&u+{\mathcal L}u+\nabla p=f, \qquad\mbox{div $u$}=0 \quad\mbox{in $D$}, \\
&u|_{\partial D}=0, \qquad \lim_{|x|\to \infty}u=0,
\end{split}
\label{bvp-resol}
\end{equation}
where ${\mathcal L}$ is given by \eqref{pre-proj-0} being independent of $t$.
To see \eqref{ell-L},
it suffices to prove
\begin{equation}
\|u\|_{W^{2,q}(D)}+\|\nabla p\|_q
\leq C\|f\|_q
\label{apri-ext}
\end{equation}
for the solution to \eqref{bvp-resol} of the class \eqref{cl-resol} below, where $f$ is not necessarily solenoidal.
Note that \eqref{apri-ext} implies the estimate 
$\|(\omega\times x)\cdot\nabla u\|_q\leq C\|f\|_q$ as well from the equation.
It follows from the uniformly boundedness of the semigroup $e^{-tL}$ on $L^q_\sigma(D)$
obtained by \cite{Shi08} and then by \cite{Hi18} as well that
$1+L$ is invertible 
\begin{equation}
\|u\|_q=\|(1+L)^{-1}Pf\|_q\leq 
\int_0^\infty e^{-t}\|e^{-tL}Pf\|_q\,dt
\leq C\|f\|_q,
\label{inv}
\end{equation}
leading to the uniqueness of solutions to \eqref{bvp-resol} within the class \eqref{cl-resol}.
Notice that the constant $C=C(m,q,D)>0$ in \eqref{inv} is taken uniformly in $(\eta,\omega)$ with
$|\eta|+|\omega|\leq m$, because so is the constant $c_0=c_0(m,q,D)>0$ for
$\|e^{-tL}Pf\|_q\leq c_0\|f\|_q$.
The whole thing is thus the $L^q$-estimate of $\nabla^2 u$ for the equation
${\mathcal L}u+\nabla p=f$ in $\mathbb R^3$ due to \cite{Fa-tohoku, FHM, GK}, and \eqref{inv} makes the issue easier.
We show the following lemma.
\begin{lemma}
Suppose that $\eta$ and $\omega$ are constant vectors.
Let $q\in (1,\infty)$ and $m\in (0,\infty)$. 
Then the unique solution $u=(1+L)^{-1}Pf$ to the problem \eqref{bvp-resol} in the class, up to constants for the pressure,
\begin{equation}
u\in D_q(L), \qquad p\in L^q_{\rm loc}(\overline{D})\;\;\mbox{with $\nabla p\in L^q(D)$}
\label{cl-resol}
\end{equation}
enjoys \eqref{apri-ext} for all $f\in L^q(D)$ with some constant
$C=C(m,q,D)>0$ whenever $|\eta|+|\omega|\leq m$.
As a consequence, we have \eqref{ell-L} for all $u\in D_q(L)$ with some constant $C=C(m,q,D)>0$.
\label{lem-apri}
\end{lemma}
 
\begin{proof}
We fix $R>1$, set $G=D_{2R}\setminus\overline{D_R}$ and single out the pressure $p$ associated with $u=(1+L)^{-1}Pf$
in such a way that $\int_{G}p\,dx=0$.
Then, as in \eqref{p-est}, via the first inequality there, we get
\begin{equation}
\|p\|_{q,G}\leq C\big(\|f\|_{q,G}+\|\nabla u\|_{q,G}+\|u\|_{q,G}\big)
\label{p-est2}
\end{equation}
from the equation.
Fix also the same cut-off function $\phi$ as in the proof of Lemma \ref{St-local-0}, and consider
\begin{equation*}
\begin{split}
&v=(1-\phi)u+\mathbb B[u\cdot\nabla \phi], \qquad p_v=(1-\phi)p, \\
&w=\phi u-\mathbb B[u\cdot\nabla\phi], \qquad p_w=\phi p,
\end{split}
\end{equation*}
where $\mathbb B=\mathbb B_G$.
Then they obey, respectively,
\begin{equation}
v+{\mathcal L}v+\nabla p_v=f_v, \qquad\mbox{div $v$}=0\quad\mbox{in $\mathbb R^3$}; \quad
\lim_{|x|\to\infty}v=0,
\label{wh-resol}
\end{equation}
\begin{equation}
w+{\mathcal L}w+\nabla p_w=f_w, \qquad \mbox{div $w$}=0\quad\mbox{in $D_{2R}$}; \quad
w|_{D_{2R}}=0,
\label{int-resol}
\end{equation}
with
\begin{equation*}
\begin{split}
&f_v=(1-\phi)f+g, \qquad
f_w=\phi f-g, \\
&g=2\nabla\phi\cdot\nabla u+\big[\Delta\phi+(\eta+\omega\times x)\cdot\nabla\phi\big]u
+(1+{\mathcal L})\mathbb B[u\cdot\nabla\phi]-(\nabla\phi)p.
\end{split}
\end{equation*}
It follows from 
\eqref{bog-est1},
\eqref{inv}, \eqref{p-est2}
and an interpolation inequality for $\|\nabla u\|_{q,G}$ that
\begin{equation}
\begin{split}
\|f_v\|_{q,\mathbb R^3}+\|f_w\|_{q,D_{2R}}
&\leq C\big(\|f\|_q+\|\nabla u\|_{q,G}+\|u\|_{q,G}\big)  \\
&\leq C_\varepsilon \|f\|_q+\varepsilon \|\nabla^2u\|_q
\end{split}
\label{cut-force}
\end{equation}
for every $\varepsilon>0$.

Let us recall some results on the $L^q$-estimates for the whole space problem \eqref{wh-resol}.
For the specific case $\omega=0$, it is just the classical Oseen (or even Stokes for $\eta=0$) problem, 
see \cite[Theorem VII.4.1]{Ga-b}.
It is conveninent to discuss the other case $\omega\neq 0$ through the Mozzi-Chasles transform \cite[Chapter VIII]{Ga-b};
then the problem \eqref{wh-resol} is reduced to the one with $\eta=0$ if $\omega\cdot\eta=0$, otherwise it is reduced to the case
when $\eta$ and $\omega$ are parallel to each other.
According to \cite{FHM} for the former and to \cite{Fa-tohoku,GK} for the latter, we have
\begin{equation}
\|\nabla^2 v\|_{q,\mathbb R^3}+\|\nabla p_v\|_{q,\mathbb R^3}\leq C\|f_v-v\|_{q,\mathbb R^3}\leq C\|f_v\|_{q,\mathbb R^3}
\label{apri-wh}
\end{equation}
with some constant $C>0$ independent of $(\eta,\omega)$
since the $L^q$-norm is invariant under the aforementioned transform which is just a translation 
$x\mapsto \widetilde x=x-\frac{\omega\times \eta}{|\omega|^2}$
and since
$\|v\|_{q,\mathbb R^3}\leq C\|f_v\|_{q,\mathbb R^3}$ by the same reasoning as in \eqref{inv}, 
where the associated semigroup is a specific case of the evolution operator \eqref{rep-wh}
in the next section.

As for the interior problem \eqref{int-resol}, we employ the $L^q$-estimate for the Stokes resolvent system 
(with the resolvent parameter $\lambda=1$), see Farwig and Sohr \cite[Theorem 1.2]{FS94}, to furnish
\begin{equation}
\begin{split}
\|\nabla^2w\|_{q,D_{2R}}+\|\nabla p_w\|_{q,D_{2R}}
&\leq C\|f_w+(\eta+\omega\times x)\cdot\nabla w-\omega\times w\|_{q,D_{2R}}  \\
&\leq C\|f_w\|_{q,D_{2R}}+\frac{1}{2}\|\nabla^2 w\|_{q,D_{2R}}
\end{split}
\label{apri-int}
\end{equation}
where the latter inequality follows from the same reason as in \eqref{cut-force}.

Since $u=v+w$ and $p=p_v+p_w$, we collect \eqref{cut-force}--\eqref{apri-int} 
and then take $\varepsilon>0$ small enough to conclude
\[
\|\nabla^2u\|_q+\|\nabla p\|_q\leq C\|f\|_q,
\]
which together with \eqref{inv} and an interpolation inequality lead us to \eqref{apri-ext}.
The proof is complete.
\end{proof}

Let us turn to the non-autonomous system.
For the construction of the evolution operator in Proposition \ref{so-far}, the next lemma on iterated 
convolutions plays an important role in \cite{HR14}.
Our proof of Theorems \ref{weighted-est} and \ref{strong-sol} is still based on this methodology.
We refer to \cite[Lemma 4.6]{GHH06}, \cite[Lemma 3.3]{HR11} and \cite[Lemma 5.2]{HR14}.
Notice that the same idea was essentially found in the classical theory, see \cite[Chapter 5, Sections 2 and 3]{Ta}.
\begin{lemma}
[\cite{GHH06, HR11, HR14}]
Let $X_1$ and $X_2$ be two Banach spaces, and let ${\mathcal T}\in (0,\infty)$. 
Suppose that there are constants $\beta,\, \gamma\in [0,1)$ and $c_*>0$ with the following property:
The families
\[
\{E_0(t,s);\, 0\leq s<t\leq {\mathcal T}\}\subset {\mathcal L}(X_1,X_2), \quad
\{Q(t,s);\, 0\leq s<t\leq {\mathcal T}\}\subset {\mathcal L}(X_1)
\]
fulfill
\[
\|E_0(t,s)\|_{{\mathcal L}(X_1,X_2)}\leq c_*(t-s)^{-\beta}, \qquad
\|Q(t,s)\|_{{\mathcal L}(X_1)}\leq c_*(t-s)^{-\gamma}
\]
for all $(t,s)$ with $0\leq s<t\leq {\mathcal T}$ as well as the strongly measurability of
$E_0(t,s)f$ and $Q(t,s)f$ in $(t,s)$ for every $f\in X_1$.
For $f\in X_1$ and $0\leq s<t\leq {\mathcal T}$, define a sequence $\{E_j(t,s)f\}_{j=0}^\infty$ by
\begin{equation}
E_{j+1}(t,s)f=\int_s^t E_j(t,\tau)Q(\tau,s)f\,d\tau\qquad (j=0,1,2,\cdots).
\label{ite-int}
\end{equation}
Then 
\[
E(t,s)f:=\sum_{j=0}^\infty E_j(t,s)f \qquad \mbox{in $X_2$}
\]
converges absolutely and uniformly in $(t,s)$ with
$0\leq s<t\leq {\mathcal T}$ and $t-s\geq \varepsilon$ for every $\varepsilon \in (0,{\mathcal T})$.
Moreover, there is a constant $C=C({\mathcal T}, c_*, \beta, \gamma)>0$ such that
\[
\|E(t,s)f\|_{X_2}\leq \sum_{j=0}^\infty \|E_j(t,s)f\|_{X_2}\leq C(t-s)^{-\beta}\|f\|_{X_1}
\]
for all $(t,s)$ with $0\leq s<t\leq {\mathcal T}$ and $f\in X_1$.
If, in particular, $\beta=0$, then the convergence of the series above is uniform in $(t,s)$ with $0\leq s\leq t\leq {\mathcal T}$
as long as $E_0(t,t)\in {\mathcal L}(X_1,X_2)$ is well-defined.
\label{lem-ite}
\end{lemma}

We note that the integrand $E_j(t,\tau)Q(\tau,s)f$ of \eqref{ite-int} is strongly measurable in $(t,\tau,s)$ as a $X_2$-valued 
function for every $f\in X_1$ under the assumptions (\cite[Proposition 1.1.28]{hy-b}).

The following elementary fact on the uniformly convergence will be needed in the proof of Theorem \ref{duha-strong},
see \eqref{unif-conv}.
\begin{lemma}
Let $q\in (1,\infty)$ and let
$\{T(t,s)\}_{t\geq s\geq 0}$ be the evolution operator on $L^q_\sigma(D)$ obtained in Proposition \ref{so-far}.
Then 
\begin{equation}
\lim_{\varepsilon\to 0+}\sup_{t\in J}\|T(t,t-\varepsilon)g(t-\varepsilon)-g(t)\|_q=0
\label{con-Tg}
\end{equation}
for every compact interval $J\subset (s,\infty)$ and $g\in C((s,\infty);\, L^q_\sigma(D))$, where $s\geq 0$. 
\label{lem-con}
\end{lemma}

\begin{proof}
Let $\iota>0$ be arbitrarily small, and
fix a compact interval $J=[\tau_0,{\mathcal T}]\subset (s,\infty)$. 
One can take finitely many points $\{t_j\}_{j=1}^N\subset J$ 
as well as $\delta>0$ such that
\begin{equation}
\sup_{t\in J}\|g(t-\varepsilon)-g(t)\|_q\leq \iota \qquad 
\forall\,\varepsilon \in (0,\delta)
\label{con-0}
\end{equation}
and that,
for every $t\in J$, there is $k\in \{1,\cdots,N\}$ satisfying
\begin{equation}
\|g(t)-g(t_k)\|_q\leq \iota.
\label{con-1}
\end{equation}
For each $j\in \{1,\cdots,N\}$ we know 
from Proposition \ref{so-far}
the continuity of $(t,s)\mapsto T(t,s)g(t_j)$ uniformly on
$\{(t,s);\, 0\leq s\leq t\leq {\mathcal T}\}$. 
Hence there is $\varepsilon_j>0$ such that, if $0<\varepsilon<\varepsilon_j$, then
\begin{equation}
\sup_{t\in J}\|T(t,t-\varepsilon)g(t_j)-g(t_j)\|_q\leq \iota.
\label{con-2}
\end{equation}
Set $\varepsilon_*:=
\min\{\varepsilon_1,\cdots,\varepsilon_N,\delta\}$,
which is independent of $t\in J$, and let
$\varepsilon \in (0,\varepsilon_*)$.
Given $t\in J$, we take $k\in \{1,\cdots, N\}$ satisfying \eqref{con-1}, from which together with
\eqref{con-0}, \eqref{con-2} with $j=k$ and \eqref{sm} it follows that
\begin{equation*}
\begin{split}
&\quad \|T(t,t-\varepsilon)g(t-\varepsilon)-g(t)\|_q  \\
&\leq \|T(t,t-\varepsilon)\big( g(t-\varepsilon)-g(t)\big)\|_q
+\|T(t,t-\varepsilon)\big(g(t)-g(t_k)\big)\|_q  \\
&\quad +\|T(t,t-\varepsilon)g(t_k)-g(t_k)\|_q
+\|g(t_k)-g(t)\|_q  \\
&\leq C\iota,
\end{split}
\end{equation*}
yielding \eqref{con-Tg}.
The proof is complete.
\end{proof}

Let $1<q<\infty$ and $R>1$. We introduce the family of linear operators $\{L_R(t)\}_{t\geq 0}$ on $L^q_\sigma(D_R)$ by
\begin{equation}
\begin{split}
D_q(L_R(t))&=D_q(A_R), \\
L_R(t)u&=P_{D_R}{\mathcal L}(t)u=A_Ru-(\eta(t)+\omega(t)\times x)\cdot\nabla u+\omega(t)\times u,
\end{split}
\label{ig-bdd}
\end{equation}
where ${\mathcal L}(t)$ is the differential operator given by \eqref{pre-proj-0} and $A_R$ is the Stokes operator
defined by \eqref{st-bdd} on the bounded domain $D_R$. 
By the same reasoning as in \eqref{ig}, the Fujita-Kato projection $P_{D_R}$ is not needed in 
the lower order terms of \eqref{ig-bdd}.
For the interior problem, under the condition \eqref{ass1-rigid}, one can apply the general theory
of evolution operators of parabolic type, see Tanabe \cite[Chapter 5]{Ta}, to find that the family $\{L_R(t)\}_{t\geq 0}$ generates
an evolution operator $\{V(t,s)\}_{t\geq s\geq 0}$ on $L^q_\sigma(D_R)$.
For every $f\in L^q_\sigma(D_R)$, we know that $u(t)=V(t,s)f$ is of class
\begin{equation}
u\in C^1((s,\infty);\, L^q_\sigma(D_R))\cap C((s,\infty);\, D_q(A_R))\cap C([s,\infty);\, L^q_\sigma(D_R))
\label{reg-sob-bdd}
\end{equation}
and satisfies
\[
\frac{du}{dt}+L_R(t)u=0, \qquad t\in (s,\infty); \quad u(s)=f
\]
in $L^q_\sigma(D_R)$.
Most of results in the following lemma are found essentially in \cite{HR14, Hi18, Hi20}, while
\eqref{evo-semi-bdd-0} is well known in the theory of evolution operators of parabolic type.
It is possible to improve \eqref{evo-semi-bdd-0}
as in Remark \ref{rem-evo-semi}
below although this refinement does not play any role for our purpose.
For later use,
the item 3 of the following lemma provides us with estimates
for $f\in W^{1,q}(D_R)$ rather than $W^{1,q}_0(D_R)$. 
\begin{lemma}
Suppose that $\eta$ and $\omega$ fulfill \eqref{ass1-rigid} for some $\theta\in (0,1]$.
Let $q\in (1,\infty)$ and $R\in (1,\infty)$.

\begin{enumerate}
\item
Let $j\in \{0,1\}$ and $r\in [q,\infty]$.
For each ${\mathcal T}\in (0,\infty)$ and $m\in (0,\infty)$, there are constants
$C=C({\mathcal T},m,q,r,\theta,D_R)>0$ and $C^\prime=C^\prime({\mathcal T},m,q,\theta,D_R)>0$ such that
\begin{equation}
\|\nabla^j V(t,s)f\|_{r,D_R}\leq C(t-s)^{-j/2-(3/q-3/r)/2}\|f\|_{q,D_R}
\label{sm-bdd}
\end{equation}
\begin{equation}
\|p(t)\|_{q,D_R}+\|\partial_tV(t,s)f\|_{W^{-1,q}(D_R)}\leq C^\prime (t-s)^{-(1+1/q)/2}\|f\|_{q,D_R}
\label{pressure-bdd}
\end{equation}
\begin{equation}
\|L_R(t)\big(V(t,s)-e^{-(t-s)L_R(t)}\big)f\|_{q,D_R}
\leq C^\prime(t-s)^{-1+\theta}\|f\|_{q,D_R}
\label{evo-semi-bdd-0}
\end{equation}
for all $(t,s)$ with $0\leq s<t\leq {\mathcal T}$ and $f\in L^q_\sigma(D_R)$ whenever \eqref{ever} is satisfied,
where $p(t)$ is the pressure associated with $V(t,s)f$ and it is singled out subject to the side condition
$\int_{D_R}p(t)\,dx=0$.

\item
Let $\delta\in [0,1]$.
For each ${\mathcal T}\in (0,\infty)$ and $m\in (0,\infty)$,
there is a constant $C=C({\mathcal T},m,q,\delta,\theta,D_R)>0$ such that
\begin{equation}
\|V(t,s)f\|_{W^{2,q}(D_R)}+\|\partial_tV(t,s)f\|_{q,D_R}+\|\nabla p(t)\|_{q,D_R}
\leq C(t-s)^{-1+\delta}\|A_R^\delta f\|_{q,D_R}
\label{frac-bdd}
\end{equation}
for all $(t,s)$ with $0\leq s<t\leq {\mathcal T}$ and $f\in D_q(A_R^\delta)$
whenever \eqref{ever} is satisfied.

\item
Let $\delta\in (0,1/2q)$.
For each ${\mathcal T}\in (0,\infty)$ and $m\in (0,\infty)$, there is a constant 
$C=C({\mathcal T},m,q,\delta,\theta,D_R)>0$ such that
\begin{equation}
\|V(t,s)f\|_{W^{2,q}(D_R)}+\|\partial_tV(t,s)f\|_{q,D_R}+\|\nabla p(t)\|_{q,D_R}
\leq C(t-s)^{-1+\delta}\|f\|_{W^{1,q}(D_R)}
\label{high-bdd}
\end{equation}
\begin{equation}
\|V(t,s)f\|_{W^{1,q}(D_R)}\leq C(t-s)^{-1/2+\delta}\|f\|_{W^{1,q}(D_R)}
\label{1st-bdd}
\end{equation}
for all $(t,s)$ with $0\leq s<t\leq {\mathcal T}$ and $f\in L^q_\sigma(D_R)\cap W^{1,q}(D_R)$ whenever \eqref{ever} is satisfied.
\end{enumerate}
\label{lem-evo-bdd}
\end{lemma}

\begin{proof}
$L^q$-$L^r$ smoothing estimate \eqref{sm-bdd} was shown by \cite{HR14} when $r<\infty$,
while \eqref{pressure-bdd} was deduced by \cite[Lemma 3.2]{Hi18}, see also \cite[Lemma 4.1]{Hi20}.
Estimate \eqref{frac-bdd} 
is found in \cite[(4.12)]{Hi20}.
By use of \eqref{frac-bdd} with $\delta=0$ and $q\in (3,\infty)$, the Gagliardo-Nirenberg inequality leads to \eqref{sm-bdd}
for $r=\infty$ and such $q$, however, even for $q\leq 3$ by the semigroup property.
Estimates \eqref{high-bdd}--\eqref{1st-bdd} are found in \cite[Lemma 4.2]{Hi20}; here, $f$ does not satisfy the boundary condition
$f|_{\partial D_R}=0$ and this is the reason why those estimates do not exhibit sharp behavior near $t=s$.
The condition $\delta\in (0,1/2q)$ ensures that $D(A_R^\delta)$ does not involve any boundary condition
(except the vanishing normal trace)
at $\partial D_R$, see \cite{Fu} and \cite[Section 2.3]{NS03} as well as \eqref{frac-ci}.
Finally, we refer to, for instance, Friedman \cite[Part 2, Lemma 7.1]{Fri} for \eqref{evo-semi-bdd-0}.
\end{proof}

A result on the uniformly convergence corresponding to Lemma \ref{lem-con} is also needed for the semigroup below
rather than the evolution operator $V(t,s)$, see \eqref{unif-conv}.
\begin{lemma}
Suppose that $\eta$ and $\omega$ fulfill \eqref{ass1-rigid} for some $\theta\in (0,1]$.
Let $q\in (1,\infty)$ and $R\in (1,\infty)$. 
Then we have
\begin{equation}
\lim_{\varepsilon\to 0+}\sup_{t\in J}\|e^{-\varepsilon L_R(t)}g(t)-g(t)\|_{q,D_R}=0
\label{con-LRg}
\end{equation}
for every compact interval $J\subset (s,\infty)$ and $g\in C((s,\infty);\, L^q_\sigma(D_R))$, where $s\geq 0$.
\label{lem-con-bdd}
\end{lemma}
 
\begin{proof}
Fix ${\mathcal T}\in (0,\infty)$ arbitrarily.
Once we have
\begin{equation}
\sup\left\{
\|e^{-\varepsilon L_R(t)}\|_{{\mathcal L}(L^q_\sigma(D_R))};\,
0\leq t\leq {\mathcal T},\, 0\leq\varepsilon\leq 1\right\}<\infty,
\label{auxi0-bdd}
\end{equation}
\begin{equation}
\lim_{\varepsilon\to 0+}\sup_{0\leq t\leq {\mathcal T}}\|e^{-\varepsilon L_R(t)}f-f\|_{q,D_R}=0\qquad\forall\, f\in L^q_\sigma(D_R),
\label{auxi1-bdd}
\end{equation}
the same argument as in the proof of Lemma \ref{lem-con} works well to conclude \eqref{con-LRg}.
The uniformly boundedness \eqref{auxi0-bdd} is verified 
through the Dunford integral of the resolvent $(\lambda +L_R(t))^{-1}$ under the condition \eqref{ass1-rigid};
indeed, this is observed at the first stage of the construction of the evolution operator of parabolic type (\cite{Ta}).
It thus suffices to show \eqref{auxi1-bdd}. 
Let $f\in C_{0,\sigma}^\infty(D_R)\subset D_q(A_R)=D_q(L_R(t))$, then \eqref{auxi0-bdd} implies that
\begin{equation}
\begin{split}
\|e^{-\varepsilon L_R(t)}f-f\|_{q,D_R}
&\leq \int_0^\varepsilon \|e^{-sL_R(t)}L_R(t)f\|_{q,D_R}\,ds  \\
&\leq C\varepsilon\|A_R f-(\eta(t)+\omega(t)\times x)\cdot\nabla f+\omega(t)\times f\|_{q,D_R}  \\
&\leq C\varepsilon \|f\|_{W^{2,q}(D_R)} 
\end{split}
\label{auxi2-bdd}
\end{equation}
with some $C>0$ independent of $t\in [0,{\mathcal T}]$.
By the standard approximation procedure with use of \eqref{auxi0-bdd}, we are led to the desired consequence.
The proof is complete.
\end{proof}

The following lemma with $\alpha=1$ was shown by \cite[Lemma 5.2]{Hi20}, in which there was a typo;
on the 6th line of the proof there, $\phi\in C^\infty_0(D_L)$ should be replaced by $\phi\in C^\infty_0(B_L)$.
For general case $\alpha\in [0,3)$, we use the boundedness of the Riesz transform
\begin{equation}
\|\rho^\alpha{\mathcal R}h\|_{q,\mathbb R^3}\leq C\|\rho^\alpha h\|_{q,\mathbb R^3}
\label{riesz}
\end{equation}
for all $h\in L^q_\alpha(\mathbb R^3)$ provided $\frac{3}{3-\alpha}<q<\infty$.
Since the other parts of the argument based on a cut-off procedure are the same as in the proof of 
Lemma 5.2 (see also the description just before this lemma) of
\cite{Hi20}, one may omit the proof.
\begin{lemma}
Let $\alpha\in [0,3)$ and $q\in (\frac{3}{3-\alpha},\infty)$.
Then there is a constant $C=C(\alpha,q,D)>0$ such that
\begin{equation}
\|\rho^\alpha \nabla Pg\|_q\leq C\big(\|\rho^\alpha \nabla g\|_q+\|g\|_q\big)
\label{key-20}
\end{equation}
for all 
$g\in L^q(D)$ with $\nabla g\in L^q_\alpha(D)$,
where the weight function $\rho$ is given by \eqref{weight}.
\label{lem-proj-wei}
\end{lemma}

\section{Evolution operator on the whole space}
\label{sect-whole}

This section studies the initial value problem for the non-autonomous system
\begin{equation}
\begin{split}
&\partial_tu+{\mathcal L}(t)u+\nabla p=0,\qquad\mbox{div $u$}=0 \quad\mbox{in $\mathbb R^3\times (s,\infty)$}, \\
&\lim_{|x|\to\infty}u=0, \qquad u(\cdot,s)=f,
\end{split}
\label{linear-wh}
\end{equation}
where ${\mathcal L}(t)$ is the differential operator given by \eqref{pre-proj-0}.
Because of $\mbox{\eqref{grund}}_1$ we see that $\nabla p=0$ within the class $\nabla p\in L^q(\mathbb R^3)$.
Hence, under the compatibility condition $\mbox{div $f$}=0$, the solution formula to \eqref{linear-wh} is obtained from the heat semigroup
\[
\big(e^{t\Delta}f\big)(x)=(4\pi t)^{-3/2}\left(e^{-|\cdot|^2/4t}*f\right)(x)
\]
simply by transformation of variables as follows, where $*$ stands for convolution in spatial variable.
It is expicitly described as 
\begin{equation}
u(x,t)=
\big(U(t,s)f\big)(x)=
\Phi(t,s)\left(e^{(t-s)\Delta}f\right)\left(\Phi(t,s)^\top\Big(x+\int_s^t\Phi(t,\tau)\eta(\tau)\,d\tau\Big)\right),
\label{rep-wh}
\end{equation}
where $\Phi(t,s)\in \mathbb R^{3\times 3}$ stands for the evolution matrix, being an orthogonal matrix,
to the differential equation
$\frac{d\varphi}{dt}=-\omega(t)\times\varphi$.

The evolution operator $U(t,s)$ was already studied by
\cite{CM97, GH11, Ha11, HR11, HR14, Hi18, Hi20}, but we still need more to our end.
We begin with weighted estimates of 
$U(t,s)$ for initial data being not necessarily solenoidal,
that provides a solution to \eqref{linear-wh} (with $\nabla p=0$), however, without $\mbox{div $u$}=0$.
Estimate \eqref{wei-wh} below with $\alpha=0$ is the usual one which holds true for all $t>s\geq 0$, but 
\eqref{wei-wh} with $\alpha>0$ describes merely the smoothing rate near $t=s$.
For the Oseen semigroup (with constant $\eta$ and $\omega=0$) in the whole space, the results 
even with anisotropic weights reflecting the wake region have been
developed well by Tomoki Takahashi \cite{T24};
in fact, he has deduced the large time behavior, too, with less decay rate than usual. 
For the non-autonomous case as well, the rate of large time decay is observed in \eqref{wei-wh-glo} below with $k=0$,
which holds true for all $t>s\geq 0$ if $\sup_{t\geq 0}|\eta(t)|\leq m$ and which
agrees with the one shown by \cite{T24}.
Compared with Theorem \ref{weighted-est} for the exterior problem, we have \eqref{wei-wh} for every $\alpha\geq 0$, $q\in [1,\infty]$
and $j\geq 0$ being an integer.
The latter estimate \eqref{wei-wh2} for $t-s\leq {\mathcal T}$, $\alpha=1$ and $j=0$ was shown by \cite[Lemma 3.1, (3.11)]{Hi20}
under the condition $\eta\in L^\infty(0,\infty;\, \mathbb R^3)$, 
which is however replaced by \eqref{ever1} below to show \eqref{wei-wh2} 
for $0\leq s<t\leq {\mathcal T}$.
One needs \eqref{wei-wh2} 
for the proof of \eqref{wei-small1} with $j=1$.
\begin{proposition}
Suppose $\eta,\, \omega\in C([0,\infty);\, \mathbb R^3)$. 
Assume that 
$\alpha\geq 0$, $q\in [1,\infty]$ and $r\in [q,\infty]$.
Let $j\geq 0$ be an integer.

\begin{enumerate}
\item
For each ${\mathcal T}\in (0,\infty)$ and $m\in (0,\infty)$, there is a constant
$C=C({\mathcal T},m,\alpha,q,r,j)>0$ such that
\begin{equation}
\|\rho^\alpha \nabla^jU(t,s)f\|_{r,\mathbb R^3}
\leq C(t-s)^{-j/2-(3/q-3/r)/2}\|\rho^\alpha f\|_{q,\mathbb R^3}
\label{wei-wh}
\end{equation}
for all $(t,s)$ with $0\leq s<t\leq {\mathcal T}$ and $f\in L^q_\alpha(\mathbb R^3)$ whenever
\begin{equation}
\sup_{0\leq t\leq {\mathcal T}}|\eta(t)|\leq m,
\label{ever1}
\end{equation}
where the weight function $\rho$ is given by \eqref{weight}.

\item
For each ${\mathcal T}\in (0,\infty)$ and $m\in (0,\infty)$, there is a constant
$C=C({\mathcal T},m,\alpha,q,r,j)>0$ such that
\begin{equation}
\|\rho^\alpha \nabla^{j+1} U(t,s)f\|_{r,\mathbb R^3}
\leq C(t-s)^{-j/2-(3/q-3/r)/2}\|\rho^\alpha \nabla f\|_{q,\mathbb R^3}
\label{wei-wh2}
\end{equation}
for all $(t,s)$ with $0\leq s<t\leq {\mathcal T}$ and $f\in L^q(\mathbb R^3)$ with $\nabla f\in L^q_\alpha(\mathbb R^3)$ whenever
\eqref{ever1} is satisfied.
\end{enumerate}
\label{weighted-wh}
\end{proposition}

\begin{proof}
Both estimates \eqref{wei-wh} and \eqref{wei-wh2} with
$\alpha=0$ follow from the corresponding estimates of the heat semigroup.
Let $\alpha>0$, $j\geq 0$ being an integer and $k\in \{0,1\}$.
Suppose $\nabla^kf\in L^q_\alpha(\mathbb R^3)$ 
(as well as $f\in L^q(\mathbb R^3)$ to be well-definedness of $U(t,s)f$ when $k=1$)
and set $u(t)=U(t,s)f$,
then it follows from
\[
\rho(x)^\alpha\leq \max\{2^{\alpha-1},1\}\big(|x-y|^\alpha+\rho(y)^\alpha\big)
\]
that
\begin{equation*}
\quad \rho(x)^\alpha|\nabla^{j+k}u(x,t)|
\leq C(I+II)
\end{equation*}
with
\begin{equation*}
\begin{split}
I:&=(t-s)^{-(3+j)/2}\int_{\mathbb R^3}|x-y|^\alpha e^{-|x-y|^2/8(t-s)}|\nabla^kg(y)|\,dy  \\
&\leq C(t-s)^{-(3+j-\alpha)/2}\int_{\mathbb R^3}e^{-|x-y|^2/16(t-s)}|\nabla^kg(y)|\,dy,
\end{split}
\end{equation*}
\[
II:=(t-s)^{-(3+j)/2}\int_{\mathbb R^3}e^{-|x-y|^2/8(t-s)}\rho(y)^\alpha |\nabla^kg(y)|\,dy,
\]
where
\begin{equation*}
\begin{split}
g(y):&=f\left(
\Phi(t,s)^\top\Big(y+\int_s^t\Phi(t,\tau)\eta(\tau)\,d\tau\Big)\right),  \\
\partial_l g(y)&=\sum_{m=1}^3 \Phi_{lm}(t,s)\big(\partial_m f\big)
\left(
\Phi(t,s)^\top\Big(y+\int_s^t\Phi(t,\tau)\eta(\tau)\,d\tau\Big)\right), 
\end{split}
\end{equation*}
for $l\in \{1,2,3\}$.
Hence, we obtain 
\begin{equation*}
\begin{split}
&\quad \|\rho^\alpha \nabla^{j+k}u(t)\|_{r,\mathbb R^3}  \\
&\leq C(t-s)^{-(j-\alpha)/2-(3/q-3/r)/2}\|\nabla^kg\|_{q,\mathbb R^3}
+C(t-s)^{-j/2-(3/q-3/r)/2}\|\rho^\alpha \nabla^kg\|_{q,\mathbb R^3}
\end{split}
\end{equation*}
for $r\in [q,\infty]$.
Taking account of
\[
\|\rho^\alpha \nabla^kg\|_{q,\mathbb R^3}\leq C\|\rho^\alpha \nabla^kf\|_{q,\mathbb R^3}+Cm^\alpha (t-s)^\alpha\|\nabla^kf\|_{q,\mathbb R^3}
\]
and $\|\nabla^kg\|_{q,\mathbb R^3}=\|\nabla^kf\|_{q,\mathbb R^3}$,
we conclude 
\begin{equation}
\begin{split}
&\quad \|\rho^\alpha \nabla^{j+k}U(t,s)f\|_{r,\mathbb R^3}  \\
&\leq C(t-s)^{-j/2-(3/q-3/r)/2}
\Big[
\|\rho^\alpha \nabla^kf\|_{q,\mathbb R^3}+(t-s)^{\alpha/2}\big\{
1+m^\alpha (t-s)^{\alpha/2}\big\}
\|\nabla^kf\|_{q,\mathbb R^3}\Big]
\end{split}
\label{wei-wh-glo}
\end{equation}
with some $C>0$ independent of $m$, that leads us to \eqref{wei-wh} and \eqref{wei-wh2} near $t=s$.
The proof is complete.
\end{proof}

Let $1<q<\infty$.
As in \eqref{ig}, we introduce the family of linear operators $\{L_{\mathbb R^3}(t)\}_{t\geq 0}$ on $L^q_\sigma(\mathbb R^3)$ by
\begin{equation}
\begin{split}
D_q(L_{\mathbb R^3}(t))
&=\{u\in L^q_\sigma(\mathbb R^3)\cap W^{2,q}(\mathbb R^3);\, (\omega(t)\times x)\cdot\nabla u\in L^q(\mathbb R^3)\},  \\
L_{\mathbb R^3}(t)u
&=P_{\mathbb R^3}{\mathcal L}(t)u={\mathcal L}(t)u,
\end{split}
\label{ig-wh}
\end{equation}
where ${\mathcal L}(t)$ is given by \eqref{pre-proj-0} and the latter equality above follows from $\mbox{\eqref{grund}}_1$.
We refer to \cite[Theorem 4.1]{MPRH} for the right description of the domain \eqref{ig-wh}
such that it is indeed the generator of the semigroup on $L^q_\sigma(\mathbb R^3)$ for each $t\geq 0$.
The first half of the following proposition tells us that the evolution operator $U(t,s)$ provides a strong solution to \eqref{linear-wh}
when $f\in L^q_{1,\sigma}(\mathbb R^3)$.
This improves the corresponding results \cite{HR11, HR14, Hi20} for $f\in Z_q(\mathbb R^3)$,
see \eqref{Z}.
As an estimate relating to \eqref{B-est-wh} below with $k=0$ over $B_R$,
we already know from \cite[Lemma 3.2]{Hi20} that $U(\cdot,s)g\in C^1((s,\infty);\, W^{-1,q}(B_R))$ subject to
\begin{equation}
\|\partial_tU(t,s)g\|_{W^{-1,q}(B_R)}\leq C(t-s)^{-1/2}\|g\|_{q,\mathbb R^3}
\label{weak-wh}
\end{equation}
for all $(t,s)$ with $0\leq s<t\leq {\mathcal T}$ and $g\in L^q_\sigma(\mathbb R^3)$.
For the proof of \eqref{wei-small1}, estimate \eqref{B-est-wh} with $k=1$ is needed via \eqref{remain-gr-wei}. 
\begin{proposition}
Suppose $\eta,\,\omega\in C([0,\infty);\, \mathbb R^3)$.
Let $q\in (1,\infty)$.
For every $f\in L^q_{1,\sigma}(\mathbb R^3)$, we have
\begin{equation*}
\begin{split}
&U(\cdot,s)f\in C^1((s,\infty);\,L^q_\sigma(\mathbb R^3))\cap C((s,\infty);\,W^{j,q}(\mathbb R^3))\quad \forall\, j\geq 0, \\
&U(t,s)f\in Y_q(\mathbb R^3)\quad\forall\,t\in (s,\infty),
\end{split}
\end{equation*}
with
\[
\partial_tU(t,s)f+L_{\mathbb R^3}(t)U(t,s)f=0, \qquad t\in (s,\infty),
\]
in $L^q_\sigma(\mathbb R^3)$.

Let $j\geq 0$ be an integer.
For each ${\mathcal T}\in (0,\infty)$, $m\in (0,\infty)$, $\mu\in (0,1)$ and $R\in (0,\infty)$, 
there are constants $C_1=C_1({\mathcal T},m,\mu,q,j)>0$,
$C_2=C_2({\mathcal T},m,q)>0$ and $C_3=C_3({\mathcal T},m,q,R)>0$
such that
\begin{equation}
\|\nabla^jU(t,s)f-\nabla^jU(\tau,s)f\|_{q,\mathbb R^3}
\leq C_1(t-\tau)^\mu (\tau-s)^{-j/2-\mu}\|\rho f\|_{q,\mathbb R^3}
\label{2nd-hoe-wh}
\end{equation}
\begin{equation}
\|\partial_tU(t,s)f\|_{q,\mathbb R^3}
\leq C_2\big[(t-s)^{-1}\|f\|_{q,\mathbb R^3}+(t-s)^{-1/2}\|\rho f\|_{q,\mathbb R^3}\big]
\label{str-est-wh}
\end{equation}
\begin{equation}
\|\partial_tU(t,s)g\|_{q,B_R} 
\leq C_3(t-s)^{-1+k/2}\|g\|_{W^{k,q}(\mathbb R^3)}
\label{B-est-wh}
\end{equation}
for all $(t,\tau,s)$ with $0\leq s<\tau<t\leq {\mathcal T}$ ($\tau$ is absent for the latter two estimates), 
$f\in L^q_{1,\sigma}(\mathbb R^3)$ and
$g\in W^{k,q}(\mathbb R^3)\cap L^q_\sigma(\mathbb R^3)$
with $k\in \{0,1\}$ whenever \eqref{ever1} is satisfied.
\label{strong-wh}
\end{proposition}

\begin{proof}
Set $u(t)=U(t,s)f$.
We know from \cite[Proposition 2.1]{HR11}, 
\cite[Proposition 3.1]{HR14} and \cite[Lemma 3.1, item 3]{Hi20}
that $u(t)$ possesses the desired regularity (except higher order spatial regularity)
and it is a strong solution
as long as $f\in Z_q(\mathbb R^3)$, see \eqref{Z}.
This combined with Proposition \ref{weighted-wh}, which claims $U(t,s)f\in Z_q(\mathbb R^3)$ for all $t>s$ when 
$f\in L^q_{1,\sigma}(\mathbb R^3)$,
implies the first half of Proposition \ref{strong-wh} on account of the semigroup property,
except for $u\in C((s,\infty);\, W^{j,q}(\mathbb R^3))$, which will be next verified together with \eqref{2nd-hoe-wh}.

For $f\in L^q_{1,\sigma}(\mathbb R^3)$, we find from \eqref{wei-wh} that
\begin{equation}
\begin{split}
&\quad \|\nabla^ju(t)-\nabla^ju(\tau)\|_{q,\mathbb R^3} 
\leq \int_\tau^t \|\nabla^j {\mathcal L}(\sigma)u(\sigma)\|_{q,\mathbb R^3} \,d\sigma  \\
&\leq C\int_\tau^t \Big(\|\nabla^{j+2}u(\sigma)\|_{q,\mathbb R^3}
+\|\rho\nabla^{j+1}u(\sigma)\|_{q,\mathbb R^3}
+\|\nabla^ju(\sigma)\|_{q,\mathbb R^3}\Big)\,d\sigma  \\
&\leq C(t-\tau)\Big\{(\tau-s)^{-j/2-1}\|f\|_{q,\mathbb R^3}+(\tau-s)^{-j/2-1/2}\|\rho f\|_{q,\mathbb R^3}
+(\tau-s)^{-j/2}\|f\|_{q,\mathbb R^3}\Big\}
\end{split}
\label{2nd-hoe1}
\end{equation}
and that
\begin{equation}
\|\nabla^ju(t)-\nabla^ju(\tau)\|_{q,\mathbb R^3}
\leq C(\tau-s)^{-j/2}\|f\|_{q,\mathbb R^3}.
\label{2nd-hoe2}
\end{equation}
Given $\mu\in (0,1)$, we compute the product
$\mbox{\eqref{2nd-hoe1}}^{\mu}\mbox{\eqref{2nd-hoe2}}^{1-\mu}$ to conclude \eqref{2nd-hoe-wh}.

By \eqref{wei-wh} again it is readily seen 
that
\begin{equation}
\begin{split}
\|\partial_tu(t)\|_{q,\mathbb R^3}
&\leq \|\Delta u(t)\|_{q,\mathbb R^3}+C
\|\rho \nabla u(t)\|_{q,\mathbb R^3}+C\|u(t)\|_{q,\mathbb R^3} \\
&\leq C(t-s)^{-1}\|f\|_{q,\mathbb R^3}+C(t-s)^{-1/2}\|\rho f\|_{q,\mathbb R^3}+C\|f\|_{q,\mathbb R^3}
\end{split}
\label{evo-str-wh}
\end{equation}
which leads to \eqref{str-est-wh}.

By \eqref{wei-wh} and \eqref{wei-wh2} with $\alpha=0$ we find
\begin{equation}
\|u(t)\|_{W^{2,q}(\mathbb R^3)}\leq C(t-s)^{-1+k/2}\|f\|_{W^{k,q}(\mathbb R^3)}
\label{less-sing-wh}
\end{equation}
for $k\in \{0,1\}$ and $f\in 
W^{k,q}(\mathbb R^3)\cap L^q_\sigma(\mathbb R^3)$.
Recall that $\partial_tu(t)$ exists as an element of 
$W^{-1,q}(B_R)$ for every $f\in L^q_\sigma(\mathbb R^3)$, see \eqref{weak-wh}; 
in fact, without the condition $\rho f\in L^q(\mathbb R^3)$, it makes sense even in $L^q(B_R)$ subject to
\begin{equation}
\|\partial_tu(t)\|_{q,B_R}
\leq C\|u(t)\|_{W^{2,q}(B_R)} 
\leq C(t-s)^{-1}\|f\|_{q,\mathbb R^3}
\label{evo-str-loc-wh}
\end{equation}
for every $f\in L^q_\sigma(\mathbb R^3)$.
If, in addition, $f\in W^{1,q}(\mathbb R^3)$,
then
\eqref{less-sing-wh} with $k=1$
implies \eqref{B-est-wh} (with $g=f$).
\end{proof}
\begin{remark}
Let $f\in Z_q(\mathbb R^3)$, see \eqref{Z}, instead of $f\in L^q_{1,\sigma}(\mathbb R^3)$.
Then one can still compute the difference
$\nabla^ju(t)-\nabla^ju(\tau)$ as in \eqref{2nd-hoe1}, however, the second term of the right-hand side 
of \eqref{2nd-hoe1} should be
replaced with $(\tau-s)^{-j/2}\|\rho \nabla f\|_{q,\mathbb R^3}$ by virtue of \eqref{wei-wh2}.
We thus obtain 
\begin{equation*}
\|\nabla^jU(t,s)f-\nabla^jU(\tau,s)f\|_{q,\mathbb R^3} 
\leq C(t-\tau)^\mu
\Big[(\tau-s)^{-j/2-\mu}\|f\|_{q,\mathbb R^3}+(\tau-s)^{-j/2}\|f\|_{q,\mathbb R^3}^{1-\mu}\|\rho \nabla f\|_{q,\mathbb R^3}^\mu\Big]
\end{equation*}
for every integer $j\geq 0$ and $\mu\in (0,1)$. 
\label{rem-hoe-wh}
\end{remark}
\begin{remark}
Although we have the H\"older estimate \eqref{2nd-hoe-wh} for $\nabla^jU(\cdot,s)f$ with $f\in L^q_{1,\sigma}(\mathbb R^3)$
even though $j\geq 2$,
it is not clear whether
$U(\cdot,s)f\in C((s,\infty);\, Y_q(\mathbb R^3))$ for such $f$, see Remark \ref{rem-conti}.
This continuity can be deduced if, in addition, $f\in L^q_{2,\sigma}(\mathbb R^3)$ because we see that
\begin{equation}
\|\rho \nabla U(t,s)f-\rho\nabla U(\tau,s)f\|_{q,\mathbb R^3}
\leq C(t-\tau)^\mu (\tau-s)^{-1/2-\mu}\|\rho^2 f\|_{q,\mathbb R^3}
\label{wei-hoe-wh}
\end{equation}
for every $\mu\in (0,1/2]$ by the same way as in \eqref{hoe-based}--\eqref{hoe-split} below on exterior domains.
For the exterior problem, 
indeed we know $T(\cdot,s)f\in C((s,\infty);\, W^{2,q}(D))$ with $f\in L^q_{1,\sigma}(D)$ from the item 3 of
Theorem \ref{strong-sol}, but the corresponding H\"older estimate to \eqref{2nd-hoe-wh} with $j=2$
is not available yet, for we 
do not have \eqref{wei-est1} with $j\geq 2$, so that the computation \eqref{2nd-hoe1} does not work on exterior domains
unless $j=0$.
\label{rem-2nd-wh}
\end{remark}

For later use, we prepare the following lemma on the corresponding semigroup.
The item 1 is soon used in the proof of Lemma \ref{lem-evo-semi}, while the item 2 is needed for \eqref{unif-conv}.
\begin{lemma}
Suppose $\eta,\,\omega\in C([0,\infty);\, \mathbb R^3)$.
Let $q\in (1,\infty)$ and let $L_{\mathbb R^3}(t)$ be the operator on $L^q_\sigma(\mathbb R^3)$ given by \eqref{ig-wh}.

\begin{enumerate}
\item
For each ${\mathcal T}\in (0,\infty)$, $m\in (0,\infty)$ and $R\in (0,\infty)$, 
there are constants $C=C({\mathcal T},m,q)>0$ and $C^\prime=C^\prime({\mathcal T},m,q,R)>0$ such that
\begin{equation}
\|L_{\mathbb R^3}(t)e^{-(t-s)L_{\mathbb R^3}(t)}f\|_{q,\mathbb R^3}
\leq C\big[(t-s)^{-1}\|f\|_{q,\mathbb R^3}+(t-s)^{-1/2}\|\rho f\|_{q,\mathbb R^3}\big]
\label{semi-wh-0}
\end{equation}
\begin{equation}
\|L_{\mathbb R^3}(t)e^{-(t-s)L_{\mathbb R^3}(t)}g\|_{q,B_R}
\leq C^\prime(t-s)^{-1}\|g\|_{q,\mathbb R^3}
\label{semi-wh-loc}
\end{equation}
for all $(t,s)$ with $0\leq s<t\leq {\mathcal T}$, $f\in L^q_{1,\sigma}(\mathbb R^3)$ 
and $g\in L^q_\sigma(\mathbb R^3)$ whenever \eqref{ever1} is satisfied.

\item
For every compact interval $J\subset (s,\infty)$ and $g\in C((s,\infty);\, L^q_\sigma(\mathbb R^3))$ we have
\begin{equation}
\lim_{\varepsilon\to 0+}\sup_{t\in J}\|e^{-\varepsilon L_{\mathbb R^3}(t)}g(t)-g(t)\|_{q,\mathbb R^3}=0,
\label{con-whg}
\end{equation}
where $s\geq 0$.
\end{enumerate}
\label{lem-sg-wh}
\end{lemma}

\begin{proof}
Consider the semigroup $u(\tau)=e^{-\tau L_{\mathbb R^3}(t)}f$ with the freezing parameter $t\in (0,{\mathcal T}]$
and $f\in L^q_{1,\sigma}(\mathbb R^3)$.
Let $\Psi_{\omega(t)}(\tau)\in \mathbb R^{3\times 3}$ be the orthogonal matrix that obeys
$\frac{d}{d\tau}\Psi_{\omega(t)}(\tau)=-\omega(t)\times \Psi_{\omega(t)}(\tau)$ with $\Psi_{\omega(t)}(0)=\mathbb I$ being the
unity matrix. Then, as in \eqref{rep-wh}, this semigroup is explicitly represented as
\[
\big(e^{-\tau L_{\mathbb R^3}(t)}f\big)(x)
=\Psi_{\omega(t)}(\tau)\big(e^{\tau\Delta}f\big)
\left(\Psi_{\omega(t)}(\tau)^\top \Big(x+\int_0^\tau\Psi_{\omega(t)}(\tau-\sigma)\eta(t)\,d\sigma\Big)\right).
\]
Therefore, we have the same weighted estimate for this semigroup as in \eqref{wei-wh}.
Hence, as in \eqref{evo-str-wh} 
we infer
\[
\|{\mathcal L}(t)u(\tau)\|_{q,\mathbb R^3}\leq C\big[\tau^{-1}\|f\|_{q,\mathbb R^3}+\tau^{-1/2}\|\rho f\|_{q,\mathbb R^3}\big]
\]
for $0<\tau\leq {\mathcal T}$.
We set $\tau=t-s$ to get \eqref{semi-wh-0}.
It is also easy to verify \eqref{semi-wh-loc} as in \eqref{evo-str-loc-wh}.

In an analogous way to
the proof of Lemma \ref{lem-con-bdd}, we conclude \eqref{con-whg} with the corresponding two claims to 
\eqref{auxi0-bdd}--\eqref{auxi1-bdd} at hand.
The former follows from the representation formula of the semigroup $e^{-\varepsilon L_{\mathbb R^3}(t)}$ described above,
whereas, as in \eqref{auxi2-bdd}, estimate
\[
\sup_{0\le t\leq{\mathcal T}}
\|e^{-\varepsilon L_{\mathbb R^3}(t)}f-f\|_{q,\mathbb R^3}
\leq C\varepsilon \|f\|_{Y_q(\mathbb R^3)} 
\]
for $f\in C^\infty_{0,\sigma}(\mathbb R^3)$ implies the latter.
The proof is complete.
\end{proof}

The next lemma plays an important role in section \ref{sect-duha}.
In fact, estimate \eqref{evo-semi-wh-0} below is the only reason why more weight $[g]_{q,2,\kappa}<\infty$ is asked 
on the external force to deduce
\eqref{duha-str-est} for the Duhamel term \eqref{duha}, see the assumption \eqref{ass2-force}
and Theorem \ref{duha-strong}.
\begin{lemma}
Suppose that $\eta$ and $\omega$ fulfill \eqref{ass1-rigid} for some $\theta\in (0,1]$.
Let $q\in (1,\infty)$.
For each ${\mathcal T}\in (0,\infty)$, $m\in (0,\infty)$ and $R\in (0,\infty)$, 
there are constants $C=C({\mathcal T},m,q,\theta)>0$ and $C^\prime=C^\prime({\mathcal T},m,q,R,\theta)>0$ such that
\begin{equation}
\|L_{\mathbb R^3}(t)\big(U(t,s)-e^{-(t-s)L_{\mathbb R^3}(t)}\big)f\|_{q,\mathbb R^3}
\leq C(t-s)^\theta\Big[(t-s)^{-1/2}\|\rho f\|_{q,\mathbb R^3}+\|\rho^2f\|_{q,\mathbb R^3}\Big]
\label{evo-semi-wh-0}
\end{equation}
\begin{equation}
\|L_{\mathbb R^3}(t)\big(U(t,s)-e^{-(t-s)L_{\mathbb R^3}(t)}\big)g\|_{q,B_R}
\leq C^\prime (t-s)^{-1/2+\theta}\|\rho g\|_{q,\mathbb R^3}
\label{evo-semi-loc-0}
\end{equation}
for all $(t,s)$ with $0\leq s<t\leq {\mathcal T}$, $f\in L^q_{2,\sigma}(\mathbb R^3)$ and
$g\in L^q_{1,\sigma}(\mathbb R^3)$ whenever \eqref{ever} is satisfied.
\label{lem-evo-semi}
\end{lemma}

\begin{proof}
Let $f\in L^q_{1,\sigma}(\mathbb R^3)$, then we have $U(\sigma,s)f\in Y_q(\mathbb R^3)$ for all $\sigma>s$ by Propsition \ref{strong-wh}.
Since
\begin{equation*}
\begin{split}
&\quad U(t,s)f-e^{-(t-s)L_{\mathbb R^3}(t)}f  \\
&=\int_s^t \partial_\sigma \left(e^{-(t-\sigma)L_{\mathbb R^3}(t)}U(\sigma,s)f\right)\,d\sigma  \\
&=\int_s^t \Big[e^{-(t-\sigma)L_{\mathbb R^3}(t)}L_{\mathbb R^3}(t)U(\sigma,s)f
-e^{-(t-\sigma)L_{\mathbb R^3}(t)}L_{\mathbb R^3}(\sigma)U(\sigma,s)f\Big] \,d\sigma  \\
&=\int_s^t e^{-(t-\sigma)L_{\mathbb R^3}(t)}
\big({\mathcal L}(t)-{\mathcal L}(\sigma)\big)
U(\sigma,s)f\,d\sigma
\end{split}
\end{equation*}
in $L^q_\sigma(\mathbb R^3)$ with 
\begin{equation}
\begin{split}
&\quad \big({\mathcal L}(t)-{\mathcal L}(\sigma)\big)w  \\
&=-\Big[
\big(\eta(t)-\eta(\sigma)\big)\cdot\nabla w+\left\{\big(\omega(t)-\omega(\sigma)\big)\times x\right\}\cdot\nabla w
-\big(\omega(t)-\omega(\sigma)\big)\times w
\Big],
\end{split}
\label{diff-str-0}
\end{equation}
it follows from \eqref{wei-wh} and \eqref{semi-wh-0} 
together with \eqref{ass1-rigid} that
\begin{equation*}
\begin{split}
&\quad \|L_{\mathbb R^3}(t)\left(U(t,s)-e^{-(t-s)L_{\mathbb R^3}(t)}\right)f\|_{q,\mathbb R^3}  \\
&\leq \int_s^t \|L_{\mathbb R^3}(t)e^{-(t-\sigma)L_{\mathbb R^3}(t)}
\big({\mathcal L}(t)-{\mathcal L}(\sigma)\big)
U(\sigma,s)f\|_{q,\mathbb R^3} \,d\sigma  \\
&\leq C\int_s^t \Big\{(t-\sigma)^{-1+\theta}
\big(\|\rho \nabla U(\sigma,s)f\|_{q,\mathbb R^3}+\|U(\sigma,s)f\|_{q,\mathbb R^3}\big)   \\
&\qquad\qquad 
+(t-\sigma)^{-1/2+\theta}\big(\|\rho^2\nabla U(\sigma,s)f\|_{q,\mathbb R^3}+\|\rho U(\sigma,s)f\|_{q,\mathbb R^3}\big)\Big\}
  \,d\sigma  \\
&\leq C\int_s^t 
\Big[(t-\sigma)^{-1+\theta}\big\{(\sigma-s)^{-1/2}\|\rho f\|_{q,\mathbb R^3}+\|f\|_{q,\mathbb R^3}\big\}  \\
&\qquad\qquad +(t-\sigma)^{-1/2+\theta}\big\{(\sigma-s)^{-1/2}\|\rho^2 f\|_{q,\mathbb R^3}+\|\rho f\|_{q,\mathbb R^3}\big\}\Big] 
\,d\sigma   \\
&=C(t-s)^{-1/2+\theta}\|\rho f\|_{q,\mathbb R^3}
+C(t-s)^\theta \|\rho^2 f\|_{q,\mathbb R^3}
\end{split}
\end{equation*}
for $0\leq s<t\leq {\mathcal T}$,
which proves \eqref{evo-semi-wh-0}
if, in addition, $f\in L^q_{2,\sigma}(\mathbb R^3)$.
Finally, when taking the norm $\|\cdot\|_{q,B_R}$ of the left-hand side,
we employ \eqref{semi-wh-loc}
instead of \eqref{semi-wh-0} to find \eqref{evo-semi-loc-0}.
The proof is complete.
\end{proof}
\begin{remark}
For \eqref{evo-semi-wh-0}, we are forced to impose $\rho^2f\in L^q(\mathbb R^3)$ 
on $f$; on the other hand, it exhibits less singular behavior $(t-s)^{-1/2+\theta}$ than the one from general theory
of evolution operators of parabolic type, see
\eqref{evo-semi-bdd-0}, and this is bacause we do not have any time-dependent coefficient in the second order term
of ${\mathcal L}(t)$ given by \eqref{pre-proj-0}.
For this reason, in fact, it turns out that the behavior \eqref{evo-semi-bdd-0} near $t=s$ can be improved as $(t-s)^{-1/2+\theta}$ 
on the bounded domain $D_R$ as well along the same argument as above,
but this improvement is not needed later.
\label{rem-evo-semi}
\end{remark}

\section{Weighted estimate}
\label{sect-wei}

In this section we develop the weighted estimates of the evolution operator $T(t,s)$ on exterior domains 
to show Theorem \ref{weighted-est}.
To this end, we have to be back to the stage of construction of
a parametrix of the evolution operator 
due to Hansel and Rhandi \cite{HR14}, see also \cite[section 5]{Hi20}.
It is done in the following way by use of evolution operators in the whole space $\mathbb R^3$ and in the bounded
domain $D_7$.
We fix three cut-off functions
\begin{equation*}
\begin{split}
&\phi\in C_0^\infty(B_5), \qquad \phi=1\;\;\mbox{in $B_4$}, \\
&\phi_0\in C_0^\infty(B_3), \qquad \phi_0=1\;\;\mbox{in $B_2$}, \\
&\phi_1\in C_0^\infty(B_7), \qquad \phi_1=1\;\;\mbox{in $B_6$},
\end{split}
\end{equation*}
and set
\[
G=\{3<|x|<5\}, \quad G_0=\{1<|x|<3\}, \quad G_1=\{5<|x|<7\}.
\]
By $\mathbb B=\mathbb B_G$, $\mathbb B_0=\mathbb B_{G_0}$ and $\mathbb B_1=\mathbb B_{G_1}$
we denote the Bogovskii operator introduced in Lemma \ref{lem-bog}.

Let $q\in (1,\infty)$.
Given $f\in L^q_{\sigma}(D)$, we set
\begin{equation}
f_0=(1-\phi_0)f+\mathbb B_0[f\cdot\nabla\phi_0], 
\qquad
f_1=\phi_1f-\mathbb B_1[f\cdot\nabla\phi_1], 
\label{modi}
\end{equation}
as in \eqref{cutoff},
where $f_0$ is understood as its extension to $\mathbb R^3$ by setting zero outside $D$.
If, in particular, $f\in L^q_{\alpha,\sigma}(D)$ with $\alpha\geq 0$, see \eqref{sole-chara-wei}, then
we have $f_0\in L^q_{\alpha,\sigma}(\mathbb R^3)$ as well as $f_1\in L^q_\sigma(D_7)$
subject to
\begin{equation}
\|\rho^\alpha f_0\|_{q,\mathbb R^3}\leq C\|\rho^\alpha f\|_q, \qquad
\|f_1\|_{q,D_7}\leq C\|f\|_q
\label{data-est}
\end{equation}
by \eqref{bog-est1}.
We also observe
\begin{equation}
\begin{split}
&\|\nabla f_0\|_{q,\mathbb R^3}+\|\nabla f_1\|_{q,D_7}\leq C\|f\|_{W^{1,q}(D)}, \\
&\|\rho^\alpha \nabla f_0\|_{q,\mathbb R^3}\leq C\big(\|\rho^\alpha \nabla f\|_q+\|f\|_q\big),
\end{split}
\label{data-est2}
\end{equation}
if, in addition, $\nabla f\in L^q_\alpha(D)$ with $\alpha\geq 0$.

As a fine approximation of the evolution operator, it is natural to take
\begin{equation}
W(t,s)f=(1-\phi)U(t,s)f_0+\phi V(t,s)f_1+\mathbb B[(U(t,s)f_0-V(t,s)f_1)\cdot\nabla\phi]
\label{appro}
\end{equation}
which fulfills $W(s,s)f=f$,
where $U(t,s)$ is the evolution operator, see \eqref{rep-wh}, for the whole space problem
and $V(t,s)$ is the one for the interior problem over $D_7$, see Lemma \ref{lem-evo-bdd}.
We fix ${\mathcal T},\, m\in (0,\infty)$ and suppose \eqref{ever}. 
Let $0\leq s<t\leq {\mathcal T}$.
Let also $r\in [q,\infty]$ and $j\in \{0,1\}$, then
we see from \eqref{bog-est1}, \eqref{sm-bdd}, \eqref{1st-bdd}
\eqref{wei-wh}, \eqref{wei-wh2} with $(j,r)=(0,q)$, \eqref{data-est} and \eqref{data-est2} that
\begin{equation}
\|\rho^\alpha \nabla^j W(t,s)f\|_r\leq C(t-s)^{-j/2-(3/q-3/r)/2}\|\rho^\alpha f\|_q
\label{appro-est}
\end{equation}
for $f\in L^q_{\alpha,\sigma}(D)$ and that
\begin{equation}
\|\rho^\alpha\nabla W(t,s)f\|_q\leq C(t-s)^{-1/2+\delta}\big(\|\rho^\alpha \nabla f\|_q+\|f\|_q\big)
\label{appro-est-wei}
\end{equation}
for $f\in L^q_\sigma(D)$ with $\nabla f\in L^q_\alpha(D)$, where $\delta\in (0,1/2q)$ is arbitrary.

With the pressure $p_1(\cdot,t)$ on $D_7$ associated with $V(t,s)f_1$ in such a way that 
$\int_{D_7}p_1(t)\,dx=0$, we consider the pair of
\[
u:=W(t,s)f, \qquad p:=\phi p_1(t),
\]
which should obey
\begin{equation}
\begin{split}
&\partial_tu+{\mathcal L}(t)u+\nabla p+K(t,s)f=0, \qquad\mbox{div $u$}=0 \quad\mbox{in $D\times (s,\infty)$}, \\
&u|_{\partial D}=0, \qquad \lim_{|x|\to\infty}u=0, \qquad u(\cdot,s)=f,
\end{split}
\label{W-pde}
\end{equation}
where
${\mathcal L}(t)$ is the differential operator given by \eqref{pre-proj-0}.
The equation is understood as
\begin{equation}
\partial_tW(t,s)f+L(t)W(t,s)f=-PK(t,s)f,\qquad t\in (s,\infty), 
\label{W-eq}
\end{equation}
in $L^q_\sigma(D)$ as long as 
$f\in L^q_{1,\sigma}(D)$, yielding $f_0\in L^q_{1,\sigma}(\mathbb R^3)$
which ensures the desired regularity of $U(t,s)f_0$ by Proposition \ref{strong-wh}.
But
this does not matter in the 
present section
because the construction of the evolution operator is based
on the integral equation \eqref{ite-IE} below.
Here, the remainder term $K(t,s)f$ is given by
\begin{equation}
\begin{split}
&\quad K(t,s)f \\
=&-2\nabla\phi\cdot\nabla(Uf_0-Vf_1)-\{\Delta\phi+(\eta+\omega\times x)\cdot\nabla\phi\}(Uf_0-Vf_1) \\
&-(\nabla\phi)p_1-\mathbb B[(\partial_tUf_0-\partial_tVf_1)\cdot\nabla\phi]
+\Delta\mathbb B[(Uf_0-Vf_1)\cdot\nabla\phi]  \\
&+(\eta+\omega\times x)\cdot\nabla \mathbb B[(Uf_0-Vf_1)\cdot\nabla\phi]
-\omega\times \mathbb B[(Uf_0-Vf_1)\cdot\nabla\phi],
\end{split}
\label{remainder}
\end{equation}
where we abbreviate $Uf_0=U(t,s)f_0$ and $Vf_1=V(t,s)f_1$.
It follows from \eqref{bog-est1}--\eqref{bog-est2},
\eqref{sm-bdd}--\eqref{pressure-bdd}, \eqref{frac-bdd} with $\delta=0$, \eqref{high-bdd},
\eqref{wei-wh} with $\alpha=0$, \eqref{wei-wh2} with $\alpha=0$, \eqref{weak-wh},
\eqref{B-est-wh}, \eqref{data-est} and \eqref{data-est2} that
\begin{equation}
\|K(t,s)f\|_q\leq C(t-s)^{-(1+1/q)/2}\|f\|_q
\label{remain-est}
\end{equation}
and that
\begin{equation}
\begin{split}
&\quad \|\nabla K(t,s)f\|_q   \\
&\leq C\big(\|U(t,s)f_0\|_{W^{2,q}(G)}+\|\partial_tU(t,s)f_0\|_{q,G}+\|V(t,s)f_1\|_{W^{2,q}(D_7)}
+\|\partial_tV(t,s)f_1\|_{q,D_7}+\|\nabla p_1(t)\|_{q,D_7}\big)  \\
&\leq 
\left\{
\begin{array}{l}
C(t-s)^{-1}\|f\|_q, \\
C(t-s)^{-1+\delta}\|f\|_{W^{1,q}(D)},
\end{array}
\right.
\end{split}
\label{remain-gr-est}
\end{equation}
where $\delta\in (0,1/2q)$ is arbitrary.
The former of \eqref{remain-gr-est} will be used for \eqref{delta-est} in the next section,
whereas the latter soon plays a role via \eqref{remain-gr-wei}
to show \eqref{gZZest} below, leading to \eqref{wei-small1}.
Note that, thanks to \eqref{B-est-wh},
one does not need any weighted norm of $f$ in the latter estimate of \eqref{remain-gr-est} unlike \cite[Lemma 5.3]{Hi20}.

We now suppose $\alpha\in [0,3)$ and $q\in (\frac{3}{3-\alpha},\infty)$.
Since the support of $K(t,s)f$ is compact in $D$, 
\eqref{remain-est} and \eqref{remain-gr-est} respectively give
\begin{equation}
\|\rho^\alpha PK(t,s)f\|_q\leq C(t-s)^{-(1+1/q)/2}\|f\|_q
\label{remain-wei}
\end{equation}
and
\begin{equation}
\|\rho^\alpha \nabla PK(t,s)f\|_q\leq
\left\{
\begin{array}{l}
C(t-s)^{-1}\|f\|_q, \\
C(t-s)^{-1+\delta}\|f\|_{W^{1,q}(D)},
\end{array}
\right.
\label{remain-gr-wei}
\end{equation}
with $\delta\in (0,1/2q)$ on account of Lemma \ref{lem-proj-wei} as well as
\eqref{FK-wei}. 

In view of \eqref{W-pde}, the reasonable idea of Hansel and Rhandi \cite{HR14} for construction 
of the evolution operator $T(t,s)$ is to solve the integral equation
\begin{equation}
T(t,s)f=W(t,s)f+\int_s^t T(t,\tau)PK(\tau,s)f\,d\tau
\label{ite-IE}
\end{equation}
through the iteration scheme ($j=0,1,2,...$)
\begin{equation}
T_{j+1}(t,s)f=\int_s^tT_j(t,\tau)PK(\tau,s)f\,d\tau, \qquad
T_0(t,s)f=W(t,s)f,
\label{iteration}
\end{equation}
so that
\begin{equation}
T(t,s)f=\sum_{j=0}^\infty T_j(t,s)f
\label{evo-series}
\end{equation}
is convergent in $L^q_\sigma(D)$ absolutely and uniformly in $(t,s)$ with $0\leq s\leq t\leq {\mathcal T}$
for every $f\in L^q_\sigma(D)$, which follows from \eqref{appro-est} with $(\alpha,j,r)=(0,0,q)$ and \eqref{remain-est}
with the aid of Lemma \ref{lem-ite}, see \cite{HR14}.

We are in a position to show Theorem \ref{weighted-est}.
\medskip

\noindent
{\it Proof of Theorem \ref{weighted-est}}.
Let $\alpha\in [0,3)$, $q\in (\frac{3}{3-\alpha},\infty)$, $r\in [q,\infty]$ and $j\in \{0,1\}$.
In view of \eqref{appro-est}, \eqref{remain-wei} and \eqref{iteration}, we apply
Lemma \ref{lem-ite} 
with
\begin{equation*}
\begin{split}
&E_0=\nabla^j W, \quad Q=PK, \quad X_1=L^q_{\alpha,\sigma}(D), \quad X_2=L^r_\alpha(D), \\
&\beta=\frac{j}{2}+\frac{3}{2}\left(\frac{1}{q}-\frac{1}{r}\right), \quad \gamma=\frac{1}{2}\left(1+\frac{1}{q}\right)
\end{split}
\end{equation*}
to conclude \eqref{wei-est1} 
provided $1/q-1/r<(2-j)/3$, 
however, this restriction on $r$ (for $j=1$) can be removed by the semigroup property.

We next show \eqref{wei-small1}. 
We intend to prove merely the case $(j,r)=(1,q)$ since the other cases are verified easily
(when $r>q$, the case $j=1$ is reduced to the case $j=0$ by the semigroup property).
Given $\alpha\in [0,3)$ and $q\in (\frac{3}{3-\alpha},\infty)$,
let us introduce the space
\begin{equation}
Z_{q,\alpha}(D):=\{u\in L^q_\sigma(D);\, \nabla u\in L^q_\alpha(D)\}
\label{Z-ge}
\end{equation}
endowed with norm 
$\|u\|_{Z_{q,\alpha}(D)}=\|u\|_q+\|\rho^\alpha\nabla u\|_q$
and show that
\begin{equation}
\|T(t,s)f\|_{Z_{q,\alpha}(D)}\leq C(t-s)^{-1/2+\delta}\|f\|_{Z_{q,\alpha}(D)}
\label{gZZest}
\end{equation}
for all 
$f\in Z_{q,\alpha}(D)$,
where $\delta\in (0,1/2q)$. 
We notice that $Z_{q,1}(D)=Z_q(D)$, see \eqref{Z}, and that \eqref{gZZest} with $\alpha=1$ is found in 
\cite[proof of Lemma 5.4]{Hi20}.
Once we have \eqref{gZZest}, we are led to the desired behavior; in fact,
since $C_{0,\sigma}^\infty(D)$ 
is dense in $L^q_{\alpha,\sigma}(D)$ for $\alpha$ and $q$ specified in this theorem,
the approximation procedure together with \eqref{wei-est1} implies that
\[
\lim_{t\to s}\,(t-s)^{1/2}\|\rho^\alpha \nabla T(t,s)f\|_q=0
\]
for every $f\in L^q_{\alpha,\sigma}(D)$. 
By \eqref{appro-est} with $(\alpha,j,r)=(0,0,q)$ and \eqref{appro-est-wei} we know
\[
\|W(t,s)f\|_{Z_{q,\alpha}(D)}\leq C(t-s)^{-1/2+\delta}\|f\|_{Z_{q,\alpha}(D)},
\]
while it follows from \eqref{remain-wei} with $\alpha=0$ and the latter of \eqref{remain-gr-wei} that
\[
\|PK(t,s)f\|_{Z_{q,\alpha}(D)}\leq C(t-s)^{-1+\delta}\|f\|_{W^{1,q}(D)},
\]
if, in particular, $\delta$ is chosen as $\delta<\min\{1/2q,\, (1-1/q)/2\}$.
With those estimates at hand, we apply Lemma \ref{lem-ite} with
\begin{equation*}
E_0=W, \quad Q=PK, \quad X_1=X_2=Z_{q,\alpha}(D), \quad
\beta=\frac{1}{2}-\delta, \quad \gamma=1-\delta
\end{equation*}
to obtain \eqref{gZZest}.

Finally, we verify \eqref{IC-1}. 
Since $C_{0,\sigma}^\infty(D)$ is dense in $L^q_{\alpha,\sigma}(D)$ and since
$T(t,s)$ is bounded near $t=s$ from $L^q_{\alpha,\sigma}(D)$
to itself, see \eqref{wei-est1} with $(j,r)=(0,q)$,
it suffices to show \eqref{IC-1} when $f\in C_{0,\sigma}^\infty(D)$.
By the item 3 of Proposition \ref{so-far},
for such data, we have the relation
\[
T(t,s)f-f=-\int_s^t \partial_\sigma T(t,\sigma)f\, d\sigma
=-\int_s^t T(t,\sigma)L(\sigma)f\,d\sigma,
\]
from which together with \eqref{wei-est1} 
and \eqref{FK-wei} it follows that
\begin{equation*}
\|\rho^\alpha \big(T(t,s)f-f\big)\|_q
\leq C\int_s^t\|\rho^\alpha L(\sigma)f\|_q\,d\sigma 
\leq C\int_s^t\|\rho^\alpha{\mathcal L}(\sigma)f\|_q\,d\sigma,
\end{equation*}
where ${\mathcal L}(t)$ is the differential operator given by \eqref{pre-proj-0}.
In this way, we obtain
\[
\|\rho^\alpha \big(T(t,s)f-f\big)\|_q
\leq C(t-s)\sup_{0\leq \sigma\leq {\mathcal T}}\|\rho^\alpha {\mathcal L}(\sigma)f\|_q
\to 0\qquad (t\to s)
\]
for every $f\in C_{0,\sigma}^\infty(D)$.
The proof is complete.
\hfill
$\Box$

\section{Strong solution}
\label{sect-strong}

This section improves the item 2 of Proposition \ref{so-far}
to show that the evolution operator $T(t,s)$ gives us a strong solution as long as the initial velocity is taken from 
$L^q_{1,\sigma}(D)$, $q\in (\frac{3}{2},\infty)$.

\medskip
\noindent
{\it Proof of Theorem \ref{strong-sol}}.
We begin with the proof of the item 1.
Let $q\in (1,\infty)$.
It suffices to show
\begin{equation}
\|\nabla^2 T(t,s)f\|_q\leq C(t-s)^{-1}\|f\|_q
\label{2nd-est}
\end{equation}
for all $f\in L^q_\sigma(D)$, which combined with \eqref{sm} with $j=0$ leads to \eqref{2nd-sm} by the semigroup property.
Given 
$f\in L^q_\sigma(D)$,
we set $u(t)=T(t,s)f$, which belongs to $W^{1,q}(D)$ for $t>s$ by \eqref{sm}. 
By interpolation with \eqref{sm} we observe 
\[
\|u(t)\|_{[L^q_\sigma(D),\, L^q_\sigma(D)\cap W^{1,q}(D)]_{2\delta}}\leq C(t-s)^{-\delta}\|f\|_q
\]
provided that $\delta\in (0,1/2)$.
If, in particular, $\delta\in (0,1/2q)$, then the complex interpolation space in the left-hand side above
coincides with $D_q(A^\delta)$, which does not involve any boundary condition (except the vanishing normal trace) at $\partial D$,
see \cite{Fu} and  \cite[Section 2.3]{NS03} as well as \eqref{frac-ci}.
We thus obtain
\begin{equation}
\|u(t)\|_{D_q(A^\delta)}
\leq C(t-s)^{-\delta}\|f\|_q
\label{combi-1}
\end{equation}
for all $f\in L^q_\sigma(D)$ as long as $\delta\in (0,1/2q)$.

We next assume that $f\in D_q(A^\delta)$ for some $\delta\in (0,1)$, which is however restricted to $\delta\in (0,1/2q)$ later.
It then follows from Lemma \ref{lem-frac} that 
the functions $f_0$ and $f_1$ defined by \eqref{modi} satisfy
\[
f_0\in H^{2\delta}_q(\mathbb R^3)\cap L^q_\sigma(\mathbb R^3), \qquad
f_1\in D_q(A_7^\delta)
\]
subject to \eqref{est-data-0} with $R=7$.
By \eqref{wei-wh} and \eqref{wei-wh2} with $(\alpha,r)=(0,q)$ we find
\begin{equation}
\|\nabla^2U(t,s)f_0\|_{q,\mathbb R^3}\leq C(t-s)^{-1+\delta}\|f_0\|_{H^{2\delta}_q(\mathbb R^3)}
\label{wh-delta}
\end{equation}
for $\delta\in [0,\frac{1}{2})$
since
$H^{2\delta}_q(\mathbb R^3)=[L^q(\mathbb R^3),\, W^{1,q}(\mathbb R^3)]_{2\delta}$.
This together with \eqref{frac-bdd} and \eqref{est-data-0} with $R=7$
implies that $W(t,s)f$ given by \eqref{appro} enjoys
\begin{equation}
\|W(t,s)f\|_{W^{2,q}(D)}\leq C(t-s)^{-1+\delta}\|f\|_{D_q(A^\delta)}
\label{appro-pre2}
\end{equation}
for $\delta\in [0,\frac{1}{2})$.

We now assume $\delta \in (0,1/2q)$, then we have 
$PK(t,s)f\in D_q(A^\delta)=[L^q_\sigma(D),\,L^q_\sigma(D)\cap W^{1,q}(D)]_{2\delta}$, where $K(t,s)f$ is the remainder term
given by \eqref{remainder}.
Since the projection $P$ is bounded on $W^{1,q}(D)$,
we deduce from \eqref{remain-est}
and the former of \eqref{remain-gr-est} that
\begin{equation}
\|PK(t,s)f\|_{D_q(A^\delta)}
\leq \|PK(t,s)f\|_{W^{1,q}(D)}^{2\delta}\|PK(t,s)f\|_q^{1-2\delta} 
\leq C(t-s)^{-1+\zeta}\|f\|_q
\label{delta-est}
\end{equation}
where
\begin{equation}
\zeta:=\frac{1-2\delta}{2}\left(1-\frac{1}{q}\right).
\label{zeta}
\end{equation}
Note that the restriction above on $\delta$ is needed since 
$PK(t,s)f$ does not fulfill the Dirichlet boundary condition.
In view of \eqref{iteration}, \eqref{appro-pre2} and \eqref{delta-est},
applying Lemma \ref{lem-ite} with
\begin{equation*}
E_0=W, \quad Q=PK, \quad X_1=D_q(A^\delta), 
\quad X_2=W^{2,q}(D), \quad
\beta=1-\delta, \quad \gamma=1-\zeta 
\end{equation*}
leads us to
\begin{equation}
\|T(t,s)f\|_{W^{2,q}(D)}
\leq C(t-s)^{-1+\delta}\|f\|_{D_q(A^\delta)}
\label{combi-2}
\end{equation}
for all $f\in D_q(A^\delta)$ as long as $\delta\in (0,1/2q)$.
Let us combine \eqref{combi-2} with \eqref{combi-1} to furnish
\begin{equation*}
\|T(t,s)f\|_{W^{2,q}(D)}
\leq C(t-s)^{-1+\delta}\|T((t+s)/2,s)f\|_{D_q(A^\delta)}
\leq C(t-s)^{-1}\|f\|_q
\end{equation*}
which implies \eqref{2nd-est}.

Let $q\in (3/2,\infty)$ and $f\in C_{0,\sigma}^\infty(D)\subset Z_q(D)$, then 
it follows from Proposition \ref{so-far} that
$T(t,s)f\in Y_q(D)\subset D_q(A)$ for all $t>s$.
By \eqref{2nd-est} we deduce \eqref{A-evo}
for all $f\in C_{0,\sigma}^\infty(D)$.
With this at hand, one can conclude \eqref{A-evo} for general $f\in L^q_\sigma(D)$ since $A$ is closed.

Assuming still $q\in (3/2,\infty)$,
we know from Proposition \ref{so-far} that $u(t)=T(t,s)f$ is indeed a strong solution to \eqref{evo}
in $L^q_\sigma(D)$ when $f\in Z_q(D)$.
We immediately see from Theorem \ref{weighted-est} with the aid of the semigroup property that
it is also the case even for $f\in L^q_{1,\sigma}(D)$, 
however, \eqref{Y-Z} and \eqref{wei-est1} imply worse smoothing rate
$\|\partial_tu(t)\|_q\leq C(t-s)^{-3/2+\delta}\|\rho f\|_q$
near $t=s$ than desired, where $\delta\in (0,1/2q)$.
But one can improve this rate as \eqref{str-est}; in fact,
by \eqref{wei-est1} and \eqref{A-evo} we find
\begin{equation*}
\begin{split}
\|\partial_tT(t,s)f\|_q
=\|L(t)T(t,s)f\|_q 
&\leq C\|T(t,s)f\|_{Y_q(D)}  \\
&\leq C(t-s)^{-1}\|f\|_q+C(t-s)^{-1/2}\|\rho f\|_q.
\end{split}
\end{equation*}
Likewise, \eqref{str-local-est} is also verified
because the projection $P$ is not needed
in the drift term of \eqref{ig}.
Finally, estimates of the pressure follow from the equation.
This completes the proof of Theorem \ref{strong-sol}.
\hfill
$\Box$

\section{H\"older estimate}
\label{sect-hoe}

In this section we deduce the H\"older estimates of the evolution operator $T(t,s)$ in $t$ and then those of the function 
$v(t)$ given by \eqref{duha} under the suitable assumption on the forcing term $g$.
\bigskip

\noindent
{\it Proof of Theorem \ref{hoelder}}.
We set $\beta=(3/q-3/r)/2$ for simplicity of notation and 
assume
$f\in L^q_{1,\sigma}(D)$, $q\in (\frac{3}{2},\infty)$.
Let $0\leq s<\tau<t\leq {\mathcal T}$.
When $j=0$ and $r\in [q,\infty)$, the proof is easy; indeed, combining \eqref{str-est} with \eqref{wei-est1} implies that
\begin{equation}
\begin{split}
&\quad \|T(t,s)f-T(\tau,s)f\|_r  \\
&\leq\int_\tau^t\|\partial_\sigma T(\sigma,s)f\|_r\,d\sigma  \\
&\leq C\int_\tau^t \Big\{(\sigma-s)^{-1}\|T((\sigma+s)/2,s)f\|_r+(\sigma-s)^{-1/2}
\|\rho T((\sigma+s)/2,s)f\|_r\Big\}\,d\sigma  \\
&\leq C\int_\tau^t \Big\{(\sigma-s)^{-\beta-1}\|f\|_q+(\sigma-s)^{-\beta-1/2}
\|\rho f\|_q\Big\} \,d\sigma  \\
&\leq C(t-\tau)\Big\{(\tau-s)^{-\beta-1}\|f\|_q+(\tau-s)^{-\beta-1/2}\|\rho f\|_q\Big\}.
\end{split}
\label{hoe-easy}
\end{equation}
On the other hand, we have
\begin{equation}
\|\nabla^j T(t,s)f-\nabla^j T(\tau,s)f\|_r
\leq C(\tau-s)^{-j/2-\beta}\|f\|_q
\label{rough}
\end{equation}
for $j\in \{0,1\}$ by \eqref{sm}.
Given $\mu\in (0,1]$, we compute the product
$\mbox{\eqref{hoe-easy}}^{\mu}\mbox{\eqref{rough}}^{1-\mu}_{j=0}$
to conclude 
\begin{equation}
\|T(t,s)f-T(\tau,s)f\|_r\leq C(t-\tau)^\mu\Big\{
(\tau-s)^{-\beta-\mu}\|f\|_q+(\tau-s)^{-\beta-\mu/2}\|\rho f\|_q\Big\}
\label{hoe-est1}
\end{equation}
which yields \eqref{hoe-est} with $j=0$ and $r\in [q,\infty)$.

The other cases are rather nontrivial because one can estimate neither
$\|\partial_\sigma T(\sigma,s)f\|_\infty$ nor $\|\nabla\partial_\sigma T(\sigma,s)f\|_r$.
The proof for those cases is based on
\begin{equation}
\begin{split}
\nabla^jT(t,s)f-\nabla^jT(\tau,s)f
&=\nabla^j(T(t,\tau)-I)T(\tau,s)f  \\
&=-\int_\tau^t \nabla^j \partial_\sigma T(t,\sigma)T(\tau,s)f\,d\sigma  \\
&=-\int_\tau^t \nabla^j T(t,\sigma)L(\sigma)T(\tau,s)f\,d\sigma.
\end{split}
\label{hoe-based}
\end{equation}
Note that $T(\tau,s)f\in Y_q(D)\subset D(L(\sigma))$ for every $\sigma\in (\tau,t)$ since $f\in L^q_{1,\sigma}(D)$, see 
Theorem \ref{strong-sol} together with the item 3 of Proposition \ref{so-far}.
Let $j=1$ and $r\in [q,\infty)$, then it follows from \eqref{sm}, \eqref{wei-est1} and \eqref{2nd-sm} that
\begin{equation}
\begin{split}
&\quad \quad \|\nabla T(t,s)f-\nabla T(\tau,s)f\|_r   \\
&\leq C\int_\tau^t (t-\sigma)^{-1/2}
\|L(\sigma)T(\tau,s)f\|_r\,d\sigma   \\
&\leq C(t-\tau)^{1/2}
\Big\{(\tau-s)^{-1-\beta}\|f\|_q+(\tau-s)^{-1/2-\beta}\|\rho f\|_q\Big\}.
\end{split}
\label{hoe-split}
\end{equation}
We thus compute 
$\mbox{\eqref{hoe-split}}^{2\mu}\mbox{\eqref{rough}}^{1-2\mu}_{j=1}$ to furnish 
\begin{equation}
\|\nabla T(t,s)f-\nabla T(\tau,s)f\|_r
\leq C(t-\tau)^\mu \Big\{(\tau-s)^{-1/2-\beta-\mu}\|f\|_q+(\tau-s)^{-1/2-\beta}\|\rho f\|_q\Big\}
\label{hoe-est2}
\end{equation}
for every $\mu\in (0,1/2]$, which leads to
\eqref{hoe-est} with $j=1$ and $r\in [q,\infty)$.

It remains to show \eqref{hoe-est} with $r=\infty$.
Given $\mu\in (0,1/2)$ with $\mu\geq \frac{1}{2}-\frac{3}{2q}$, we set $p=3/(1-2\mu)$, then we have $p\geq q$.
In view of \eqref{sm} with $(j,r)=(1,\infty)$, that is covered by \eqref{wei-est1} as well, 
we utilize \eqref{hoe-based} and follow the computation 
as in \eqref{hoe-split} to infer
\begin{equation}
\begin{split}
&\quad \|\nabla T(t,s)f-\nabla T(\tau,s)f\|_\infty   \\
&\leq C\int_\tau^t (t-\sigma)^{-1+\mu}\|L(\sigma)T(\tau,s)f\|_p\,d\sigma  \\
&\leq C(t-\tau)^\mu \Big\{(\tau-s)^{-1/2-\beta-\mu}\|f\|_q+(\tau-s)^{-\beta-\mu}\|\rho f\|_q\Big\}
\end{split}
\label{hoe-est3}
\end{equation}
with $\beta=3/2q$.
By the similar fashion we find
\begin{equation}
\quad \|T(t,s)f-T(\tau,s)f\|_\infty 
\leq C(t-\tau)^\mu \Big\{(\tau-s)^{-\beta-\mu}\|f\|_q+(\tau-s)^{-\beta+1/2-\mu}\|\rho f\|_q\Big\}
\label{hoe-est4}
\end{equation}
for given $\mu\in (0,1)$ with $\mu\geq 1-\frac{3}{2q}$.
For $0<\mu<1-\frac{j}{2}-\frac{3}{2q}$ with $j\in \{0,1\}$,
we have only to combine \eqref{hoe-est3}--\eqref{hoe-est4} with \eqref{rough} as we did.
This completes the proof of Theorem \ref{hoelder}.
\hfill
$\Box$
\bigskip

\noindent
{\it Proof of Corollary \ref{duha-hoelder}}.
Let us keep every notation in the proof of Theorem \ref{hoelder}, however,
we take $\mu$ 
as in \eqref{duha-hoe-exp}.
We apply \eqref{hoe-est} and \eqref{sm} to
\[
\nabla^j v(t)-\nabla^j v(\tau)
=\int_s^\tau \nabla^j \big(T(t,\sigma)-T(\tau,\sigma)\big)g(\sigma)\,d\sigma
+\int_\tau^t\nabla^j T(t,\sigma)g(\sigma)\,d\sigma
=:I+J,
\]
then we have
\begin{equation*}
\|I\|_r\leq C\int_s^\tau (t-\tau)^\mu (\tau-\sigma)^{-j/2-\beta-\mu}
\|\rho g(\sigma)\|_q\, d\sigma  \\
\leq C(t-\tau)^\mu (\tau-s)^{-j/2-\beta-\mu+1-\kappa}[g]_{q,1,\kappa}
\end{equation*}
and
\begin{equation*}
\|J\|_r \leq
C\int_\tau^t (t-\sigma)^{-j/2-\beta}\|g(\sigma)\|_q \,d\sigma
\leq C(t-\tau)^{1-j/2-\beta}(\tau-s)^{-\kappa}[g]_{q,0,\kappa}
\end{equation*}
on account of \eqref{ass1-force}.
These estimates at once imply \eqref{duha-hoe-est}.
\hfill
$\Box$

\section{Regularity of the Duhamel term}
\label{sect-duha}

This section is devoted to the proof of Theorem \ref{duha-strong}.
To this end, we begin with a sketch of the idea in the next several paragraphs.
In view of the construction of the evolution operator \eqref{evo-series}, 
the dominant part of singularity $(t-s)^{-1}$ 
near $t=s$ in \eqref{str-est}
comes from the operator $T_0(t,s)=W(t,s)$ given by \eqref{appro}.
Having this in mind, let us split the evolution operator into
\begin{equation}
T(t,s)=W(t,s)+S(t,s)
\label{T-split}
\end{equation}
through the scheme \eqref{iteration},
where
\begin{equation}
S(t,s)f=\sum_{j=1}^\infty T_j(t,s)f
\label{S-series}
\end{equation}
for $f\in L^q_\sigma(D)$.

Let $q\in (3/2,\infty)$.
For every $f\in L^q_{1,\sigma}(D)$, we have the item 3 of Theorem \ref{strong-sol}.
We also
know that $W(t,s)f$
possesses the same regularity as in \eqref{st-cl} for $T(t,s)f$ 
and fulfills the equation
\eqref{W-eq}; thus, so does $S(t,s)f$, along with
\begin{equation}
\partial_tS(t,s)f+L(t)S(t,s)f=PK(t,s)f, \qquad t\in (s,\infty),
\label{W-S}
\end{equation}
in $L^q_\sigma(D)$.

Let us consider the operator $T_1(t,s)$ defined by \eqref{iteration} with $j=0$.
From \eqref{frac-bdd}, \eqref{str-est-wh},
\eqref{data-est}, \eqref{appro-est} with $(\alpha,j,r)=(1,1,q)$ and
\eqref{appro-pre2}, it follows that
\begin{equation}
\begin{split}
&\|W(t,s)g\|_{Y_q(D)}\leq C(t-s)^{-1+\delta}\|g\|_{D_q(A^\delta)}+C(t-s)^{-1/2}\|\rho g\|_q,  \\
&\|\partial_tW(t,s)h\|_q\leq C(t-s)^{-1}\|h\|_q+C(t-s)^{-1/2}\|\rho h\|_q,
\end{split}
\label{appro-est2-0}
\end{equation}
for all $g\in D_q(A^\delta)\cap L^q_{1,\sigma}(D)$ with $\delta\in [0,1/2)$ and $h\in L^q_{1,\sigma}(D)$.
We choose $\delta\in (0,1/2q)$, and then
gather \eqref{appro-est2-0} with $g=PK(\tau,s)f$, \eqref{delta-est} and \eqref{remain-wei} to obtain
\begin{equation}
\begin{split}
\|T_1(t,s)f\|_{Y_q(D)}
&\leq C\int_s^t (t-\tau)^{-1+\delta}
\left(\|PK(\tau,s)f\|_{D_q(A^\delta)}+\|\rho PK(\tau,s)f\|_q\right)\,d\tau  \\
&\leq C\int_s^t (t-\tau)^{-1+\delta}(\tau-s)^{-1+\zeta}\,d\tau\;
\|f\|_q \\
&=C(t-s)^{-1+\delta+\zeta}\|f\|_q
\end{split}
\label{T1-est}
\end{equation}
for every $f\in L^q_\sigma(D)$,
where $\zeta$ is given by \eqref{zeta}.
In view of \eqref{iteration}, \eqref{T1-est} and \eqref{remain-wei} with $\alpha=0$,
we can apply Lemma \ref{lem-ite} with
\begin{equation*}
E_0=T_1, \quad Q=PK, \quad X_1=L^q_\sigma(D), \quad X_2=Y_q(D), \quad
\beta=1-\delta-\zeta, \quad \gamma=\frac{1}{2}\left(1+\frac{1}{q}\right)
\end{equation*}
to conclude
\begin{equation}
\begin{split}
&\|S(t,s)f\|_{Y_q(D)}
\leq C(t-s)^{-1+\delta+\zeta}\|f\|_q,  \\
&\|\partial_tS(t,s)g\|_q\leq C(t-s)^{-(1+1/q)/2}\|g\|_q,
\end{split}
\label{S-est}
\end{equation}
for all $f\in L^q_\sigma(D)$ and $g\in L^q_{1,\sigma}(D)$.
The latter estimate follows from the former one
together with \eqref{remain-est} on account of \eqref{W-S}, where we
note that
\begin{equation}
1-\zeta>
(1+1/q)/2>1-\delta-\zeta=(1+1/q)/2-\delta/q
\label{de-ka}
\end{equation}
by \eqref{zeta}.
Thus, we have actually less singular behavior near $t=s$ even without any weighted norm of data in \eqref{S-est} and,
therefore, $S(t,s)$ must be harmless.
When $q\in (3,\infty)$, the splitting \eqref{T-split}--\eqref{S-series} of the evolution operator would not be needed, 
see Remark \ref{rem-duha-alt} below,
however, we should study the regularity of the Duhamel term \eqref{duha} under the reasonable condition $q\in (3/2,\infty)$.
For this aim, it is better to split the evolution operator and to focus on the analysis of the dominant part $W(t,s)$.

Now, let us consider the function $v(t)$ given by \eqref{duha}
with $g(t)$ satisfying \eqref{ass2-force}--\eqref{ass3-force}.
We look into the detailed structute of $W(t,s)$ to 
prove the desired regularity \eqref{duha-cl} together with the representation
\begin{equation}
\begin{split}
\partial_tv(t)=&\; g(t)
+\int_s^t \partial_tS(t,\sigma)g(\sigma)\,d\sigma 
+\int_s^{(s+t)/2}\partial_tW(t,\sigma)g(\sigma)\,d\sigma  \\
&+\int_{(s+t)/2}^t(1-\phi)\partial_tU(t,\sigma)\big(g_0(\sigma)-g_0(t)\big)\,d\sigma \\
&-\int_{(s+t)/2}^t (1-\phi)L_{\mathbb R^3}(t)\left(U(t,\sigma)-e^{-(t-\sigma)L_{\mathbb R^3}(t)}\right)
g_0(t)\,d\sigma  \\
&-(1-\phi)\left(I-e^{-\frac{t-s}{2}L_{\mathbb R^3}(t)}\right)g_0(t)  \\
&+\int_{(s+t)/2}^t \phi\partial_tV(t,\sigma)\big(g_1(\sigma)-g_1(t)\big)\,d\sigma \\
&-\int_{(s+t)/2}^t \phi L_7(t)\left(V(t,\sigma)-e^{-(t-\sigma)L_7(t)}\right)
g_1(t)\,d\sigma \\
&-\phi\left(I-e^{-\frac{t-s}{2}L_7(t)}\right)g_1(t)  \\
&+\int_{(s+t)/2}^t \partial_t\mathbb B
\left[\Big(U(t,\sigma)g_0(\sigma)-V(t,\sigma)g_1(\sigma)\Big)\cdot\nabla\phi\right]\,d\sigma,
\end{split}
\label{repre-C1}
\end{equation}
where $L_{\mathbb R^3}(t)$ is the generator of $U(t,s)$ on $\mathbb R^3$, see \eqref{ig-wh}, and
$L_7(t)$ given by \eqref{ig-bdd} with $R=7$ is the generator of $V(t,s)$ on the bounded domain $D_7$, while
$g_0(t)$ and $g_1(t)$ are the modifications of $g(t)$ as in \eqref{modi}.
Since $g=(1-\phi)g_0+\phi g_1$, the formula \eqref{repre-C1}, that we are going to show,
is rewritten as
\begin{equation}
\begin{split}
\partial_tv(t)=
&\int_s^t \partial_tS(t,\sigma)g(\sigma)\,d\sigma
+\int_s^{(s+t)/2}\partial_tW(t,\sigma)g(\sigma)\,d\sigma   \\
&+\int_{(s+t)/2}^t(1-\phi)\partial_tU(t,\sigma)\big(g_0(\sigma)-g_0(t)\big)\,d\sigma \\
&-\int_{(s+t)/2}^t (1-\phi)L_{\mathbb R^3}(t)\left(U(t,\sigma)-e^{-(t-\sigma)L_{\mathbb R^3}(t)}\right)
g_0(t)\,d\sigma  \\
&+(1-\phi) e^{-\frac{t-s}{2}L_{\mathbb R^3}(t)}g_0(t)  \\
&+\int_{(s+t)/2}^t \phi\partial_tV(t,\sigma)\big(g_1(\sigma)-g_1(t)\big)\,d\sigma \\
&-\int_{(s+t)/2}^t \phi L_7(t)\left(V(t,\sigma)-e^{-(t-\sigma)L_7(t)}\right)
g_1(t)\,d\sigma \\
&+\phi e^{-\frac{t-s}{2}L_7(t)}g_1(t)  \\
&+\int_{(s+t)/2}^t \partial_t\mathbb B
\left[\Big(U(t,\sigma)g_0(\sigma)-V(t,\sigma)g_1(\sigma)\Big)\cdot\nabla\phi\right]\,d\sigma.
\end{split}
\label{repre-C1-alt}
\end{equation}

Moreover, we will justify the representation
\begin{equation}
\begin{split}
L(t)v(t)=&\int_s^t L(t)S(t,\sigma)g(\sigma)\,d\sigma  
+\int_s^{(s+t)/2}L(t)W(t,\sigma)g(\sigma)\,d\sigma   \\
&+\int_{(s+t)/2}^t P\left[(1-\phi)L_{\mathbb R^3}(t)U(t,\sigma)\big(g_0(\sigma)-g_0(t)\big)\right]\,d\sigma  \\
&+\int_{(s+t)/2}^tP\left[(1-\phi)L_{\mathbb R^3}(t)\left(U(t,\sigma)-e^{-(t-\sigma)L_{\mathbb R^3}(t)}\right)g_0(t)\right]\,d\sigma  \\
&+P\left[(1-\phi)\left(I-e^{-\frac{t-s}{2}L_{\mathbb R^3}(t)}\right)g_0(t)\right]  \\
&+\int_{(s+t)/2}^t P\left[\phi L_7(t)V(t,\sigma)\big(g_1(\sigma)-g_1(t)\big)\right]\,d\sigma  \\
&+\int_{(s+t)/2}^t P\left[\phi L_7(t)\left(V(t,\sigma)-e^{-(t-\sigma)L_7(t)}\right)g_1(t)\right]\,d\sigma  \\
&+P\left[\phi\left(I-e^{-\frac{t-s}{2}L_7(t)}\right)g_1(t)\right]  \\
&-\int_{(s+t)/2}^t P \Big(K(t,\sigma)g(\sigma)+\partial_t\mathbb B
\left[\Big(U(t,\sigma)g_0(\sigma)-V(t,\sigma)g_1(\sigma)\Big)\cdot\nabla\phi\right]\Big) \,d\sigma.
\end{split}
\label{repre-Lv}
\end{equation}
Taking into account $P[(1-\phi)g_0+\phi g_1]=Pg=g$, we rewrite \eqref{repre-Lv} as
\begin{equation}
\begin{split}
L(t)v(t)=&\; g(t)+\int_s^t L(t)S(t,\sigma)g(\sigma)\,d\sigma  
+\int_s^{(s+t)/2}L(t)W(t,\sigma)g(\sigma)\,d\sigma   \\
&+\int_{(s+t)/2}^t P\left[(1-\phi)L_{\mathbb R^3}(t)U(t,\sigma)\big(g_0(\sigma)-g_0(t)\big)\right]\,d\sigma  \\
&+\int_{(s+t)/2}^tP\left[(1-\phi)L_{\mathbb R^3}(t)\left(U(t,\sigma)-e^{-(t-\sigma)L_{\mathbb R^3}(t)}\right)g_0(t)\right]\,d\sigma  \\
&-P\left[(1-\phi) e^{-\frac{t-s}{2}L_{\mathbb R^3}(t)}g_0(t)\right]  \\
&+\int_{(s+t)/2}^t P\left[\phi L_7(t)V(t,\sigma)\big(g_1(\sigma)-g_1(t)\big)\right]\,d\sigma  \\
&+\int_{(s+t)/2}^t P\left[\phi L_7(t)\left(V(t,\sigma)-e^{-(t-\sigma)L_7(t)}\right)g_1(t)\right]\,d\sigma  \\
&-P\left[\phi e^{-\frac{t-s}{2}L_7(t)}g_1(t)\right]  \\
&-\int_{(s+t)/2}^t P\Big(K(t,\sigma)g(\sigma)+\partial_t\mathbb B
\left[\Big(U(t,\sigma)g_0(\sigma)-V(t,\sigma)g_1(\sigma)\Big)\cdot\nabla\phi\right]
\Big)\,d\sigma.
\end{split}
\label{repre-Lv-alt}
\end{equation}
Here, how one derives the representation \eqref{repre-Lv}--\eqref{repre-Lv-alt} is interpreted as follows.
We have to consider ${\mathcal L}(t)W(t,\sigma)g(\sigma)$ 
in the integrand above, where ${\mathcal L}(t)$ is the differential operator \eqref{pre-proj-0}.
Recalling how the remainder term 
\eqref{remainder} is obtained, we can describe
\begin{equation*}
\begin{split}
&\quad {\mathcal L}(t)W(t,\sigma)g(\sigma)  \\
&=(1-\phi){\mathcal L}(t)U(t,\sigma)g_0(\sigma)+\phi {\mathcal L}(t)V(t,\sigma)g_1(\sigma)  \\
&\quad -K(t,\sigma)g(\sigma)-(\nabla \phi)p_1-\partial_t\mathbb B
\left[\Big(U(t,\sigma)g_0(\sigma)-V(t,\sigma)g_1(\sigma)\Big)\cdot\nabla\phi\right],
\end{split}
\end{equation*}
where $p_1$ is the pressure associated with $V(t,\sigma)g_1(\sigma)$ on the bounded domain $D_7$.
Since
$(I-P_{D_7}){\mathcal L}(t)V(t,\sigma)g_1(\sigma)=-\nabla p_1$ with $P_{D_7}$ being the projection on $D_7$,
we have
\[
\phi{\mathcal L}(t)V(t,\sigma)g_1(\sigma)-(\nabla\phi)p_1
=\phi L_7(t)V(t,\sigma)g_1(\sigma)-\nabla(\phi p_1).
\]
Applying the projection $P$, we are led to
\begin{equation}
\begin{split}
&\quad L(t)W(t,\sigma)g(\sigma)  \\
&=P\left[(1-\phi)L_{\mathbb R^3}(t)U(t,\sigma)g_0(\sigma)\right]+
P\left[\phi L_7(t)V(t,\sigma)g_1(\sigma)\right]  \\
&\quad -P\Big( K(t,\sigma)g(\sigma)+\partial_t\mathbb B
\left[\Big(U(t,\sigma)g_0(\sigma)-V(t,\sigma)g_1(\sigma)\Big)\cdot\nabla\phi\right]\Big).
\end{split}
\label{auxi-deri}
\end{equation}
\begin{lemma}
Suppose that $\eta$ and $\omega$ fulfill \eqref{ass1-rigid} for some $\theta\in (0,1]$.
Let $q\in (\frac{3}{2},\infty)$ and assume that, given $s$ and ${\mathcal T}$ with
$0\leq s<{\mathcal T}<\infty$,
the function $g(t)$ fulfills \eqref{ass2-force}--\eqref{ass3-force} for some $\mu\in (0,1]$ and $\kappa\in [0,1)$.
Let $w_1(t)$ and $w_2(t)$ be the functions defined by
the right-hand sides of \eqref{repre-C1-alt} and \eqref{repre-Lv-alt}, respectively.
Then they are well-defined and satisfy the relation
\begin{equation}
Pw_1(t)+w_2(t)=g(t), \qquad t\in (s,{\mathcal T}]
\label{pre-inhomo}
\end{equation}
in $L^q_\sigma(D)$.

For each 
$m\in (0,\infty)$ and $R\in (1,\infty)$, there are constants
$C=C({\mathcal T},m,q,\mu,\kappa,\theta,D)>0$ and
$C^\prime=C^\prime({\mathcal T},m,q,R,\mu,\kappa,\theta,D)>0$ such that
\begin{equation}
\|w_1(t)\|_q+\|w_2(t)\|_q\leq C(t-s)^{-\kappa}\big([g]_{q,2,\kappa}+\{g\}_{q,\mu,\kappa}\big)
\label{pre-duha-str}
\end{equation}
\begin{equation}
\|w_1(t)\|_{q,D_R}+\|w_2(t)\|_{q,D_R}\leq C^\prime(t-s)^{-\kappa}\big([g]_{q,1,\kappa}+\{g\}_{q,\mu,\kappa}\big)
\label{pre-duha-str-loc}
\end{equation}
for all $t\in (s,{\mathcal T}]$ whenever \eqref{ever} is satisfied, where
$[g]_{q,\alpha,\kappa}$ and $\{g\}_{q,\mu,\kappa}$ are respectively
defined as \eqref{g-wei} and \eqref{g-hoe-semi}.

If, in particular, the generator $L(t)=L$ is independent of $t$, then we have \eqref{pre-duha-str} in which
$[g]_{q,2,\kappa}$ is replaced by $[g]_{q,1,\kappa}$ under less assumptions \eqref{ass1-force} and \eqref{ass3-force}.
\label{lem-before}
\end{lemma}

\begin{proof}
We make use of 
\eqref{bog-est2}, 
\eqref{pressure-bdd}--\eqref{evo-semi-bdd-0}, 
\eqref{frac-bdd} with $\delta=0$,
\eqref{weak-wh}, \eqref{str-est-wh},
\eqref{evo-semi-wh-0}--\eqref{evo-semi-loc-0},
\eqref{data-est},
\eqref{remain-est}, \eqref{appro-est2-0}
and \eqref{S-est}
together with the assumptions \eqref{ass2-force}--\eqref{ass3-force}
to deduce \eqref{pre-duha-str}--\eqref{pre-duha-str-loc}.
Among many terms, we here describe estimates of merely two terms in which we see the role of \eqref{ass3-force}
and the reason why $[g]_{q,2,\kappa}$ is needed:
\begin{equation*}
\begin{split}
&\quad \left\|\int_{(s+t)/2}^t(1-\phi)\partial_tU(t,\sigma)\big(g_0(\sigma)-g_0(t)\big)\,d\sigma\right\|_q   \\
&\leq C\int_{(s+t)/2}^t\Big[
(t-\sigma)^{-1}\|g(\sigma)-g(t)\|_q
+(t-\sigma)^{-1/2}\|\rho\big(g(\sigma)-g(t)\big)\|_q
\Big]\,d\sigma  \\
&\leq C\int_{(s+t)/2}^t (t-\sigma)^{-1+\mu}(\sigma-s)^{-\kappa-\mu}\,d\sigma \,\{g\}_{q,\mu,\kappa}
+C\int_{(s+t)/2}^t (t-\sigma)^{-1/2}\big\{(\sigma-s)^{-\kappa}+(t-s)^{-\kappa}\big\}\,d\sigma \,[g]_{q,1,\kappa}  \\
&\leq C(t-s)^{-\kappa}\{g\}_{q,\mu,\kappa}+C(t-s)^{-\kappa+1/2}[g]_{q,1,\kappa}
\end{split}
\end{equation*}
and
\begin{equation*}
\begin{split}
&\quad \left\|\int_{(s+t)/2}^t(1-\phi)L_{\mathbb R^3}(t)\big(U(t,\sigma)-e^{-(t-\sigma)L_{\mathbb R^3}(t)}\big)g_0(t)\,d\sigma\right\|_q   \\
&\leq C\int_{(s+t)/2}^t\Big[
(t-\sigma)^{-1/2+\theta}\|\rho g(t)\|_q
+(t-\sigma)^\theta\|\rho^2 g(t)\|_q
\Big]\,d\sigma  \\
&\leq C(t-s)^{-\kappa+1/2+\theta}[g]_{q,1,\kappa}+C(t-s)^{-\kappa+1+\theta}[g]_{q,2,\kappa}.
\end{split}
\end{equation*}
Estimates of the other terms are also derived readily.
By \eqref{W-eq} and \eqref{W-S} along with
\[
P\partial_tW(t,s)f=\partial_tW(t,s)f,\qquad
P\partial_tS(t,s)f=\partial_tS(t,s)f
\] 
for every $f\in L^q_{1,\sigma}(D)$, we furnish \eqref{pre-inhomo}.
Especially for the autonomous case, the terms that need \eqref{evo-semi-wh-0}, see the latter estimate above,
are absent in \eqref{repre-C1-alt} and \eqref{repre-Lv-alt}.
Hence, we have the last assertion.
\end{proof}
\begin{remark}
When $q\in (3,\infty)$, we have \eqref{wei-est1} with $\alpha=2$, which enables us to deduce the analogous estimate
of $\|L(t)\big(T(t,s)-e^{-(t-s)L(t)}\big)f\|_q$ for $f\in L^q_{2,\sigma}(D)$ to \eqref{evo-semi-wh-0} 
by the same argument as in the proof of Lemma \ref{lem-evo-semi}
with the aid of the similar estimate of 
$\|L(t)e^{-(t-s)L(t)}g\|_r$ for $g\in L^r_{1,\sigma}(D)$ with $r\in (3/2,\infty)$ to \eqref{semi-wh-0}.
As a consequence, 
the regularity \eqref{des-reg} below is available provided $q\in (3,\infty)$, subject to 
\begin{equation}
\begin{split}
\partial_tv(t)=
&\int_s^{(s+t)/2}\partial_tT(t,\sigma)g(\sigma)\,d\sigma
+\int_{(s+t)/2}^t \partial_tT(t,\sigma)\big(g(\sigma)-g(t)\big)\,d\sigma  \\
&-\int_{(s+t)/2}^tL(t)\big(T(t,\sigma)-e^{-(t-\sigma)L(t)}\big)g(t)\,d\sigma 
+e^{-\frac{t-s}{2}L(t)}g(t)
\end{split}
\label{other-repre-C1}
\end{equation}
and
\begin{equation}
\begin{split}
L(t)v(t)=
&\int_s^{(s+t)/2}L(t)T(t,\sigma)g(\sigma)\,d\sigma
+\int_{(s+t)/2}^t L(t)T(t,\sigma)\big(g(\sigma)-g(t)\big)\,d\sigma  \\
&+\int_{(s+t)/2}^tL(t)\big(T(t,\sigma)-e^{-(t-\sigma)L(t)}\big)g(t)\,d\sigma 
+\big(I-e^{-\frac{t-s}{2}L(t)}\big) g(t),
\end{split}
\label{other-repre-Lv}
\end{equation}
both of which make sense for such $q$.
These representations \eqref{other-repre-C1}--\eqref{other-repre-Lv} 
are simpler than \eqref{repre-C1-alt} and \eqref{repre-Lv-alt}.
The proof is also easier without splitting
\eqref{T-split}--\eqref{S-series} of the evolution operator.
Since \eqref{FK-wei} with $\alpha=2$ holds for $q>3$, one may expect the representations 
\eqref{other-repre-C1}--\eqref{other-repre-Lv} in many situations,
however, the condition $q\in (3,\infty)$ is not sharp to get the regularity \eqref{des-reg} below.
The advantage of our approach is that there is no further restriction on the summability exponent for
weighted estimates on the whole space $\mathbb R^3$, see Proposition \ref{weighted-wh}. 
\label{rem-duha-alt}
\end{remark}

Once we have the desired regularity 
\begin{equation}
v\in C^1((s,{\mathcal T}];\, L^q_\sigma(D)), \qquad
v(t)\in D_q(L(t))\quad\forall\,t\in (s,{\mathcal T}]
\label{des-reg}
\end{equation}
together with \eqref{repre-C1-alt} and \eqref{repre-Lv-alt},
we are immediately led to the equation \eqref{duha-eq} from \eqref{pre-inhomo} since $P\partial_tv(t)=\partial_tv(t)$.
By \eqref{des-reg}, \eqref{ass3-force} and \eqref{duha-eq} we see that $L(\cdot)v\in C((s,{\mathcal T}];\, L^q_\sigma(D))$.

If, in particular, the generator $L(t)=L$ is independent of $t$,
we have \eqref{ell-L} by Lemma \ref{lem-apri}, which
along with the continuity of $Lv$ in $t$, that we have just observed, and \eqref{duha-hoe-est} give
$v\in C((s,{\mathcal T}];\, W^{2,q}(D))$.

Moreover, we have
\begin{equation}
\|\partial_tv(t)\|_q+\|L(t)v(t)\|_q 
\leq C(t-s)^{-\kappa}
\big([g]_{q,2,\kappa}+\{g\}_{q,\mu,\kappa}\big)
\label{duha-str-noch}
\end{equation}
\begin{equation}
\|\partial_tv(t)\|_{q,D_R}+\|L(t)v(t)\|_{q,D_R}
\leq C(t-s)^{-\kappa}
\big([g]_{q,1,\kappa}+\{g\}_{q,\mu,\kappa}\big)
\label{duha-str-local}
\end{equation}
by \eqref{repre-C1-alt}, \eqref{repre-Lv-alt} and \eqref{pre-duha-str}--\eqref{pre-duha-str-loc}.
On the other hand, we use
\eqref{wei-est1} to find
\begin{equation*}
\|\rho\nabla v(t)\|_q\leq C\int_s^t (t-\sigma)^{-1/2}\|\rho g(\sigma)\|_q\,d\sigma,
\end{equation*}
yielding \eqref{duha-drift-est}.
This combined with \eqref{des-reg}--\eqref{duha-str-noch} 
implies that $v(t)\in Y_q(D)$ for all $t\in (s,{\mathcal T}]$, see \eqref{Y}, and that
\begin{equation}
\begin{split}
\|Av(t)\|_q
&=\|L(t)v(t)+(\eta(t)+\omega(t)\times x)\cdot\nabla v(t)-\omega(t)\times v(t)\|_q  \\
&\leq C(t-s)^{-\kappa}\big([g]_{q,2,\kappa}+\{g\}_{q,\mu,\kappa}\big)
\end{split}
\label{est-Av}
\end{equation}
with $A$ being the Stokes operator on the exterior domain $D$, see \eqref{stokes},
where we have the following less rate near $t=s$ for lower order terms 
\[
\|\nabla^jv(t)\|_q\leq C(t-s)^{-\kappa+1-j/2}[g]_{q,0,\kappa}, \qquad j\in \{0,1\}.
\]
Since $\| v\|_{W^{2,q}(D)}\leq C(\|Av\|_q+\|v\|_q)$, we get \eqref{duha-str-est} for $\|\nabla^2v(t)\|_q$.

If, in particular, the generator $L(t)=L$ is independent of $t$,
then $[g]_{q,2,\kappa}$ can be replaced by $[g]_{q,1,\kappa}$ in several estimates deduced in the preceding paragraph
under less assumptions \eqref{ass1-force} and \eqref{ass3-force}
on account of the last statement in Lemma \ref{lem-before}.

By the same manner as in \eqref{est-Av}, it follows from \eqref{duha-str-local} that
\begin{equation}
\begin{split}
\|Av(t)\|_{q,D_R}
&\leq \|L(t)v(t)\|_{q,D_R}+C\|\nabla v(t)\|_q+C\|v(t)\|_q  \\
&\leq 
C(t-s)^{-\kappa}\big([g]_{q,1,\kappa}+\{g\}_{q,\mu,\kappa}\big)
\end{split}
\label{Av-local}
\end{equation}
since the projection $P$ is not needed in the drift term of \eqref{ig}.
By virtue of \eqref{Av-local} with $D_R$ replaced by $D_{2R}$,
we make use of \eqref{St-loc-est-0} to conclude \eqref{duha-str-loc-est} for $\|\nabla^2v(t)\|_{q,D_R}$.

Let us end up with completion of the proof of Theorem \ref{duha-strong}.

\medskip
\noindent
{\it Proof of Theorem \ref{duha-strong}.}
Let $q\in (3/2,\infty)$.
We fix $s$ and ${\mathcal T}$ such that $0\leq s<{\mathcal T}<\infty$.
From what we have just observed above, our task is to prove \eqref{des-reg} along with \eqref{repre-C1-alt} and \eqref{repre-Lv-alt}.
Let us take any compact interval $J\subset (s,{\mathcal T}]$, then it suffices to show
\begin{equation}
v\in C^1(J;\,L^q_\sigma(D)), \qquad \partial_tv(t)=w_1(t)\quad\forall\, t\in J,
\label{reg-cpt1}
\end{equation}
\begin{equation}
v(t)\in D_q(L(t)), \qquad L(t)v(t)=w_2(t)\quad\forall\,t\in J,
\label{reg-cpt2}
\end{equation}
where $w_1(t)$ and $w_2(t)$ are the functions defined by the right-hand sides of \eqref{repre-C1-alt} and \eqref{repre-Lv-alt},
respectively.
Consider
\[
v_\varepsilon(t)=\int_s^{t-\varepsilon}T(t,\sigma)g(\sigma)\,d\sigma \qquad (t\in J),
\]
with $\varepsilon >0$ being small enough in such a way that $s<\frac{s+t}{2}<t-\varepsilon$ for all $t\in J$.
By the item 3 of Theorem \ref{strong-sol} we find
\[
v_\varepsilon \in C^1(J;\,L^q_\sigma(D)), \qquad
v_\varepsilon(t)\in 
D_q(L(t)) \quad\forall\, t\in J
\]
with
\[
\partial_t v_\varepsilon(t)
=T(t,t-\varepsilon)g(t-\varepsilon)
+\int_s^{t-\varepsilon}\big(\partial_tS(t,\sigma)+
\partial_tW(t,\sigma)\big)g(\sigma)\,d\sigma,
\]
\[
L(t)v_\varepsilon(t)=
\int_s^{t-\varepsilon}\big(L(t)S(t,\sigma)+
L(t)W(t,\sigma)\big)g(\sigma)\,d\sigma.
\]

In view of \eqref{appro} and by taking into account \eqref{auxi-deri}
one can compute
\begin{equation*}
\begin{split}
&\int_s^{t-\varepsilon}\partial_tW(t,\sigma)g(\sigma)\,d\sigma  \\
=&\int_s^{(s+t)/2}\partial_tW(t,\sigma)g(\sigma)\,d\sigma
+\int_{(s+t)/2}^{t-\varepsilon}(1-\phi)\partial_tU(t,\sigma)\big(g_0(\sigma)-g_0(t)\big)\,d\sigma \\
&-\int_{(s+t)/2}^{t-\varepsilon} (1-\phi)L_{\mathbb R^3}(t)\left(U(t,\sigma)-e^{-(t-\sigma)L_{\mathbb R^3}(t)}\right)
g_0(t)\,d\sigma  \\  
&-(1-\phi)\left(e^{-\varepsilon L_{\mathbb R^3}(t)}-e^{-\frac{t-s}{2}L_{\mathbb R^3}(t)}\right)g_0(t)  \\ 
&+\int_{(s+t)/2}^{t-\varepsilon} \phi\partial_tV(t,\sigma)\big(g_1(\sigma)-g_1(t)\big)\,d\sigma \\ 
&-\int_{(s+t)/2}^{t-\varepsilon} \phi L_7(t)\left(V(t,\sigma)-e^{-(t-\sigma)L_7(t)}\right) 
g_1(t)\,d\sigma \\ 
&-\phi\left(e^{-\varepsilon L_7(t)}-e^{-\frac{t-s}{2}L_7(t)}\right)g_1(t)  \\ 
&+\int_{(s+t)/2}^{t-\varepsilon} \partial_t\mathbb B 
\left[\Big(U(t,\sigma)g_0(\sigma)-V(t,\sigma)g_1(\sigma)\Big)\cdot\nabla\phi\right]\,d\sigma                                                                  
\end{split}
\end{equation*}
as well as
\begin{equation*}
\begin{split}
&\int_s^{t-\varepsilon} L(t)W(t,\sigma)g(\sigma)\,d\sigma  \\
=&\int_s^{(s+t)/2}L(t)W(t,\sigma)g(\sigma)\,d\sigma
+\int_{(s+t)/2}^{t-\varepsilon} P\left[(1-\phi)L_{\mathbb R^3}(t)U(t,\sigma)\big(g_0(\sigma)-g_0(t)\big)\right]\,d\sigma  \\ 
&+\int_{(s+t)/2}^{t-\varepsilon} 
P\left[(1-\phi)L_{\mathbb R^3}(t)\left(U(t,\sigma)-e^{-(t-\sigma)L_{\mathbb R^3}(t)}\right)g_0(t)\right]\,d\sigma  \\ 
&+P\left[(1-\phi)\left(e^{-\varepsilon L_{\mathbb R^3}(t)}-e^{-\frac{t-s}{2}L_{\mathbb R^3}(t)}\right)g_0(t)\right]  \\
&+\int_{(s+t)/2}^{t-\varepsilon} P\left[\phi L_7(t)V(t,\sigma)\big(g_1(\sigma)-g_1(t)\big)\right]\,d\sigma  \\
&+\int_{(s+t)/2}^{t-\varepsilon} P\left[\phi L_7(t)\left(V(t,\sigma)-e^{-(t-\sigma)L_7(t)}\right)g_1(t)\right]\,d\sigma  \\ 
&+P\left[\phi\left(e^{-\varepsilon L_7(t)}-e^{-\frac{t-s}{2}L_7(t)}\right)g_1(t)\right]  \\
&-\int_{(s+t)/2}^{t-\varepsilon} P\Big(K(t,\sigma)g(\sigma)+\partial_t\mathbb B
\left[\Big(U(t,\sigma)g_0(\sigma)-V(t,\sigma)g_1(\sigma)\Big)\cdot\nabla\phi\right]
\Big)\,d\sigma.
\end{split}
\end{equation*}

By 
\eqref{ass2-force}--\eqref{ass3-force}, 
\eqref{bog-est2},
\eqref{pressure-bdd}--\eqref{evo-semi-bdd-0}, \eqref{frac-bdd} with $\delta=0$,
\eqref{weak-wh}, \eqref{str-est-wh}, \eqref{evo-semi-wh-0}, 
\eqref{data-est},
\eqref{remain-est}, 
\eqref{S-est}--\eqref{de-ka} 
together with
\eqref{con-Tg}, \eqref{con-LRg} 
and \eqref{con-whg}, we deduce that,
with some $C>0$ dependent on $J$, however, independent of $\varepsilon$,
\begin{equation}
\begin{split}
&\quad \sup_{t\in J}\Big(\|\partial_t v_\varepsilon(t)-w_1(t)\|_q
+\|L(t)v_\varepsilon(t)-w_2(t)\|_q\Big)  \\
&\leq C\Big[
\varepsilon^\mu \{g\}_{q,\mu,\kappa}+
\varepsilon^{1+\theta}[g]_{q,2,\kappa}+
\varepsilon^{1/2}[g]_{q,1,\kappa}+
\big(\varepsilon^\theta+\varepsilon^{(1-1/q)/2}\big)[g]_{q,0,\kappa}
\Big]  \\
&\quad +\sup_{t\in J}\Big[
\|T(t,t-\varepsilon)g(t-\varepsilon)-g(t)\|_q
+\|e^{-\varepsilon L_{\mathbb R^3}(t)}g_0(t)-g_0(t)\|_{q,\mathbb R^3}
+\|e^{-\varepsilon L_7(t)}g_1(t)-g_1(t)\|_{q,D_7}
\Big]
\end{split}
\label{unif-conv}
\end{equation}
which goes to zero as $\varepsilon \to 0$.
Here, it is convenient to use \eqref{repre-C1} and \eqref{repre-Lv} rather than \eqref{repre-C1-alt} and \eqref{repre-Lv-alt},
respectively.
We thus conclude \eqref{reg-cpt1} and \eqref{reg-cpt2} since $L(t)$ is a closed operator.

If, in particular, the generator $L(t)=L$ is independent of $t$,
the term $\varepsilon^{1+\theta}[g]_{q,2,\kappa}$
arising from \eqref{evo-semi-wh-0} 
is absent in \eqref{unif-conv},
see also the description just before Lemma \ref{lem-evo-semi}.
Hence, the argument above works well without any change under less assumptions \eqref{ass1-force} and \eqref{ass3-force}.
The proof is complete.
\hfill
$\Box$

\section{Navier-Stokes flow}
\label{sect-ns}

Once we have fine linear theory, the proof of Theorem \ref{ns-thm} is standard as in the classic literature
\cite{FK64, Ka84}, nonetheless, we describe it for completeness.
In what follows we always assume that $\eta$ and $\omega$ fulfill \eqref{ass3-rigid} and \eqref{ass4-rigid}
for some $\vartheta\in (0,1]$ and $\gamma\in [0,1)$, and this will not be mentioned even in
the statement of propositions and so on.

Let us begin with the uniqueness of solutions, independently of the existence. 
\begin{proposition}
Suppose $v_0\in L^q_{\sigma}(D)$ for some $q\in (3,\infty)$.
Let ${\mathcal T}\in (0,\infty)$.
Then the function $v(t)$, which is of class
\begin{equation}
v\in L^\infty(0,{\mathcal T};\, L^q_\sigma(D)), \qquad t^{1/2}\nabla v\in L^\infty(0,{\mathcal T};\, L^q(D))
\label{uni-cl}
\end{equation}
and satisfies \eqref{ns-int} in $L^q_\sigma(D)$ on the interval $(0,{\mathcal T})$, is at most one.
\label{prop-uni}
\end{proposition}

\begin{proof}
We follow the argument by Fujita and Kato \cite[Theorem 3.1]{FK64}.
Suppose that $v_1$ and $v_2$ are solutions of the class \eqref{uni-cl} to \eqref{ns-int}. 
Then $w:=v_1-v_2$ obeys
\begin{align*}
w(t)=-\int_{0}^{t}T(t, \xi)P(w\cdot \nabla v_1+v_2\cdot \nabla w+b\cdot\nabla w+w\cdot  \nabla b)(\xi)\,d\xi,
\end{align*}
to which one applies \eqref{sm} to infer
\begin{align*}
\|w(t)\|_q+t^{\frac{1}{2}}\|\nabla w(t)\|_q
&\le CE(\mathcal T)H(t)\big(t^{\frac{1}{2}-\frac{3}{2q}}+t^{1-\frac{3}{2q}}\big)
\end{align*}
for $t\in (0,{\mathcal T})$, where we have set
\begin{align*}
&H(t):=\sup_{0< \xi< t}\left(\|w(\xi)\|_q+\xi^{\frac{1}{2}}\|\nabla w(\xi)\|_q\right),\\
&E({\mathcal T}):=\sup_{0<\xi <{\mathcal T}}\left(\xi^{\frac{1}{2}}\|\nabla v_1(\xi)\|_q+\|v_2(\xi)\|_q+\|b(\xi)\|_q+\|\nabla b(\xi)\|_q\right).
\end{align*}
Here and in what follows, we simply write $\sup$ instead of esssup.
We thus find $t_1\in (0,{\mathcal T})$ such that
$H(t_1)=0$ 
and, therefore, $w=0$ on $(0, t_1)$. 

At the next step, it turns out that
there exists $\delta>0$, independent of $s\in [t_1,{\mathcal T})$, with the following property: 
if $w=0$ on $(0, s)$, then $w=0$ on 
$(0, s+\delta)$ as long as $s+\delta<{\mathcal T}$, otherwise, $w=0$ on $(0,{\mathcal T})$. 
From this claim we are eventually led to $w=0$ on $(0,{\mathcal T})$ in finitely many steps.
In fact, assuming $w=0$ on $(0,s)$ and 
setting
\begin{align*}
&\widetilde{H}(t):=\sup_{s< \xi< t}\big(\|w(\xi)\|_q+\|\nabla w(\xi)\|_q\big), \\
&\widetilde{E}(\mathcal T):=\sup_{t_1< \xi<{\mathcal T}}\big(\|\nabla v_1(\xi)\|_q+\|v_2(\xi)\|_q+\|b(\xi)\|_q+\|\nabla b(\xi)\|_q\big),
\end{align*}
we easily observe
\[
\|w(t)\|_q+\|\nabla w(t)\|_q
\le C\widetilde{E}(\mathcal T)\widetilde{H}(t)\left((t-s)^{1-\frac{3}{2q}}+(t-s)^{\frac{1}{2}-\frac{3}{2q}}\right)
\]
for $t\in (s,{\mathcal T})$,
which implies $\widetilde H(s+\delta)=0$ for some $\delta>0$ independent of $s$. 
The proof is complete.
\end{proof}

We set
\begin{equation}
\begin{split}
&(\Psi v)(t)=\overline{v}(t)+(\Phi v)(t), \qquad
\overline{v}(t)=T(t,0)v_0+\int_0^tT(t,s)PF(s)\,ds, \\
&(\Phi v)(t)=-\int_0^tT(t,s)P(v\cdot\nabla v+b\cdot\nabla v+v\cdot\nabla b)(s)\,ds,
\end{split}
\label{sol-map}
\end{equation}
and intend to find a solution to \eqref{ns-int} through constructing a fixed point of the map $\Psi$
in a suitable closed ball of the Banach space
\[
X_{\mathcal T}=\{v\in C((0,{\mathcal T}];\, L^\infty(D));\; 
\|v\|_{X_{\mathcal T}}<\infty\}
\]
for some ${\mathcal T}\in (0,1]$, to be determined later,
endowed with norm
\begin{equation}
\|v\|_{X_{\mathcal T}}=\sup_{0<t\leq {\mathcal T}}t^{3/2q}\|v(t)\|_\infty+\sup_{0<t\leq {\mathcal T}}t^{1/2}\|\rho\nabla v(t)\|_q,
\label{funct-sp}
\end{equation}
where the weight function $\rho$ is given by \eqref{weight}.

Let us derive some estimates of $\overline{v}$, $\Phi v$ and $\Psi v$ for later use.
\begin{lemma}
Suppose $v_0\in L^q_{1,\sigma}(D)$ for some $q\in (3,\infty)$.
Let $r\in [q,\infty]$, $j\in \{0,1\}$ and 
${\mathcal T}\in (0,1]$. 
Then we have the following.
\begin{enumerate}
\item
There is a constant $C=C(m_0,q,r,\gamma,D)>0$ such that
\begin{equation}
\|\rho\nabla^j\overline{v}(t)\|_r\leq CK\,t^{-j/2-(3/q-3/r)/2} 
\label{top-est}
\end{equation}
for all $t\in (0,{\mathcal T}]$ with
\begin{equation}
K:=\|\rho v_0\|_q+m_0+m_0^2,
\label{K}
\end{equation}
where $m_0$ is given by \eqref{quan2}.
Moreover,
\begin{equation}
\lim_{t\to 0}t^{j/2+(3/q-3/r)/2}\|\rho\nabla^j\overline{v}(t)\|_r=0 
\label{top-beha}
\end{equation}
except the case $(j,r)=(0,q)$;
in that case, it holds that
\begin{equation}
\lim_{t\to 0}\|\rho\big(\overline{v}(t)-v_0\big)\|_q=0.
\label{top-ic}
\end{equation}

\item
There are constants
$c_k=c_k(m_0,q,r,\gamma,D)>0\;(k=1,\,2)$
such that
\begin{equation}
\|\rho\nabla^j (\Phi v)(t)\|_r\leq t^{-j/2-(3/q-3/r)/2}
\Big(c_1
{\mathcal T}^{1/2-3/2q}
\|v\|_{X_{\mathcal T}}^2
+c_2m_0
{\mathcal T}^{1/2}
\|v\|_{X_{\mathcal T}}\Big)
\label{Phi-est}
\end{equation}
for all $t\in (0,{\mathcal T}]$ and $v\in X_{\mathcal T}$.
Moreover,
\begin{equation}
\lim_{t\to 0}t^{j/2+(3/q-3/r)/2}\|\rho\nabla^j (\Phi v)(t)\|_r=0
\label{Phi-beha}
\end{equation}
for all $v\in X_{\mathcal T}$. 

\item
Given $\mu$ satisfying
\begin{equation}
\max\left\{\frac{3}{2q}+\frac{1}{2}-\frac{j}{2}-\beta,\; \gamma-\frac{j}{2}-\beta\right\}<\mu<1-\frac{j}{2}-\beta, \qquad
\beta:=\frac{3}{2}\left(\frac{1}{q}-\frac{1}{r}\right),
\label{Psi-hoe-exp}
\end{equation}
there is a constant $C=C(m_0,q,r,\mu,\gamma,D)>0$ such that
\begin{equation}
\|\nabla^j(\Psi v)(t)-\nabla^j(\Psi v)(\tau)\|_r
\leq C(t-\tau)^\mu\tau^{-j/2-\beta-\mu}\Big(K 
+{\mathcal T}^{1/2-3/2q}\|v\|_{X_{\mathcal T}}^2+m_0{\mathcal T}^{1/2}\|v\|_{X_{\mathcal T}}\Big) 
\label{Psi-hoe}
\end{equation}
for all $(t,\tau)$ with $0<\tau<t\leq {\mathcal T}$ and
$v\in X_{\mathcal T}$,
where $K$ is given by \eqref{K}.
\end{enumerate}
\label{lem-Phi}
\end{lemma}

\begin{proof}
We use several estimates of the evolution operator $T(t,s)$ with constants determined by \eqref{ever} with
${\mathcal T}=1,\, \theta=1-\gamma$ and $m=m_0/(1-\gamma)$, see \eqref{imp-hoe}--\eqref{quan2}.
By
\eqref{wei-est1}--\eqref{IC-1} and \eqref{force-est} with $\alpha=1$
we immediately see \eqref{top-est} and \eqref{top-beha}--\eqref{top-ic}.

We readily find from \eqref{FK-wei}, \eqref{wei-est1} and \eqref{lift-est} that
\begin{align*}
&\quad \|\rho\nabla^j (\Phi v)(t)\|_r   \\
&\leq C\int_{0}^{t}(t-\xi)^{-\frac{3}{2}\left(\frac{1}{q}-\frac{1}{r}\right)-\frac{j}{2}}
\big(\|v(\xi)\|_{\infty}\|\rho \nabla v(\xi)\|_q+\|b(\xi)\|_{\infty}\|\rho \nabla v(\xi)\|_q+\|v(\xi)\|_{\infty}\|\rho \nabla b(\xi)\|_q \big)\,d\xi \\
&\leq Ct^{-\frac{3}{2}\left(\frac{1}{q}-\frac{1}{r}\right)-\frac{j}{2}+\frac{1}{2}-\frac{3}{2q}}
\|v\|_{X_{\mathcal T}}^2
+Ct^{-\frac{3}{2}\left(\frac{1}{q}-\frac{1}{r}\right)-\frac{j}{2}}\,
\big(t^{\frac{1}{2}}+t^{1-\frac{3}{2q}}\big)m_0\|v\|_{X_{\mathcal T}}
\end{align*}
for all $t\in (0,{\mathcal T}]$,
which implies \eqref{Phi-est}--\eqref{Phi-beha}, where ${\mathcal T}^{1-3/2q}\leq {\mathcal T}^{1/2}$
for ${\mathcal T}\in (0,1]$ is taken into account for simplicity.

By \eqref{lift-est} and \eqref{force-est} we employ
Theorem \ref{hoelder} and Corollary \ref{duha-hoelder} with 
$\kappa=\gamma,\, \frac{3}{2q}+\frac{1}{2},\, \frac{1}{2}$ and $\frac{3}{2q}$, respectively,
to furnish \eqref{Psi-hoe}.
The proof is complete.
\end{proof}
\begin{proposition}
Suppose $v_0\in L^q_{1,\sigma}(D)$ for some $q\in (3,\infty)$.
There is ${\mathcal T}_0={\mathcal T}_0(\|\rho v_0\|_q, m_0, q, \gamma, D)\in (0,1]$ 
such that equation \eqref{ns-int}
admits a solution $v\in X_{{\mathcal T}_0}$ satisfying 
$\|v\|_{X_{{\mathcal T}_0}}\leq CK$ 
with some constant $C=C(m_0,q,\gamma,D)>0$, where $K$ is given by \eqref{K}.
\label{prop-exis}
\end{proposition}

\begin{proof}
Let ${\mathcal T}\in (0,1]$.
We know from
\eqref{Psi-hoe} that $\Psi v\in C((0,{\mathcal T}];\, L^\infty(D))$ for all $v\in X_{\mathcal T}$.
In view of \eqref{top-est} one can take a constant $N>0$ fulfilling
$\|\overline{v}\|_{X_1}\leq NK$ (with ${\mathcal T}=1$). 
It follows from \eqref{Phi-est} with $(j,r)=(0,\infty)$ and $(1,q)$ that $\Psi v\in X_{\mathcal T}$ subject to
\begin{equation}
\|\Psi v\|_{X_{\mathcal T}}
\leq NK+c_1^\prime{\mathcal T}^{1/2-3/2q}\|v\|_{X_{\mathcal T}}^2+c_2^\prime m_0{\mathcal T}^{1/2}\|v\|_{X_{\mathcal T}}
\label{into}
\end{equation}
for all $v\in X_{\mathcal T}$ with 
\[
c_k^\prime=c_k(m_0,q,\infty,\gamma,D)+c_k(m_0,q,q,\gamma,D) \qquad (k=1,\, 2),
\]
where $c_1$ and $c_2$ are the constants in \eqref{Phi-est}.
Exactly the same computation as in \eqref{Phi-est} gives
\begin{equation}
\quad \|\Psi v-\Psi w\|_{X_{\mathcal T}}
\leq\Big\{
c_1^\prime{\mathcal T}^{1/2-3/2q}
\big(\|v\|_{X_{\mathcal T}}+\|w\|_{X_{\mathcal T}}\big)
+c_2^\prime m_0{\mathcal T}^{1/2}\Big\}
\|v-w\|_{X_{\mathcal T}}
\label{contra}
\end{equation}
for all $v,\, w\in X_{\mathcal T}$
with the same constants $c_1^\prime$ and $c_2^\prime$ as in \eqref{into}.
We now choose ${\mathcal T}_0\in (0,1]$ such that
\begin{equation}
4c_1^\prime{\mathcal T}_0^{1/2-3/2q}NK+2c_2^\prime m_0{\mathcal T}_0^{1/2}\leq 1.
\label{short-t}
\end{equation}
By $X_{{\mathcal T}_0}(R)=\{v\in X_{{\mathcal T}_0};\, \|v\|_{{\mathcal T}_0}\leq R\}$ we denote the closed ball
centered at the origin with radius $R>0$ in $X_{{\mathcal T}_0}$.
Then \eqref{into}--\eqref{contra} with ${\mathcal T}={\mathcal T}_0$ implies that the map $\Psi$ 
is contractive from $X_{{\mathcal T}_0}(2NK)$ into itself.
The proof is complete.
\end{proof}

Let us close the paper with completion of the proof of Theorem \ref{ns-thm}.
\medskip

\noindent
{\it Proof of Theorem \ref{ns-thm}}.
Let $v(t)$ be the solution obtained in Proposition \ref{prop-exis}.
The behavior \eqref{ns-est1}
as well as the initial condition \eqref{ns-ic}
follows directly from \eqref{top-beha}--\eqref{top-ic} and \eqref{Phi-beha}.
Let $r\in [q,\infty]$ and $j\in \{0,1\}$.
Estimate $\|v\|_{X_{{\mathcal T}_0}}\leq 2NK$ together with \eqref{short-t} allows us to rewrite \eqref{Phi-est} and \eqref{Psi-hoe}
with ${\mathcal T}={\mathcal T}_0$ as
\begin{equation}
\|\rho\nabla^j (\Phi v)(t)\|_r\leq CK\,t^{-j/2-(3/q-3/r)/2}
\label{sol-est}
\end{equation}
for all $t\in (0,{\mathcal T}_0]$ and
\begin{equation}
\|\nabla^j v(t)-\nabla^jv(\tau)\|_r\leq CK(t-\tau)^\mu \tau^{-j/2-(3/q-3/r)/2-\mu}
\label{sol-hoe}
\end{equation}
for all $(t,\tau)$ with $0<\tau<t\leq {\mathcal T}_0$, where $\mu$ satisfies \eqref{Psi-hoe-exp}.
We immediately conclude \eqref{ns-est0} from \eqref{top-est} and \eqref{sol-est}. 

What remains to show is \eqref{ns-cl}
together with \eqref{ns-est2}.
As to $T(t,0)v_0$ with $v_0\in L^q_{1,\sigma}(D)$, 
the desired estimate is obvious due to the item 3 of Theorem \ref{strong-sol}.
Thanks to Theorem \ref{duha-strong}, 
it thus suffices to verify the conditions \eqref{ass2-force}--\eqref{ass3-force} 
for 
$g=P(v\cdot\nabla v),\, P(b\cdot\nabla v),\,P(v\cdot\nabla b),\, PF_0$ and $P\partial_tb$
with 
$\kappa=\frac{3}{2q}+\frac{1}{2},\, \frac{1}{2},\, \frac{3}{2q},\, 0$ and $\gamma$, respectively, where $s=0$.
In fact, from \eqref{FK-wei}, 
\eqref{ns-est0}, \eqref{sol-hoe} with $(j,r)=(0,\infty)$ and $(1,q)$ we find
\begin{equation*}
\|\rho^2P(v\cdot\nabla v)(t)\|_q\leq C\|\rho v(t)\|_\infty\|\rho\nabla v(t)\|_q\leq CK^2\,t^{-3/2q-1/2}
\end{equation*}
for $t\in (0,{\mathcal T}_0]$ and
\begin{equation*}
\|P(v\cdot\nabla v)(t)-P(v\cdot\nabla v)(\tau)\|_q
\leq CK^2(t-\tau)^\mu \tau^{-3/2q-1/2-\mu}
\end{equation*}
for 
$\frac{t}{2}<\tau<t\leq {\mathcal T}_0$
with 
$\mu$ satisfying
$\max\{\frac{3}{2q},\, \gamma-\frac{1}{2}\}<\mu<\frac{1}{2}$
by taking into account Remark \ref{rem-hoe-multi}.
Estimates of the other forces above are deduced in an analogous way
with the aid of \eqref{lift-est} and \eqref{force-est} with $\alpha=2$ as well.
We now use Theorem \ref{duha-strong} to obtain the corresponding regularity of $\Phi v$ to \eqref{ns-cl} along with
\[
\|\partial_t(\Phi v)(t)\|_q+\|\nabla^2(\Phi v)(t)\|_q
\leq CK^2\,t^{-3/2q-1/2}+Cm_0K\big(t^{-1/2}+t^{-3/2q}\big)
\]
for all $t\in (0,{\mathcal T}_0]$, yielding
\[
\sup_{0<t\leq {\mathcal T}_0}\big(t\|\partial_t(\Phi v)(t)\|_q+t\|\nabla^2(\Phi v)(t)\|_q\big)
\leq C{\mathcal T}_0^{1/2-3/2q}K^2+Cm_0\big({\mathcal T}_0^{1/2}+{\mathcal T}_0^{1-3/2q}\big)K
\leq CK
\]
on account of \eqref{short-t}.
The other estimate
\[
\sup_{0<t\leq {\mathcal T}_0}\big(t\|\partial_t\overline v(t)\|_q+t\|\nabla^2\overline v(t)\|_q\big)
\leq C(\|\rho v_0\|_q+m_1)
\]
is also observed readily, where $m_1$ is given by \eqref{quan2}.
The proof is complete.
\hfill
$\Box$

\bigskip
\noindent
{\it Acknowledgements}.
The first author is partially supported by the THERS Make New Standards Program for the Next Generation Researchers
through JPMJSP2125 from JST.
The second author is partially supported by the Grant-in-aid for Scientific Research 22K03372 from JSPS.

%\bigskip
%\noindent
%{\it Declarations}.
%The authors state that there is no conflict of interest.

%%

\end{document}